\documentclass[a4paper]{article}

\usepackage{a4wide}
\usepackage{amsmath, amssymb, amsthm}
\usepackage{tikz}
\usetikzlibrary{patterns}
\usepackage[shortlabels]{enumitem}
\def\qed{\hfill $\Box$}

\usepackage{graphicx}
\usepackage{here}

\makeatletter

\def\@maketitle{
\begin{center}
{\huge\@title\par}
\vspace{5mm}
{\Large \@author \par}
\vspace{5mm}
\end{center}
\par\vskip 1.5em
}

\makeatother

\theoremstyle{plain}
\newtheorem*{thm}{Theorem}
\newtheorem{Th}{Theorem}[section]
\newtheorem{Prop}[Th]{Proposition}
\newtheorem{Cor}[Th]{Corollary}
\newtheorem{Lem}[Th]{Lemma}

\theoremstyle{definition}
\newtheorem{Def}[Th]{Definition}
\newtheorem{notation}[Th]{Notation}

\theoremstyle{remark}
\newtheorem{remark}[Th]{Remark}

\newcommand{\conv}{\mathop{\mathrm{Conv}}\nolimits}
\newcommand{\Newt}{\mathop{\mathrm{Newt}}\nolimits}
\newcommand{\Cay}{\mathop{\mathrm{Cay}}\nolimits}
\newcommand{\Subdiv}{\mathop{\mathrm{Subdiv}}\nolimits}
\newcommand{\Int}{\mathop{\mathrm{int}}\nolimits}
\newcommand{\MM}{\mathop{\mathrm{M}}\nolimits}

\renewcommand{\contentsname}{Contents}
\renewcommand{\refname}{References}

\renewcommand{\thefootnote}{\fnsymbol{footnote}}

\title{Smooth tropical complete intersection curves\\of genus $3$ in $\mathbb{R}^3$}
\author{Masayuki Sukenaga\footnote{
Department of Mathematics, Hiroshima University, 1-3-1 Kagamiyama, Higashihiroshima, 739-8526 JAPAN\\
\hspace{5.4mm}E-mail: d215394@hiroshima-u.ac.jp}}

\begin{document}

\setcounter{page}{1}

\maketitle

\begin{abstract}
We develop a method for describing the tropical complete intersection of a tropical hypersurface and a tropical plane in $\mathbb{R}^3$.
This involves a method for determining the topological type of the intersection of a tropical plane curve and $\mathbb{R}_{\leq 0}^2$ by using a polyhedral complex.
As an application, we study smooth tropical complete intersection curves of genus $3$ in $\mathbb{R}^3$.
In particular, we show that there are no smooth tropical complete intersection curves in $\mathbb{R}^3$ whose skeletons are the lollipop graph of genus $3$.
This gives a partial answer to a problem of Morrison in \cite{Morrison}.
\end{abstract}

{\flushleft{\hspace{5.4mm}\small{{\bf Keywords:} Tropical curve; Skeleton; Polyhedral complex}}}

\section{Introduction}

On the set $\mathbb{R} \cup \{ -\infty \}$, we define the sum $\oplus$ of two numbers as their maximum and the product $\odot$ as their usual sum.
Then ($\mathbb{R} \cup \{ -\infty \}, \oplus, \odot$) satisfies the axiom of semifields and is called the tropical algebra.
We can define tropical polynomials over $\mathbb{R} \cup \{ -\infty \}$ as polynomials with respect to these operations.
For a given tropical polynomial $f$, the tropical hypersurface $\mathbf{V}(f)$ is defined as the set of points where $f$ is not linear.
A tropical hypersurface in $\mathbb{R}^2$ is called a tropical curve.

More generally, a tropical curve is a vertex-weighted metric graph satisfying certain conditions.
For a given tropical curve $C$, we can define the skeleton of $C$ by contracting its rays (infinite edges) and this gives a coarse topological classification.
The skeleton of a tropical smooth plane curve defined by a tropical polynomial $f$ is a trivalent connected graph.
It is a point if $d:=\deg(f)=1$ or $2$, homeomorphic to a circle if $d=3$, and of genus $3$ if $d=4$.

There are exactly five trivalent connected graphs of genus (equal to the first Betti number of the graph as a topological space) $3$, as depicted in Figure \ref{gokonogurafu}.
The first four graphs are skeletons of some smooth tropical plane curves, but the fifth graph, called the lollipop graph of genus $3$, is not \cite{CDMY}.
\begin{figure}[H]
\centering
\begin{tikzpicture}
\draw (-0.8,0.5)--(0.8,0.5)--(0,-0.8)--cycle;
\draw (-0.8,0.5)--(0,0)--(0.8,0.5);
\draw (0,0)--(0,-0.8);

\draw (1.3,-0.5) to [out=135, in=225] (1.3,0.5);
\draw (1.3,-0.5) to [out=45, in=315] (1.3,0.5);
\draw (2.3,-0.5) to [out=135, in=225] (2.3,0.5);
\draw (2.3,-0.5) to [out=45, in=315] (2.3,0.5);
\draw (1.3,-0.5)--(2.3,-0.5);
\draw (1.3,0.5)--(2.3,0.5);

\draw (3,-0.5) to [out=135, in=225] (3,0.5);
\draw (3,-0.5) to [out=45, in=315] (3,0.5);
\draw (3,-0.5)--(3.6,0)--(3,0.5);
\draw (3.6,0)--(3.9,0);
\draw (4.2,0)circle(0.3);

\draw (5.4,0)--(5.7,0);
\draw (6.3,0)--(6.6,0);
\draw (5.1,0)circle(0.3);
\draw (6,0)circle(0.3);
\draw (6.9,0)circle(0.3);

\draw (8.2,0.07)--(8.2,-0.2)--(8.4,-0.4);
\draw (8.2,-0.2)--(8,-0.4);
\draw (8.2,0.37)circle(0.3);
\draw (8.615,-0.615)circle(0.3);
\draw (7.785,-0.615)circle(0.3);
\end{tikzpicture}
\caption{The five trivalent connected graphs of genus $3$.}
\label{gokonogurafu}
\end{figure}
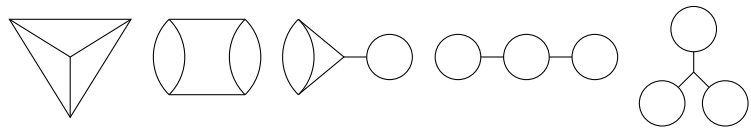
See \cite{JT} for other examples of tropical curves that cannot occur as tropical plane curves.

On the other hand, the lollipop graph of genus $3$ can be realized on a tropical plane in $\mathbb{R}^5$ \cite{HMRT}.
It is an open question whether it can be realized on a tropical plane in $\mathbb{R}^3$ or $\mathbb{R}^4$ \cite{Morrison}.
Let us consider the tropical plane $\mathbf{V}(f)$ in $\mathbb{R}^3$, where $f$ is the tropical polynomial $f=x\oplus y\oplus z\oplus 0$.
Explicitly, if we define
\begin{eqnarray*}
XY:=\{ (x, y, z)\in )\mathbb{R}^3\ |\ x\leq 0, y\leq 0, z=0\},\text{ and so on,}
\end{eqnarray*}
then $\mathbf{V}(f)$ is expressed as $\mathbf{V}(f)=XY\cup YZ\cup ZX\cup XW\cup YW\cup ZW$ (see Figure \ref{the_tropical_plane}).
In this paper, we develop a method for studying topological types of intersections of $\mathbf{V}(f)$ and a tropical hypersurface.
In particular, we ask whether the lollipop graph can be realized as the skeleton of a tropical curve on the tropical plane $\mathbf{V}(f)$.
More specifically, we consider the case where $C=\mathbf{V}(f)\cap \mathbf{V}(g)$ is a smooth tropical complete intersection curve \cite{MS} of the tropical plane $\mathbf{V}(f)$ and a tropical hypersurface $\mathbf{V}(g)$.
Our main theorem is as follows:

\begin{thm} (Theorem \ref{mainth})
Let $f$ and $g$ be tropical polynomials of degrees $d$ and $e$ in three variables.
We assume that
\begin{itemize}
\item $\Newt(f)=\conv(\{ (0, 0, 0), (d, 0, 0), (0, d, 0), (0, 0, d)\})$,
\item $\Newt(g)=\conv(\{ (0, 0, 0), (e, 0, 0), (0, e, 0), (0, 0, e)\})$,
\item $C=\mathbf{V}(f)\cap \mathbf{V}(g)$ is a smooth complete intersection curve.
\end{itemize}
Then the skeleton of $C$ is not the lollipop graph of genus $3$.
\end{thm}

We study complete intersections $\mathbf{V}(f)\cap \mathbf{V}(g)$ in the following way.
First, as we explained above, $\mathbf{V}(f)$ is a union of quarter planes.
Each quarter plane can be identified with $\mathbb{R}_{\leq 0}^2$.
So we study the intersection of a tropical plane curve and $\mathbb{R}_{\leq 0}^2$.
The tropical hypersurface defined by a tropical polynomial $p$ is dual to a polyhedral complex whose support is the Newton polytope of $p$.
We show that if a tropical plane curve $C'$ satisfies certain conditions, the restriction of $C'$ to $\mathbb{R}^2_{\leq 0}$ is dual to a connected subcomplex of the dual polyhedral complex of $C'$.
We apply this method to give a topological description of smooth complete intersection curves of genus $3$ on the tropical plane $\mathbf{V}(f)=\mathbf{V}(x\oplus y\oplus z\oplus 0)$.
We will mainly deal with skeletons of tropical curves in this paper, but we can also describe where the cycles and rays of $C=\mathbf{V}(f)\cap \mathbf{V}(g)$ are, etc.

This paper is organized as follows.
Section $2$ gives definitions concerning tropical varieties and the duality theorem for tropical hypersurfaces.
In Section $3$, we define the skeleton of a tropical curve and show several lemmas used in the proof of the main theorem.
In Section $4$, we describe the restriction of a tropical plane curve to $\mathbb{R}^2_{\leq 0}$ and show that the intersection of $\mathbf{V}(g)$ and $XY$ is dual to a subset surrounded by certain paths in a lattice subdivision of $\Delta_4=\conv(\{(0, 0), (4, 0), (0, 4)\})$.
In the last Section $5$, we prove the main theorem by dividing into cases according to the monomial of $g$ that corresponds to the domain of $\mathbb{R}^3\setminus \mathbf{V}(g)$ containing the origin $(0, 0, 0)$.

Directions for future research are as follows.
Beyond resolving the genus $3$ question, this paper presents framework that could be adapted for other open problems.
For instance, which genus $4$ skeletons occur in $\mathbb{R}^3$, say from the intersections of quadric surfaces with cubic surfaces?
And what skeletons can appear on tropical planes in higher dimensions (e.g. can the lollipop graph appear on a tropical plane in $\mathbb{R}^4$)?

\section*{Acknowledgement}
I would like to express my gratitude to my advisor Nobuyoshi Takahashi for giving a lot of advice.

\section{Tropical algebra and tropical curves}
First we give basic definitions and facts on tropical geometry. For details, see \cite{MS} and \cite{Morrison}. 
\begin{Def}
(Tropical algebra and tropical hypersurfaces). 
We define the \textit{tropical algebra} ($\mathbb{R} \cup \{ -\infty \}$, $\oplus$, $\odot$) by
\begin{eqnarray*}
x\oplus y&:=&\max(x, y),\\
x\odot y&:=&x+y.
\end{eqnarray*}
The tropical product $x\odot y$ will be also written as $xy$.
We use the notation $\{\!\{\,\}\!\}$ for multisets.
For a given tropical polynomial $f=\bigoplus_{i=1}^{k} \alpha_i x_1^{a_{i1}}\cdots x_n^{a_{in}}\neq -\infty$, where a monomial which does not appear is regarded as having the coefficient $-\infty$, we define the \textit{tropical hypersurface} $\mathbf{V}(f)$ by
\begin{eqnarray*} 
\mathbf{V}(f)=\left\{ (t_1, \dots, t_n)\in \mathbb{R}^n \middle| 
\begin{array}{l}
\text{the maximum in}\\
\mbox{$\{\!\{ \alpha_1+a_{11}t_1+\dots +a_{1n}t_n, \dots, \alpha_k+a_{k1}t_1+\dots +a_{kn}t_n\}\!\}$}\\
\text{is achieved at least twice}
\end{array}
\right\}.
\end{eqnarray*}
We define $\mathbf{V}(-\infty):=\mathbb{R}^n$.
When $n=2$, we call $\mathbf{V}(f)$ a tropical plane curve.
The \textit{Newton polytope} Newt$(f)$ of $f$ is defined to be the convex hull of $\{ (a_{i1}, \ldots, a_{in}) \ | \ \alpha_i\neq -\infty \}$.
In the case of $n=2$, we call the Newton polytope the ``\textit{Newton polygon}.''
\end{Def}

We will see in a moment that we can regard $\mathbf{V}(f)$ as a polyhedral complex, whose cells are the loci where a fixed set of terms are maximal and there is a lattice subdivision of Newt$(f)$ that is dual to $\mathbf{V}(f)$.

\begin{Def}
(Dual subdivision).
For $f=\bigoplus_{i=1}^{k} \alpha_i x_1^{a_{i1}}\cdots x_n^{a_{in}}$, we define a convex polyhedral set $A_f$ as follows:
\[
A_f=\conv(\{ (a_{i1}, \dots, a_{in}, \alpha)\in \mathbb{Z}^n\times \mathbb{R}\ \mid \alpha \leq \alpha_i\}).
\]
Bounded faces of $A_f$ are called \textit{upper faces} of $A_f$.
The projections of upper faces of $A_f$ give a lattice subdivision of $\Newt(f)$.
Thus $f$ determines a polyhedral complex supported on $\Newt(f)$, which we denote by $\Subdiv_{f}$.
\end{Def}

\begin{Th}\label{dualitytheorem}
(The Duality Theorem, \cite[Proposition 3.11]{Mik}). 
Let $f$ be a tropical polynomial in $n$ variables. 
Then the tropical hypersurface $\mathbf{V}(f)$ is the support of a polyhedral complex $\Sigma_f$ of dimension $n-1$ in $\mathbb{R}^n$.
It is dual to the subdivision $\Subdiv_f$ in the following sense:
\begin{itemize}
  \item (Closures of) domains of $\mathbb{R}^n\setminus \mathbf{V}(f)$ correspond to lattice points in $\Subdiv_f$.
  \item $d$-dimensional cells in $\Sigma_f$ $(0\leq d\leq n-1)$ correspond to $(n-d)$-dimensional cells in $\Subdiv_f$.
  \item These correspondences are inclusion-reversing. 
\end{itemize}
\end{Th}

Next, we will explain the balancing condition (See \cite{MS} Section $3. 3$ for details).
Let $f$ be a tropical polynomial.
Then, by Theorem \ref{dualitytheorem}, the subdivision $\Subdiv_f$ is dual to the polyhedral complex $\Sigma_f$.
Let $\tau$ be a $(n-2)$-dimensional cell of $\Sigma_f$ and $\sigma$ a $(n-1)$-dimensional cell of $\Sigma_f$ containing $\tau$ as its face.
We know that $\tau$ corresponds to a $2$-dimensional cell $\Delta_{\tau}$ in $\Subdiv_f$ and $\sigma$ corresponds to a facet $e(\sigma)$ of $\Delta_{\tau}$.
We define $m(\sigma)$ to be the lattice length of $e(\sigma)$.
Here, a lattice vector in $\mathbb{Z}^n$ is primitive if its coefficients are relatively prime.
Each rational vector can be rescaled to integer coordinates, and dividing by the greatest common divisor of the coefficients gives a primitive lattice vector.
For the facet $e(\sigma)$ of $\Delta_{\tau}$, there is a unique outward pointing normal vector $\mathbf{v}_{\sigma}$ which is primitive.

\begin{Def}
Let $f$ be a tropical polynomial.
The tropical hypersurface $\mathbf{V}(f)$ is \textit{balanced} if for each $(n-2)$-dimensional cell $\tau$ in $\Sigma_f$, we have
\[
\sum m(\sigma)\mathbf{v}_{\sigma}=0,
\]
where the sum is taken over every $(n-1)$-dimensional cell $\sigma$ containing $\tau$ as its facet.
\end{Def}

\begin{Prop}\label{balancing_condition}
\cite[Proposition 3.3.2]{MS}.
The $(n-1)$-dimensional polyhedral complex $\mathbf{V}(f)$ goven by a tropical polynomial $f$ in $n$ unknowns is balanced for the weights $m(\sigma)$ defined above.
\end{Prop}

\begin{Def}
A tropical plane curve is \textit{smooth} if its dual subdivision is a unimodular triangulation, meaning that every polygon in the subdivision is a triangle with no lattice points besides its vertices.
\end{Def}

Next, we will define smooth complete intersection curves in $\mathbb{R}^3$.
See \cite{MS} Section $4. 6$ for details on the following contents.

\begin{Def}
(Cayley polytope). 
Let $p=\bigoplus_{i=1}^k \alpha_i x^{a_i} y^{b_i} z^{c_i}$, $q=\bigoplus_{j=1}^l \beta_j x^{a'_j} y^{b'_j} z^{c'_j}$, $P=\Newt(p)$, and $Q=\Newt(q)$. 
We define the \textit{Cayley polytope} of $P$ and $Q$ as follows:
\[
\Cay(P, Q)=\conv((P\times \{ 0\}) \cup (Q\times \{ 1\}))\subset \mathbb{R}^4.
\]
The tropical polynomials $p$ and $q$ also determine a convex polyhedral set
\[
\conv(\{(a_i, b_i, c_i, 0, \alpha) \mid \alpha \leq \alpha_i \} \cup \{(a'_j, b'_j, c'_j, 1, \beta) \mid \beta \leq \beta_j \}) \subset \mathbb{R}^5.
\]
The projections of upper faces of this polyhedral set give a subdivision of $\Cay(P, Q)$, and we call it the \textit{Cayley subdivision} associated to $p$ and $q$.
A \textit{mixed cell} of the Cayley subdivision is a cell $A$ with $\dim(A\cap \{(x, y, z, i, u) \mid x, y, z, u\in \mathbb{R} \})\geq 1$ for $i=0, 1$.
\end{Def}

\begin{Def}
(Smooth tropical complete intersection curve). 
Let $p$ and $q$ be tropical polynomials in three variables and  let $C=\mathbf{V}(p)\cap \mathbf{V}(q)$. 
Write $P=\Newt(p)$ and $Q=\Newt(q)$. 
If all $4$-dimensional cells in the Cayley subdivision of $\Cay(P, Q)$ associated to $p$ and $q$ have the minimum possible volume $1/24$, we call $C$ a \textit{smooth complete intersection curve}.
If all $4$-dimensional mixed cells in the Cayley subdivision have the volume $1/24$, we call $C$ a \textit{weakly smooth complete intersection curve}.
\end{Def}

\begin{Th}\label{kanzenkousa}
(\cite[Theorem 8]{Morrison}). 
Let $f$ and $g$ be  tropical polynomials of degrees $d$ and $e$ in three variables.
We assume that
\begin{eqnarray*}
\Newt(f)=\conv(\{ (0, 0, 0), (d, 0, 0), (0, d, 0), (0, 0, d)\}),\\
\Newt(g)=\conv(\{ (0, 0, 0), (e, 0, 0), (0, e, 0), (0, 0, e)\}),\hspace{1mm}
\end{eqnarray*}
and the tropical curve $C=\mathbf{V}(f)\cap \mathbf{V}(g)$ is a smooth complete intersection.
Then $C$ is a trivalent graph and has
\begin{itemize}
  \item $d^2e+de^2$ vertices,
  \item $(3/2)d^2e+(3/2)de^2-2de$ edges (bounded one-dimensional cells),
  \item $4de$ rays (unbounded one-dimensional cells).
\end{itemize}
The genus of the graph $C$ equals $(1/2)d^2e+(1/2)de^2-2de+1$. 
\end{Th}

\section{Skeletons of tropical curves and the lollipop graph}
First, we define the skeleton of a tropical curve according to \cite{Morrison}.

\begin{Def}
(Skeletons of tropical curves).
Let $C$ be a tropical curve.
We delete all rays from $C$, and then successively remove any vertices incident to exactly one edge, along with such edges.
We remove the vertices incident to two edges and fuse the two edges into one.
The resulting collection of edges and vertices is called the \textit{skeleton} of the tropical curve $C$.
We define the \textit{genus} of the graph as $E-V+1$ for a graph with $E$ edges and $V$ vertices.
\end{Def}

\begin{Def}
(Troplanar graph).
A graph that is the skeleton of some smooth tropical plane curve is called tropically planer, or \textit{troplanar} for short.
\end{Def}

Next, we introduce a tropical plane in $\mathbb{R}^3$.

\begin{Def}
(Tropical plane in $\mathbb{R}^3$).
We define a \textit{tropical plane} in $\mathbb{R}^3$ as a tropical hypersurface defined by a tropical polynomial of degree $1$ in three variables.
We call the specific one defined by $x\oplus y\oplus z\oplus 0$ ``the tropical plane'' (see Figure \ref{the_tropical_plane}).
\end{Def}

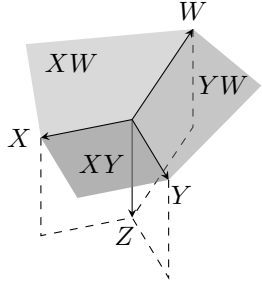
\begin{figure}[H]
\centering
\begin{tikzpicture}
\coordinate (O) at (0,0);
\coordinate (X) at (-1.2,-0.24);
\coordinate (Y) at (0.48,-0.8);
\coordinate (Z) at (0,-1.3);
\coordinate (W) at (0.8,1.2);
\coordinate (XW) at (-1.4,1);
\coordinate (YW) at (1.7,0.45);
\coordinate (XY) at (-0.72,-1.04);
\coordinate (XZ) at (0,-1.414);
\coordinate (YZ) at (1.2,0.7);
\draw[dashed] (W)--(0.8,-0.1)--(Z);
\draw[dashed] (Y)--(0.48,-2.1)--(Z);
\draw[dashed] (X)--(-1.2,-1.54)--(Z);
\draw[->,>=stealth] (O)--(Z);
\coordinate [label=below:$Z$] (a) at (-0.1,-1.3);
\fill[gray, opacity=0.3] (X)--(XW)--(W)--(O)--cycle;
\fill[gray, opacity=0.6] (O)--(X)--(XY)--(Y)--cycle;
\fill[gray, opacity=0.45] (O)--(Y)--(YW)--(W)--cycle;
\coordinate [label=below:$XY$] (A) at (-0.4,-0.35);
\coordinate [label=below:$XW$] (B) at (-0.8,1);
\coordinate [label=below:$YW$] (C) at (1.2,0.7);
\draw[->,>=stealth] (O)--(X) node[left] {$X$};
\draw[->,>=stealth] (O)--(Y);
\coordinate [label=below:$Y$] (b) at (0.65,-0.75);
\draw[->,>=stealth] (O)--(W) node[above] {$W$};
\end{tikzpicture}
\caption{The tropical plane in $\mathbb{R}^3$.}
\label{the_tropical_plane}
\end{figure}

\begin{notation}
Let
\[
\mathbf{x}:=(-1, 0, 0),\ \mathbf{y}:=(0, -1, 0),\ \mathbf{z}:=(0, 0, -1),\ \mathbf{w}:=(1, 1, 1).
\]
We write the four rays that are the $1$-dimensional cells of the tropical plane in $\mathbb{R}^3$ as
\[
X:=\mathbb{R}_{\geq 0}\mathbf{x},\ Y:=\mathbb{R}_{\geq 0}\mathbf{y},\ Z:=\mathbb{R}_{\geq 0}\mathbf{z},\ W:=\mathbb{R}_{\geq 0}\mathbf{w},
\]
and write six copies of quarter planes as follows:
\begin{eqnarray*}
XY:=\mathbb{R}_{\geq 0}\mathbf{x}+\mathbb{R}_{\geq 0}\mathbf{y},\hspace{4mm} YZ:=\mathbb{R}_{\geq 0}\mathbf{y}+\mathbb{R}_{\geq 0}\mathbf{z},\hspace{4mm} ZX:=\mathbb{R}_{\geq 0}\mathbf{z}+\mathbb{R}_{\geq 0}\mathbf{x},\\
XW:=\mathbb{R}_{\geq 0}\mathbf{x}+\mathbb{R}_{\geq 0}\mathbf{w},\hspace{2mm} YW:=\mathbb{R}_{\geq 0}\mathbf{y}+\mathbb{R}_{\geq 0}\mathbf{w},\hspace{2mm} ZW:=\mathbb{R}_{\geq 0}\mathbf{z}+\mathbb{R}_{\geq 0}\mathbf{w}.
\end{eqnarray*}
\end{notation}
Then, we have $\mathbf{V}(x\oplus y\oplus z\oplus 0)=XY\cup YZ\cup ZX\cup XW\cup YW\cup ZW$.
\medbreak
We ask which graphs can be realized as the skeleton of a tropical curve.
There are exactly five trivalent connected graphs of genus $3$ (see Figure \ref{gokonogurafu}).
The first four graphs are troplanar, in particular they are skeletons of some smooth tropical plane quartic curves, but the fifth graph, called the lollipop graph of genus $3$, is not \cite{CDMY}.

On the other hand, the lollipop graph can be realized on a tropical plane in $\mathbb{R}^5$ \cite{HMRT}.
However, it is an open question if the lollipop graph can be realized on a tropical plane in $\mathbb{R}^3$ or $\mathbb{R}^4$.

Starting in the next section, we will consider smooth tropical complete intersection curves of tropical hypersurfaces and the tropical plane $\mathbf{V}(f)$ in $\mathbb{R}^3$, and show that the lollipop graph cannot be realized under certain conditions.

\begin{Lem}\label{d=4}
Let $f$ and $g$ be tropical polynomials of degrees $d$ and $e$ ($0<d\leq e$) in three variables.
We assume that
\begin{itemize}
\item $\Newt(f)=\conv(\{ (0, 0, 0), (d, 0, 0), (0, d, 0), (0, 0, d)\})$,
\item $\Newt(g)=\conv(\{ (0, 0, 0), (e, 0, 0), (0, e, 0), (0, 0, e)\})$,
\item $C=\mathbf{V}(f)\cap \mathbf{V}(g)$ is a smooth complete intersection curve of genus $3$.
\end{itemize}
Then $(d, e)=(1, 4)$ and $C$ is a trivalent graph with $20$ vertices, $22$ edges (bounded one-dimensional cells), and $16$ rays (unbounded one-dimensional cells).
\end{Lem}

\begin{proof}
Assume that $C$ has genus $3$, then by Theorem \ref{kanzenkousa}, we have 
\[
(1/2)d^2e+(1/2)de^2-2de+1=3.
\]
Thus we have  $(d, e)=(1, 4)$.
Again by Theorem \ref{kanzenkousa}, $C$ is a trivalent graph and has $20$ vertices, $22$ edges, and $16$ rays.
\end{proof}

\section{The restrictions of tropical plane curves to $\mathbb{R}^2_{\leq 0}$}
In this section, we assume that 
\begin{eqnarray*}
f&=&x\oplus y\oplus z\oplus 0,\\
g&=&\bigoplus_{a+b+c\leq m}\alpha _{abc}x^a y^b z^c,\\
\Newt(g)&=&\conv(\{ (0, 0, 0), (m, 0, 0), (0, m, 0), (0, 0, m)\}), 
\end{eqnarray*}
and that $C=\mathbf{V}(f)\cap \mathbf{V}(g)$ is a smooth tropical complete intersection curve.

First, we introduce a notation for the restriction of a tropical plane curve to $\mathbb{R}^2_{\leq 0}$.

\begin{notation}
For a tropical polynomial $p$ in $n$ variables, we define
\[
\mathbf{V}_{\!-}(p)=\mathbf{V}(p)\cap \mathbb{R}^n_{\leq 0}.
\] 
\end{notation}

We consider the case where $C=\mathbf{V}(f)\cap \mathbf{V}(g)\subset \mathbb{R}^3$ is a smooth complete intersection curve for $f=x\oplus y\oplus z\oplus 0$ and a tropical polynomial $g$ of degree $m$.

\begin{Lem}\label{genten}
Let $f=x\oplus y\oplus z\oplus 0$ and $g$ be a tropical polynomial.
We assume that $C=\mathbf{V}(f)\cap \mathbf{V}(g)$ is a weakly smooth tropical complete intersection curve.
Then, we have
\[
(0, 0, 0)\notin \mathbf{V}(g).
\]
\end{Lem}

\begin{proof}
Let $P=\Newt(f)$ and $Q=\Newt(g)$.
If $(0, 0, 0)\in \mathbf{V}(g)$, there exist terms $\alpha_{a_1b_1c_1}x^{a_1}y^{b_1}z^{c_1}$ and $\alpha_{a_2b_2c_2}x^{a_2}y^{b_2}z^{c_2}$ of $g$ such that for all terms $\alpha_{abc}x^ay^bz^c$ of $g$
\[
\alpha_{a_1b_1c_1}=\alpha_{a_2b_2c_2}\geq \alpha_{abc}.
\]
We define
\begin{eqnarray*}
T_f&=&\left\{ \begin{pmatrix} a \\ b \\ c \\ 0 \\ 0 \end{pmatrix} \middle| x^ay^bz^c\text{ is a term of $f$}\right\}=\left\{ \begin{pmatrix} 0 \\ 0 \\ 0 \\ 0 \\ 0 \end{pmatrix}, \begin{pmatrix} 1 \\ 0 \\ 0 \\ 0 \\ 0 \end{pmatrix}, \begin{pmatrix} 0 \\ 1 \\ 0 \\ 0 \\ 0 \end{pmatrix}, \begin{pmatrix} 0 \\ 0 \\ 1 \\ 0 \\ 0 \end{pmatrix} \right\},\\
T_g&=&\left\{ \begin{pmatrix} a \\ b \\ c \\ 1 \\ \alpha_{abc} \end{pmatrix} \middle| \alpha_{abc}x^ay^bz^c\text{ is a term of $g$}\right\},\\
\mathbf{p}&=&(0, 0, 0, \alpha_{a_1b_1c_1}, -1),\quad \mathbf{w}_1=\begin{pmatrix} a_1 \\ b_1 \\ c_1 \\ 1 \\ \alpha_{a_1b_1c_1} \end{pmatrix},\quad \mathbf{w}_2=\begin{pmatrix} a_2 \\ b_2 \\ c_2 \\ 1 \\ \alpha_{a_2b_2c_2}\end{pmatrix}.
\end{eqnarray*}
Then, we have $\mathbf{w}_1, \mathbf{w}_2\in T_g$ and
\[
\forall \mathbf{v}\in T_f,\ \forall \mathbf{w}\in T_g,\ 0=\mathbf{p}\cdot \mathbf{w}_1=\mathbf{p}\cdot \mathbf{w}_2=\mathbf{p}\cdot \mathbf{v}\leq \mathbf{p}\cdot \mathbf{w}.
\]
Hence there exists a four-dimensional cell in the Cayley subdivision of $\Cay(P, Q)$ containing $\conv \left\{ \begin{pmatrix} 0 \\ 0 \\ 0 \\ 0 \end{pmatrix}, \begin{pmatrix} 1 \\ 0 \\ 0 \\ 0 \end{pmatrix}, \begin{pmatrix} 0 \\ 1 \\ 0 \\ 0 \end{pmatrix}, \begin{pmatrix} 0 \\ 0 \\ 1 \\ 0 \end{pmatrix}, \begin{pmatrix} a_1 \\ b_1 \\ c_1 \\ 1 \end{pmatrix}, \begin{pmatrix} a_2 \\ b_2 \\ c_2 \\ 1  \end{pmatrix}
\right\}$.
Since this cell has more than $5$ vertices, its volume is strictly greater than $1/24$.
This contradicts to the assumption that $C$ is a weakly smooth complete intersection curve.
\end{proof}

\begin{Lem}\label{koutenyonko}
Let $f=x\oplus y\oplus z\oplus 0$ and $g$ be a tropical polynomial with $\Newt(g)=\conv(\{ (0, 0, 0)$, $(m, 0, 0), (0, m, 0), (0, 0, m)\})$.
We assume that $C=\mathbf{V}(f)\cap \mathbf{V}(g)$ is a weakly smooth complete intersection curve.
Then, we have
\[
\lvert \mathbf{V}(g)\cap (X\cup Y\cup Z\cup W)\rvert=m.
\]
\end{Lem}

\begin{proof}
If $\mathbf{V}(g)\cap X$ contains a line segment $L$, $C$ cannot branch at a point in $L$ other than the origin by $C\subset \mathbf{V}(f)$, the balancing condition and trivalence.
Therefore, if $\mathbf{V}(g)\cap X$ contains a line segment, it contains the origin and this contradicts Lemma \ref{genten}.
Thus, $\mathbf{V}(g)\cap X$ does not contain a line segment.
The same is true for the intersections of $\mathbf{V}(g)$ and $Y$, $Z$, $W$, and therefore we have $\lvert \mathbf{V}(g)\cap (X\cup Y\cup Z\cup W)\rvert<\infty$.

Let $P=\Newt(f)$ and $Q=\Newt(g)$.
We assume that $(t, 0, 0)\in \mathbf{V}(g)\cap X$.
Since the intersection points of $\mathbf{V}(g)$ and $X$ are isolated, there exist two terms $\alpha_{a_1b_1c_1}x^{a_1}y^{b_1}z^{c_1}$ and $\alpha_{a_2b_2c_2}x^{a_2}y^{b_2}z^{c_2}$ $(a_1\neq a_2)$ of $g$ such that for all terms $\alpha_{abc}x^ay^bz^c$ of $g$, we have
\[
\alpha_{a_1b_1c_1}+a_1t=\alpha_{a_2b_2c_2}+a_2t\geq \alpha_{abc}+at.
\]
Let
\begin{eqnarray*}
T_1&=&\left\{ \begin{pmatrix} 0 \\ 0 \\ 0 \\ 0 \\ 0 \end{pmatrix}, \begin{pmatrix} 0 \\ 1 \\ 0 \\ 0 \\ 0 \end{pmatrix}, \begin{pmatrix} 0 \\ 0 \\ 1 \\ 0 \\ 0 \end{pmatrix} \right\},\quad T_2=\left\{ \begin{pmatrix} a \\ b \\ c \\ 1 \\ \alpha_{abc} \end{pmatrix} \middle| \alpha_{abc}x^ay^bz^c\text{ is a term of $g$}\right\} \cup \left\{ \begin{pmatrix} 1 \\ 0 \\ 0 \\ 0 \\ 0 \end{pmatrix}\right\},\\
\mathbf{p}&=&(-t, 0, 0, \alpha_{a_1b_1c_1}+a_1t, -1),\quad \mathbf{w}_1=\begin{pmatrix} a_1 \\ b_1 \\ c_1 \\ 1 \\ \alpha_{a_1b_1c_1} \end{pmatrix},\quad \mathbf{w}_2=\begin{pmatrix} a_2 \\ b_2 \\ c_2 \\ 1 \\ \alpha_{a_2b_2c_2}
\end{pmatrix}.
\end{eqnarray*}
Then, we have $\mathbf{w}_1, \mathbf{w}_2\in T_2$ and
\[
\forall \mathbf{v}\in T_1,\ \forall \mathbf{w}\in T_2,\ 0=\mathbf{p}\cdot \mathbf{w}_1=\mathbf{p}\cdot \mathbf{w}_2=\mathbf{p}\cdot \mathbf{v}\leq \mathbf{p}\cdot \mathbf{w}.
\]
Hence there exists a four-dimensional cell in the Cayley subdivision of $\Cay(P, Q)$ containing $\conv \left\{ \begin{pmatrix} 0 \\ 0 \\ 0 \\ 0 \end{pmatrix}, \begin{pmatrix} 0 \\ 1 \\ 0 \\ 0 \end{pmatrix}, \begin{pmatrix} 0 \\ 0 \\ 1 \\ 0 \end{pmatrix}, \begin{pmatrix} a_1 \\ b_1 \\ c_1 \\ 1 \end{pmatrix}, \begin{pmatrix} a_2 \\ b_2 \\ c_2 \\ 1  \end{pmatrix}
\right\}$.
The volume of this convex polygon is $\lvert a_1-a_2 \rvert /24$ because $\left| \det \begin{pmatrix} 0 & 0 & a_1 & a_2 \\ 1 & 0 & b_1 & b_2 \\ 0 & 1 & c_1 & c_2 \\ 0 & 0 & 1 & 1 \end{pmatrix} \right|=\lvert a_1-a_2 \rvert$ $(\neq0)$.
Here, $C$ is a weakly smooth tropical complete intersection curve and $a_1\neq a_2$, and so we have $\lvert a_1-a_2 \rvert =1$.
Let $\alpha_{a_0b_0c_0}x^{a_0}y^{b_0}z^{c_0}$ be the term of $g$ which is the maximum term at the origin.
The exponent of $x$ of the term of $g$ which is the maximum term at $(-s, 0, 0)$ for $s$ sufficiently large is $0$.
Since the maximum terms of $g$ on $X$ change only on $\mathbf{V}(g)\cap X$ and the exponents of $x$ change by $1$, we have $\lvert \mathbf{V}(g)\cap X\rvert=a_0$.
Similar equalities also hold for the intersections of $\mathbf{V}(g)$ and $Y$, $Z$.
We perform appropriate coordinate transformation when dealing with $W$, and we have
\[
\lvert \mathbf{V}(g) \cap W\rvert =m-a_0-b_0-c_0.
\]
Consequently, we have
\[
\lvert \mathbf{V}(g)\cap (X\cup Y\cup Z\cup W)\rvert=a_0+b_0+c_0+(m-a_0-b_0-c_0)=m.
\]
\end{proof}

From the proof of Lemma \ref{koutenyonko}, we have the following.
\begin{Cor}\label{XYZW}
We have
\begin{eqnarray*}
\lvert \mathbf{V}(g)\cap X\rvert =a_0,\quad \lvert \mathbf{V}(g)\cap Y\rvert =b_0,\quad \lvert \mathbf{V}(g)\cap Z\rvert =c_0,\quad \lvert \mathbf{V}(g)\cap W\rvert =d_0.
\end{eqnarray*}
\end{Cor}

By Lemma \ref{genten} and \ref{koutenyonko}, we have
\begin{eqnarray*}
(0, 0, 0)\notin \mathbf{V}(g),\quad \lvert \mathbf{V}(g)\cap (X\cup Y\cup Z\cup W)\rvert=m.
\end{eqnarray*}
Let $x^{a_0}y^{b_0}z^{c_0}$ be the monomial of $g$ which corresponds to the domain of $\mathbb{R}^3 \setminus \mathbf{V}(g)$ containing $(0, 0, 0)$ and $d_0=m-a_0-b_0-c_0$.
We write the homogenization of $g$ as
\[
g^h=\bigoplus_{a+b+c+d=m}\alpha _{abc}x^a y^b z^c w^d.
\]
Note that there is a one to one correspondence between the terms of $g$ and the terms of $g^h$.
Here, the monomial of $g^h$ which correspondes to the domain of $\mathbb{R}^3 \setminus \mathbf{V}(g)$ containing $(0, 0, 0)$ is $x^{a_0}y^{b_0}z^{c_0}w^{d_0}$.

\begin{remark}
In the following, we expand the definition of trivalent graph.
For the tropical hypersurface $\mathbf{V}(g)$, there exist a $t<0$ such that $\mathbf{V}(g)\cap (X\cup Y)\subset [t, 0]\times [t, 0]\subset \mathbb{R}^2$.
Let $A$ be $\mathbb{R}_{\leq 0}^2$ or $\mathbb{R}_{< 0}^2$.
If $A=\mathbb{R}_{\leq 0}^2$, we set $B=[t, 0]\times [t, 0]$ and if $A=\mathbb{R}_{<0}^2$, we set $B=[t, -\epsilon]\times [t, -\epsilon]$, where $\epsilon >0$ is sufficiently small.
Then, we can naturally regard $\mathbf{V}(g)\cap B$ as a bounded graph by adding the vertices on the intersection of $\mathbf{V}(g)$ and the boundary of $B$.
When each vertex of $\mathbf{V}(g)\cap B$ has valence $1$ or $3$, we call $\mathbf{V}(g)\cap A$ a trivalent graph.
\end{remark}

When we think about the intersection of $\mathbf{V}(f)$ and $\mathbf{V}(g)$, we have to restrict $\mathbf{V}(g)$ to the six copies of quarter planes.
For simplicity of notation, we write $(z=0)$ for $\{ (x, y, 0) \in \mathbb{R}^3\mid x, y \in \mathbb{R}\}$.
The intersection of $\mathbf{V}(g)$ and $XY\setminus (X\cup Y)$ is a trivalent graph, but the intersection of $\mathbf{V}(g)$ and $(z=0)$ is not always one-dimensional.
The two-dimensional parts of the intersection come from the two-dimensional parts of $\mathbf{V}(g)$, i.e. cells of $\mathbf{V}(g)$ defined by two terms of $g$.
We write the tropical polynomial in two variables defined by substituting $z=0$ to $g$ as $g|_{z=0}$.
We want to compare the tropical plane curve $\mathbf{V}(g|_{z=0})$ in $\mathbb{R}^2$ and the intersection of $\mathbf{V}(g)$ and $(z=0)$.
Recall that we set $g= \bigoplus_{a+b+c \leq m}\alpha _{abc}x^a y^b z^c$.
Then we have
\[
g|_{z=0}=\bigoplus_{a+b\leq m}\alpha_{ab}x^a y^b,\ \alpha_{ab}=\max(\alpha_{ab0}, \ \alpha_{ab1}, \ldots, \ \alpha_{a, b, m-a-b}).
\]
Note also that $g|_{z=0}=g^h|_{z=w=0}$.

Recall that for a given tropical polynomial $h$ and a domain of $\mathbb{R}^n\setminus \mathbf{V}(h)$, there is a term of $h$ that attains the only maximum among the terms of $h$ on this domain.
Thus the domains of $\mathbb{R}^n\setminus \mathbf{V}(h)$ correspond to some of the terms of $h$.

\begin{Lem}\label{a_0+b_0}
The following holds.
\begin{itemize}
\item {$\lvert \mathbf{V}(g|_{z=0})\cap (X\cup Y)\rvert =a_0+b_0$.}
\item {The monomial of $g|_{z=0}$ which corresponds to the domain of $\mathbb{R}^2 \setminus \mathbf{V}(g|_{z=0})$ containing $(0, 0)$ is $x^{a_0}y^{b_0}$.}
\item {There are no vertices of $\mathbf{V}(g|_{z=0})$ on $X\cup Y$ and $\mathbf{V}_{\!-}(g|_{z=0})=\mathbf{V}(g|_{z=0})\cap \mathbb{R}^2_{\leq 0}$ is a trivalent graph.}
\end{itemize}
\end{Lem}

\begin{proof}
By Corollary \ref{XYZW} we have $\lvert \mathbf{V}(g|_{z=0})\cap (X\cup Y)\rvert =a_0+b_0$.
Because $x^{a_0}y^{b_0}z^{c_0}$ is the monomial of $g$ which corresponds to the domain of $\mathbb{R}^3 \setminus \mathbf{V}(g)$ containing $(0, 0, 0)$, $\alpha_{a_0b_0c_0}$ is the only maximum of the multiset $\{\!\{ \alpha \mid \text{$\alpha$ is the coefficient of a term of $g$}\}\!\}$.
Therefore, $\alpha_{a_0b_0}=\alpha_{a_0b_0c_0}$ is the only maximum of the multiset $\{\!\{ \alpha \mid \text{$\alpha$ is the coefficient of a term of $g|_{z=0}$}\}\!\}$ and it is the monomial of $g|_{z=0}$ which corresponds to the domain of $\mathbb{R}^2 \setminus \mathbf{V}(g|_{z=0})$ containing $(0, 0)$.
We only have to show that the vertices of $\mathbf{V}(g|_{z=0})$ are not on $X\cup Y$ because $\mathbf{V}(g|_{z=0})\cap \mathbb{R}^2_{<0}=\mathbf{V}(g)\cap (XY\setminus (X\cup Y))$ is a trivalent graph by Theorem \ref{kanzenkousa}.
Assume that there is a vertex $(-t, 0)$ of $\mathbf{V}(g|_{z=0})$ (with valence $> 2$) on $X$.
Let $\alpha_{a_1b_1}x^{a_1}y^{b_1}$ be the term of $\mathbf{V}(g|_{z=0})$ which corresponds to the domain containing $(-t+\delta, 0)$ and $\alpha_{a_2b_2}x^{a_2}y^{b_2}$ the term of $\mathbf{V}(g|_{z=0})$ which corresponds to the domain containing $(-t-\delta, 0)$ for sufficiently small $\delta >0$.
Since the valence of the vertex $(-t, 0)$ is greater than $2$, there is a term $\alpha_{a_3b_3}x^{a_3}y^{b_3}$ such that $\alpha_{a_1b_1}-a_1t=\alpha_{a_2b_2}-a_2t=\alpha_{a_3b_3}-a_3t$.
Here, we have
\begin{eqnarray*}
&\alpha_{a_1b_1}-a_1(t-\delta)>\alpha_{a_3b_3}-a_3(t-\delta),& \\
&\alpha_{a_2b_2}-a_2(t+\delta)>\alpha_{a_3b_3}-a_3(t+\delta),& 
\end{eqnarray*}
and hence we have $a_2<a_3<a_1$, and by Corollary \ref{XYZW}, $a_1=a_2+1$.
This is a contradiction.
Thus there are no vertices of $\mathbf{V}(g|_{z=0})$ which have valence more than two on $X$.
The same holds for $Y$.
\end{proof}

Next, we set
\begin{eqnarray*}
A&=&\{ \alpha_{ab}x^ay^b \mid \text{$\alpha_{ab}$ is the only maximum of the multiset $\{\!\{ \alpha_{ab0}, \ \alpha_{ab1}, \ldots, \ \alpha_{a, b, m-a-b}\}\!\}$}\,\},\\
B&=&\left\{ \alpha_{ab}x^ay^b \middle| 
\begin{array}{l}
\text{There are two or more terms that equal to $\alpha_{ab}$}\\
\text{in the multiset $\{\!\{\alpha_{ab0}, \ \alpha_{ab1}, \ldots, \ \alpha_{a, b, m-a-b}\}\!\}$}
\end{array} \right\},\\
A'&=&\left\{ \alpha_{ab}x^ay^b\in A \ \middle| 
\begin{array}{l}
\text{There exists $(s, t)\in \mathbb{R}^2$ such that for all terms $\alpha_{cd}x^cy^d$ of $g|_{z=0}$}\\
\text{other than $\alpha_{ab}x^ay^b$, $\alpha_{ab}+as+bt>\alpha_{cd}+cs+dt$}
\end{array} \right\},\\
B'&=&\left\{ \alpha_{ab}x^ay^b\in B \ \middle| 
\begin{array}{l}
\text{There exists $(s, t)\in \mathbb{R}^2$ such that for all terms $\alpha_{cd}x^cy^d$ of $g|_{z=0}$}\\
\text{other than $\alpha_{a, b}x^ay^b$, $\alpha_{ab}+as+bt>\alpha_{cd}+cs+dt$}
\end{array} \right\}.
\end{eqnarray*}
Then, the set of terms of $g|_{z=0}$ that correspond to domains of $\mathbb{R}^2\setminus \mathbf{V}(g|_{z=0})$ equal to $A'\sqcup B'$.
We write $D$ for the union of the domains of $\mathbb{R}^2\setminus \mathbf{V}(g|_{z=0})$ that correspond to the terms of $g|_{z=0}$ which are contained in $B'$.
Note that $\mathbf{V}(g|_{z=0})\cap D=\emptyset$.
We have the following lemma when we identify $(x, y)\in \mathbb{R}^2$ with $(x, y, 0)\in \mathbb{R}^3$.

\begin{Lem}
We have
\[
\mathbf{V}(g)\cap (z=0)=\mathbf{V}(g|_{z=0}) \sqcup D.
\]
Therefore, the $2$-dimensional part of $\mathbf{V}(g)\cap (z=0)$ is $D$ and all the terms of $g|_{z=0}$ that correspond to the domains of $\mathbb{R}^2 \setminus \mathbf{V}(g|_{z=0})$ which intersect $\mathbb{R}^2_{\leq 0}$ are contained in  $A'$.
\end{Lem}

\begin{proof}
Assume that $(s, t, 0)\in \mathbf{V}(g)\cap (z=0)$.
By the definition of $\mathbf{V}(g)$, there exist two terms $\alpha _{a_1b_1c_1}x^{a_1} y^{b_1} z^{c_1}, \alpha _{a_2b_2c_2}x^{a_2} y^{b_2} z^{c_2}$ of $g$ such that for all $(a, b, c)\neq(a_1, b_1, c_1), (a_2, b_2, c_2)$,
\[
g(s, t,0)=\alpha _{a_1b_1c_1}+a_1s+b_1t=\alpha _{a_2b_2c_2}+a_2s+b_2t\geq \alpha _{abc}+as+bt.
\]
In particular, we have
\[
\forall c\in \mathbb{Z},\ \alpha_{a_1b_1c_1}+a_1s+b_1t\geq \alpha_{a_1b_1c}+a_1s+b_1t,
\]
and therefore $\alpha_{a_1b_1}=\alpha_{a_1b_1c_1}$.
We have $\alpha_{a_2b_2}=\alpha_{a_2b_2c_2}$ in the same way.
We also have
\[
\forall (a, b)\neq(a_1, b_1), (a_2, b_2),\ \alpha _{a_1b_1}+a_1s+b_1t=\alpha _{a_2b_2}+a_2s+b_2t\geq \alpha _{ab}+as+bt.
\]
We assume that $(s, t)\notin \mathbf{V}(g|_{z=0})$.
This does not contradict the previous inequality only if $\alpha_{a_1b_1}x^{a_1}y^{b_1}$ and $\alpha_{a_2b_2}x^{a_2}y^{b_2}$ are the same term, i.e. $(a_1, b_1)=(a_2, b_2)$.
Therefore, we have
\[
\alpha_{a_1b_1c_1}=\alpha_{a_2b_2c_2}=\max(\{\!\{ \alpha_{a_1b_10}, \ \alpha_{a_1b_11}, \ldots, \ \alpha_{a_1, b_1, m-a_1-b_1} \}\!\}).
\]
Thus, $\alpha_{a_1b_1}x^{a_1}y^{b_1}$ belongs to $B$.
From $(s, t)\notin \mathbf{V}(g|_{z=0})$, $\alpha_{a_1b_1}x^{a_1}y^{b_1}$ belongs to $B'$, and we have $(s, t)\in D$.
Hence, we have
\[
\mathbf{V}(g)\cap (z=0)\subset \mathbf{V}(g|_{z=0}) \sqcup D.
\]

Conversely, assume that $(s, t)\in \mathbf{V}(g|_{z=0})\sqcup D$.
If $(s, t)\in \mathbf{V}(g|_{z=0})$, there exist two terms $\alpha _{a_1b_1}x^{a_1} y^{b_1}$, $\alpha _{a_2b_2}x^{a_2} y^{b_2}$ of $g|_{z=0}$ such that
\begin{eqnarray*}
g(s, t,0)=\alpha _{a_1b_1}+a_1s+b_1t=\alpha _{a_2b_2}+a_2s+b_2t\geq \alpha _{ab}+as+bt\\
\text{ for any term $\alpha_{ab}x^a y^b$ of $g|_{z=0}$.}
\end{eqnarray*}
In particular, there exist two terms $\alpha _{a_1b_1c_1}x^{a_1} y^{b_1} z^{c_1}, \alpha _{a_2b_2c_2}x^{a_2} y^{b_2} z^{c_2}$ of $g$ such that
\begin{eqnarray*}
\alpha_{a_1b_1c_1}+a_1s+b_1t=\alpha _{a_2b_2c_2}+a_2s+b_2t\geq \alpha_{abc}+as+bt\\\text{ for any term $\alpha_{abc}x^ay^bz^c$ of $g$.}
\end{eqnarray*}
If $(s, t)\in D$, there exists a term $\alpha_{ab}x^a y^b$ in $B'$ such that
\[
\alpha_{ab}+as+bt>\alpha_{a'b'}+a's+b't \text{ for any term $\alpha_{a'b'}x^{a'} y^{b'}$ other than }\alpha_{ab}x^ay^b.
\]
Since $\alpha_{ab}x^a y^b\in B'$, we have $\alpha_{ab}x^a y^b\in B$, and there are two or more terms that equal to $\alpha_{ab}$ in $\{\!\{\alpha_{ab0}, \ \alpha_{ab1}, \ldots, \ \alpha_{a, b, m-a-b}\}\!\}$.
Thus there exist two terms $\alpha _{abc}x^a y^b z^c, \alpha _{abd}x^a y^b z^{d}$ of $g$ such that
\[
\alpha_{abc}+as+bt=\alpha _{abd}+as+bt\geq \alpha_{a'b'c'}+a's+b't \text{ for any term $\alpha_{a'b'c'}x^{a'} y^{b'} z^{c'}$ of $g$.}
\]
Hence, we have
\[
\mathbf{V}(g)\cap (z=0)\supset \mathbf{V}(g|_{z=0}) \sqcup D.
\]
\end{proof}

Note that $\Newt(g|_{z=0})=\conv(\{(0, 0), (m, 0), (0, m)\})=:\Delta_m$.
Let $K:=\Subdiv_{g|_{z=0}}$ be the dual subdivision of $\Delta_m$ given by $g|_{z=0}$.
The set $\mathbb{R}^2 \setminus \mathbf{V}(g|_{z=0})$ is a union of domains which correspond to some of the monomials of $g|_{z=0}$, and these monomials correspond to vertices $(\in \mathbb{Z}^2)$ of $K$.
The set of the vertices and edges of $\mathbf{V}(g|_{z=0})$ and the closures of the domains of $\mathbb{R}^2 \setminus \mathbf{V}(g|_{z=0})$ constitutes a polyhedral complex by Theorem \ref{dualitytheorem}.
We write $\tilde{K}$ for this polyhedral complex.
We mark all the simplices of $K$ which correspond to the simplices of $\tilde{K}$ which intersect with $\mathbb{R}^2_{\leq 0}$, or equivalently, with $\mathbb{R}_{<0}^2$ by Lemma \ref{a_0+b_0}.

\begin{Lem}
Each face of a marked simplex is marked and the union of all marked simplices is connected.
\end{Lem}

\begin{proof}
Let $A$ be a marked simplex and $B$ a face of $A$.
Let $A'\in \tilde{K}$ be the simplex corresponding to $A$ and $B'\in \tilde{K}$ the simplex corresponding to $B$.
Then $A'$ intersects $\mathbb{R}^2_{\leq 0}$ and is a face of $B'$.
If $A'$ is a vertex, we have $A'\in \mathbb{R}^2_{<0}$ by Lemma \ref{a_0+b_0}, and hence $B'$ intersects $\mathbb{R}^2_{\leq 0}$.
If $A'$ is an edge, we have $A'\cap \mathbb{R}^2_{<0}\neq \emptyset$ by Lemma \ref{a_0+b_0}, and hence $B'$ intersects $\mathbb{R}^2_{\leq 0}$.
Thus $B$ is marked.

Next, let $D$ and $E$ be marked simplices, and $D'$ and $E'$ be the simplices of $\tilde{K}$ corresponding to $D$ and $E$.
Let $p\in D'\cap \mathbb{R}^2_{\leq 0}$ and $q\in E'\cap \mathbb{R}^2_{\leq 0}$.
Note that the line segment $p$-$q$ is contained in $\mathbb{R}^2_{\leq 0}$.
Thus $D$ and $E$ are connected through the marked simplices corresponding to the simplices of $\tilde{K}$ which intersect the line segment $p$-$q$.
\end{proof}

Thus, we have a subcomplex $K'$ of $K$ whose support is connected.
This complex $K'$ consists of $0$, $1$ and $2$-simplices.

\begin{notation}
Let $K'_{i}$ denote the set of $i$-simplices of $K'$ and $|K'_i|$ the set underlying the complex $K'_i$, and so on .
\end{notation}

\begin{Lem}\label{markedorigin}
The origin $(0, 0)$ is in $K'_0$, i.e. it is marked.
\end{Lem}

\begin{proof}
Since $\alpha_{00}\neq -\infty$, for sufficiently large $t>0$, we have
\[
\alpha_{00}+0\times(-t)+0\times(-t)=\alpha_{00}>\alpha_{ab}+a(-t)+b(-t)\text{ for any nonzero }(a, b)\in \mathbb{Z}_{\geq 0}^2. 
\]
Here, we have $(-t, -t)\in \mathbb{R}^2_{\leq 0}$.
This means that the domain corresponding to $(0, 0)$ is nonempty and intersects $\mathbb{R}_{\leq 0}^2$.
\end{proof}

\begin{Lem}\label{renketsuseibun}
The set of connected components of $\lvert K'\rvert \setminus \lvert K'_0 \rvert$ is in one-to-one correspondence with the set of connected components of $\mathbf{V}_{\!-}(g|_{z=0})=\mathbf{V}(g|_{z=0})\cap \mathbb{R}^2_{\leq 0}$.
\end{Lem}

\begin{proof}
Note that since a simplex $A$ of $\tilde{K}$ is a convex set, the intersection $A\cap \mathbb{R}^2_{\leq 0}$ is a connected convex set if it is nonempty.
Let $V$ and $E$ be a vertex and an edge of $\mathbf{V}_{\!-}(g|_{z=0})$.
Let $E'$ and $\Delta$ be the elements of $K'_1$ and $K'_2$ corresponding to $E$ and $V$.
The vertex $V$ is an endpoint of $E$ if and only if $\Delta$ contain $E'$ as its face.
Therefore, the set of connected components of $\lvert K'\rvert \setminus \lvert K'_0 \rvert$ is in one-to-one correspondence with the set of connected components of $\mathbf{V}_{\!-}(g|_{z=0})$.
\end{proof}

To describe what $K'$ looks like (Figure \ref{|K'|}), we start with a number of lemmas.
There are functions $u\colon \{0, 1, \dots, a_0\} \to \{0, 1, \dots, m\}$ and $r\colon \{0, 1, \dots, b_0\} \to \{0, 1, \dots, m\}$ such that for any $i\in \{0, 1, \dots, a_0\}$ and $j\in \{0, 1, \dots, b_0\}$, $(i, u(i))$ is a vertex of $K'$ which corresponds to a domain which intersects $X$ and $(r(j), j)$ is a vertex of $K'$ which corresponds to a domain which intersects $Y$.
Note that the vertices of $K'$ correspond to the domains of $\mathbb{R}^2_{\leq 0} \setminus \mathbf{V}(g|_{z=0})$.
Hence by Corollary \ref{XYZW}, the functions $u$ and $r$ are well-defined.
We remark that edges $(i, u(i))$-$(i-1, u(i-1))$ and $(r(j), j)$-$(r(j-1), j-1)$ can coincide, in which case they are of slope $1$.

\begin{Lem}\label{uenitotsu}
We have
\begin{eqnarray*}
\alpha_{0u(0)}&<&\alpha_{1u(1)}<\dots <\alpha_{a_0-1, u(a_0-1)}<\alpha_{a_0u(a_0)}=\alpha_{a_0b_0},\\
\alpha_{r(0)0}&<&\alpha_{r(1)1}<\dots <\alpha_{r(b_0-1), b_0-1}<\alpha_{r(b_0)b_0}=\alpha_{a_0b_0}.
\end{eqnarray*}
If we define a piecewise linear function on the interval $[0, a_0]$ by connecting the points $(0, \alpha_{0u(0)})$, $(1,\alpha_{1u(1)}), \dots, (a_0, \alpha_{a_0u(a_0)})$ and a piecewise linear function on the interval $[0, b_0]$ by connecting the points $(0, \alpha_{r(0)0})$, $(1,\alpha_{r(1)1}), \dots, (b_0, \alpha_{r(b_0)b_0})$, then these functions are concave downward.
\end{Lem}

\begin{proof}
We set
\[
\mathbf{V}(g|_{z=0}) \cap X=\{ (s_1, 0), (s_2, 0), \dots, (s_{a_0}, 0)\} \ (s_1<s_2<\dots <s_{a_0}<0).
\]
Since the values $s_i\ (1\leq i\leq a_0)$ are obtained as solutions of $\alpha_{i-1, u(i-1)}+(i-1)x=\alpha_{iu(i)}+ix$, we have
\[
s_i=\alpha_{i-1, u(i-1)}-\alpha_{iu(i)}.
\]
Because of the inequalities $s_1<s_2<\dots <s_{a_0}<0$, we have
\begin{eqnarray*}
&\alpha_{0u(0)}<\alpha_{1u(1)}<\dots <\alpha_{a_0-1, u(a_0-1)}<\alpha_{a_0u(a_0)},&\\
&\alpha_{a_0u(a_0)} -\alpha_{a_0-1, u(a_0-1)} <\dots <\alpha_{2u(2)} -\alpha_{1u(1)}<\alpha_{1u(1)} -\alpha_{0u(0)}.&
\end{eqnarray*}
Therefore, the function on the interval $[0, a_0]$ obtained by connecting the points $(0, \alpha_{0u(0)})$, $(1,\alpha_{1u(1)})$, $\dots$, $(a_0, \alpha_{a_0u(a_0)})$ is concave downward.

The statement on $\alpha_{r(0)0}, \dots, \alpha_{r(b_0)b_0}$ follows in the same way.
\end{proof}

\begin{Lem}\label{yuiitsunosaidaichi}
For all $i\in \{0, \dots, a_0\}$, the term $\alpha_{iu(i)}$ is the only maximum of $\{\!\{ \alpha_{xy}\mid 0\leq x\leq i,\ (x, y)\in \Delta_m\}\!\}$.
Similarly, for all $j\in \{0, \dots, b_0\}$, the term $\alpha_{r(j)j}$ is the only maximum of $\{\!\{ \alpha_{xy}\mid 0\leq y\leq j, \ (x, y)\in \Delta_m\}\!\}$.
\end{Lem}

\begin{proof}
First, we show that for all $i\in \{0, \dots, a_0\}$, $\alpha_{iu(i)}$ is the only maximum of $\{\!\{ \alpha_{i0}, \dots, \alpha_{i, m-i}\}\!\}$.
Since the domain which corresponds to $(i, u(i))$ has an intersection with $X$, there exists a point $(-t, 0)\in X$ such that
\[
\forall (a, b)\neq (i, u(i)),\ \alpha_{iu(i)}+i(-t)+u(i)\times 0>\alpha_{ab}+a(-t)+b\times 0.
\]
Therefore, we have 
\begin{eqnarray*}
\forall j\in \{0, \dots, m-i\} \setminus \{u(i)\},\ &\alpha_{iu(i)}+i(-t)>\alpha_{ij}+i(-t),&\\
\text{i.e.}&\alpha_{iu(i)}>\alpha_{ij}.&
\end{eqnarray*}
Thus, $\alpha_{iu(i)}$ is the only maximum of $\{\!\{ \alpha_{i0}, \dots, \alpha_{i, m-i}\}\!\}$.
By Lemma \ref{uenitotsu}, we have
\[
\alpha_{0u(0)}<\alpha_{1u(1)}<\dots <\alpha_{a_0-1, u(a_0-1)}<\alpha_{a_0u(a_0)},
\]
and $\alpha_{iu(i)}$ is the only maximum of $\{\!\{ \alpha_{xy}\mid 0\leq x\leq i,\ (x, y)\in \Delta_m\}\!\}$.
The second claim is proven in the same way.
\end{proof}

\begin{Lem}\label{ur}
Assume that $0\leq i\leq a_0$ and $0\leq j\leq b_0$.
Then, we have
\begin{eqnarray*}
r(j)\leq i&\Rightarrow& j\leq u(i), \\
u(i)\leq j&\Rightarrow& i\leq r(j), \\
r(j)\leq i,\ u(i)=j &\Rightarrow& (r(j), j)=(i, u(i)), \\
u(i)\leq j,\ r(j)=i &\Rightarrow& (r(j), j)=(i, u(i)). 
\end{eqnarray*}
\end{Lem}

\begin{proof}
We will show the first claim.
Assume that $r(j)\leq i$ and $j\geq u(i)$.
Since $\alpha_{iu(i)}\in \{ \alpha_{xy}\mid y\leq j\}$, by Lemma \ref{yuiitsunosaidaichi}, we have
\[
\alpha_{iu(i)} \leq \alpha_{r(j)j}.
\]
We have $\alpha_{r(j)j} \leq \alpha_{r(j)u(r(j))}$ by $\alpha_{r(j)j}\in \{ \alpha_{xy}\mid x\leq r(j)\}$.
Since $r(j)\leq i$, we have $\alpha_{r(j)u(r(j))}\leq \alpha_{iu(i)}$.
Therefore, we have
\[
\alpha_{r(j)j} \leq \alpha_{iu(i)},
\]
hence
\[
\alpha_{r(j)j}= \alpha_{iu(i)}. 
\]
Here, by Lemma \ref{yuiitsunosaidaichi}, $\alpha_{r(j)j}$ is the only maximum of $\{\!\{ \alpha_{xy}\mid y\leq j\}\!\}$, hence we have $(r(j), j)=(i, u(i))$. 
The second claim holds in the same way.
The third and fourth claims hold from the above proof.
\end{proof}

The path obtained by connecting the lattice points which correspond to the domains which have intersections with $X$ and the path obtained by connecting the lattice points which correspond to the domains which have intersections with $Y$ are as in Figure \ref{kousitennnozu}.
\begin{figure}[H]
\centering
\begin{equation*}
\begin{array}{lccrr}
(0, u(0))\leftarrow (1, u(1))\leftarrow \dots \leftarrow (a_0-1, u(a_0-1))\leftarrow (a_0, u(a_0))=&(a_0, b_0)&\\
&||&\\
&(r(b_0), b_0)&\\
&\downarrow&\\
&(r(b_0-1), b_0-1)&\\
&\downarrow&\\
&\vdots&\\
&\downarrow&\\
&(r(1), 1)&\\
&\downarrow&\\
&(r(0), 0)&
\end{array}
\end{equation*}
\caption{The lattice points corresponding to the domains which have intersections with $X\cup Y$.}
\label{kousitennnozu}
\end{figure}

Let $G$ be the line graph obtained by connecting the points $(i, u(i))$, $0\leq i\leq a_0$ and $G'$ the line graph obtained by connecting the points $(r(j), j)$, $0\leq j\leq b_0$.
We also write the function on the interval $[0, a_0]$ obtained by $G$ as $u$ and the function on the interval $[0, b_0]$ obtained by $G'$ as $r$.
Let $U$ be the upper part of $G$, $\{ (x, y)\mid 0\leq x\leq a_0,\ y>u(x)\}$, and $R$ the right part of $G'$, $\{ (x, y)\mid 0\leq y\leq b_0,\ x>r(y)\}$ (see Figure \ref{|K'|}).

Let us see that the graphs $G$ and $G'$ do not ``cross'':
\begin{Lem}\label{sakaime}
The intersection of $U$ and $R$ is empty.
\end{Lem}

\begin{proof}
Note that the edges $(i, u(i))$-$(i+1, u(i+1))$ and $(r(j), j)$-$(r(j+1), j+1)$ belong to $K'$.
Therefore, two such edges of $G$ and $G'$ intersect only when they are the same or share one of the endpoints.
Assume that $U\cap R\neq \emptyset$.
We first show that $G\cap R\neq \emptyset$. 
We have $R=(R\cap \bar{U})\sqcup (R\setminus \bar{U})$.
The latter set $R\setminus \bar{U}$ is nonempty and open in $R$.
If $G\cap R=\emptyset$, we have $R\cap \bar{U}=R\cap (\Int(U)\cup G\cup \{(0, y)\mid y\geq u(0)\} \cup \{(a_0, y)\mid y\geq b_0\})=R\cap \Int(U)\neq \emptyset$, and hence we have $R=(R\cap \Int(U))\sqcup (R\setminus \bar{U})$, and this contradicts the connectedness of $R$.
Therefore, we have $G\cap R\neq \emptyset$.
Let $i=\inf \{x\mid (x, u(x))\in R\}$.
This is an integer by the remark at the beginning.
Let $j=u(i)$.
Then, we have $r(j)=i$.
Here, $u(i+1)\leq j$ is impossible by the second inequality of Lemma \ref{ur}, hence we have $u(i+1)\geq j+1$.
Since the line segment which connects $(i, j)$ and $(i+1, u(i+1))$ has an intersection with the interior of $R$, we have $r(j+1)\leq i$ and this contradicts Lemma \ref{ur}.
Therefore, we have $U\cap R= \emptyset$.
\end{proof}

The complex $K'$ corresponds to the part surrounded by the path
\begin{eqnarray*}
(r(0), 0)\to (r(1), 1)\to \dots \to (r(b_0-1), b_0-1)\to (r(b_0), b_0)=(a_0, b_0)=(a_0, u(a_0))\to \\
\to (a_0-1, u(a_0-1))\to \dots \to (1, u(1))\to (0, u(0))\to \dots \to (0, 0)\to \dots \to (r(0), 0).
\end{eqnarray*}

Rigorously, we have the following.

\begin{Prop}\label{K'nokatachi}
\[
|K'|=\mathbb{R}^2_{\geq 0}\setminus \left\{ (x, y)\in \mathbb{R}^2 \middle| 
\begin{array}{l}
0\leq x\leq a_0,\ y>u(x)\ \rm{ or}\\
0\leq y\leq b_0,\ x>r(y)\ \rm{ or}\\
x>a_0,\ y>b_0
\end{array} \right\}=\Delta_m\setminus (U\cup R).
\]
\begin{figure}[H]
\centering
\begin{tikzpicture}
 \coordinate (A1) at (0,0);
 \coordinate (A2) at (0.5,0);
 \coordinate (A3) at (3,0);
 \coordinate (A4) at (1.5,0.5);
 \coordinate (A5) at (1,1);
 \coordinate (A6) at (1.5,1.5);
 \coordinate (A7) at (0.5,2.5);
 \coordinate (A8) at (0,1.5);
 \coordinate (A9) at (0,3);
 \fill[pattern=north east lines] (A1)--(A2)--(A4)--(A5)--(A7)--(A8)--cycle;
 \draw[thin,->,>=stealth] (A1)--(3.5,0) node[right] {$x$};
 \draw[thin,->,>=stealth] (A1)--(0,3.5) node[above] {$y$};
 \draw[thick, dotted] (A3)--(A9);
 \draw (A6)--(3.5,1.5);
 \draw (A6)--(1.5,3.5);
 \draw (A8)--(A1)--(A2);
 \draw[thick] (A2)--(A4)--(A5)--(A7)--(A8);
 \draw[thick] (A5)--(A6);
 \foreach \t in {2,4,5,6,7,8} \fill[black] (A\t) circle (0.06);
 \fill[black] (A3) circle (0.045);
 \fill[black] (A9) circle (0.045);
 \coordinate [label=below:$m$] (b1) at (3,0);
 \coordinate [label=left:$m$] (b2) at (0,3); 
 \coordinate [label=below:$U$] (a1) at (0.75,3.15);
 \coordinate [label=below:$R$] (a2) at (2.65,1);
 \coordinate [label=below:$|K'|$] (a3) at (-0.5,1.25);
 \coordinate [label=above:$\text{$(a_0, b_0)$}$] (a4) at (2.15,1.5);
\end{tikzpicture}
\caption{The set underlying the complex $K'$.}
\label{|K'|}
\end{figure}
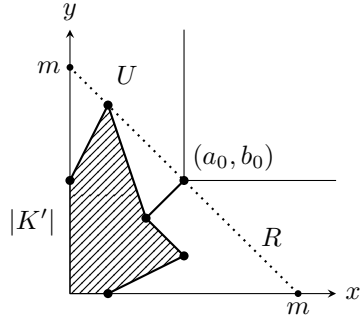
\end{Prop}

\begin{proof}
By Lemma \ref{markedorigin}, we have $(0, 0)\in K'$.
A lattice point is in $|K'|$ if and only if it is connected with the origin by marked edges which do not intersect $U$ and $R$.
\end{proof}

We saw how to describe the intersection of $\mathbf{V}(g)$ and $XY$.
We can concider the intersections of $\mathbf{V}(g)$ and the other quarter planes in the same way.

\begin{Def}
We define the $6$ linear maps from $\mathbb{R}^2$ to $\mathbb{R}^3$ as follows:
\begin{eqnarray*}
&&\iota_{XY}\colon \mathbb{R}^2\to \mathbb{R}^3;\ (x, y)\mapsto (x, y, 0),\hspace{20.3mm}\iota_{YZ}\colon \mathbb{R}^2\to \mathbb{R}^3;\ (y, z)\mapsto (0, y, z),\\
&&\hspace{0.4mm}\iota_{ZX}\colon \mathbb{R}^2\to \mathbb{R}^3;\ (z, x)\mapsto (x, 0, z),\hspace{19.2mm}\iota_{XW}\colon \mathbb{R}^2\to \mathbb{R}^3;\ (x, w)\mapsto (x-w, -w, -w),\\
&&\iota_{YW}\colon \mathbb{R}^2\to \mathbb{R}^3;\ (y, w)\mapsto (-w, y-w, -w),\hspace{5mm}\iota_{ZW}\colon \mathbb{R}^2\to \mathbb{R}^3;\ (z, w)\mapsto (-w, -w, z-w).
\end{eqnarray*}
\end{Def}

For these maps, the following holds.

\begin{Prop}
\begin{eqnarray*}
\iota_{XY}(\mathbf{V}_{\!-}(g^h|_{z=w=0}))=\mathbf{V}(g)\cap XY, &\quad& \iota_{YZ}(\mathbf{V}_{\!-}(g^h|_{x=w=0}))=\mathbf{V}(g)\cap YZ, \\
\iota_{ZX}(\mathbf{V}_{\!-}(g^h|_{y=w=0}))=\mathbf{V}(g)\cap ZX, &\quad& \iota_{XW}(\mathbf{V}_{\!-}(g^h|_{y=z=0}))=\mathbf{V}(g)\cap XW, \\
\iota_{YW}(\mathbf{V}_{\!-}(g^h|_{x=z=0}))=\mathbf{V}(g)\cap YW, &\quad& \iota_{ZW}(\mathbf{V}_{\!-}(g^h|_{x=y=0}))=\mathbf{V}(g)\cap ZW.
\end{eqnarray*}
\end{Prop}

\begin{proof}
We have already argued the first claim holds; the other five follow by identical arguments.
\end{proof}

\section{Proof of the main theorem}
Assume that $f$ and $g$ satisfy the same assumptions as in the previous section.

\begin{notation}
For a subset $A$ of $\mathbb{R}^n$, let $\pi_0(A)$ denote the set of its connected components.
\end{notation}

\begin{notation}
For a tropical polynomial $p$ in $n$ variables, we set
\[
\mathbf{D}(p):=\mathbb{R}^n\setminus \mathbf{V}(p).
\]
\end{notation}

Each element of $\pi_0 (\mathbf{D}(p))$ corresponds to a $0$-simplex of $\text{Subdiv}_p$.

\begin{notation}
For a tropical polynomial $q$ in two variables, let $\mathcal{P}_2(q)$ denote the set of all elements of $\pi_0(\mathbf{D}(q))$ which have nonempty intersections with $\mathbb{R}^2_{\leq 0}$, $\mathcal{P}_1(q)$ the set of all edges of $\mathbf{V}(q)$ which have nonempty intersections with $\mathbb{R}^2_{\leq 0}$, $\mathcal{P}_0(q)$ the set of all vertices of $\mathbf{V}(q)$ which have nonempty intersections with $\mathbb{R}^2_{\leq 0}$.
\end{notation}

Note that the elements of $\mathcal{P}_i(g|_{z=0})$ in fact intersect with $\mathbb{R}_{<0}^2$ by Lemma \ref{a_0+b_0}.
Recall that in the previous section, we wrote $K$ for the dual lattice subdivision of $\Delta_m$ which corresponds to $g|_{z=0}=g^h|_{z=w=0}$ and $K'$ for the subcomplex of $K$ which consists of $(2-i)$-simplices, $i=0, 1, 2$, corresponding to elements of $\mathcal{P}_i(g^h|_{z=w=0})$.

\begin{notation}
We write $K_{XY}$ for the subcomplex of $\text{Subdiv}_{g^h|_{z=w=0}}$ which consists of $(2-i)$-simplices, $i=0, 1, 2$, corresponding to elements of $\mathcal{P}_i(g^h|_{z=w=0})$.
We define $K_{YZ}, \dots, K_{ZW}$ in a similar way.
Let $K_{**}$ denote one of the complexes $K_{XY}, \dots, K_{YW}$ and $K_{ZW}$.
\end{notation}

In this section, we give a concrete description of the skeleton of $C=\mathbf{V}(f)\cap \mathbf{V}(g)$.
The next lemma is useful.

\begin{Lem}\label{cycle}
A domain $D\in \mathcal{P}_2(g^h|_{z=w=0})$ is enclosed by a cycle in $\mathbf{V}_{\!-}(g^h|_{z=w=0})$ if and only if the $0$-simplex of $K_{XY}$ corresponding to $D$ is contained in the interior of $|K_{XY}|$.
The same is true for $g^h|_{x=w=0}, \dots, g^h|_{x=y=0}$ and $K_{YZ}, \dots, K_{ZW}$.
\end{Lem}

\begin{proof}
First, note that a domain $D\in \pi_0 (\mathbf{D}(g^h|_{z=w=0}))$ is enclosed by a cycle in $\mathbf{V}(g^h|_{z=w=0})$ if and only if the corresponding $0$-simplex $P$ of $\text{Subdiv}_{g^h|_{z=w=0}}$ is in the interior of $|\text{Subdiv}_{g^h|_{z=w=0}}|$.
This is also equivalent to the condition that the point $P$ is in the interior of the union of the $2$-simplices of $\text{Subdiv}_{g^h|_{z=w=0}}$ which contain $P$ as their cell.
Here, the $1$-simplices which are adjacent to $P$ correspond to the edges which enclose $D$.
Therefore, a domain $D\in \mathcal{P}_2(g^h|_{z=w=0})$ is enclosed by a cycle in $\mathbf{V}_{\!-}(g^h|_{z=w=0})$ if and only if the $0$-simplex of $K_{XY}$ corresponding to $D$ is contained in the interior of $|K_{XY}|$.
\end{proof}

\begin{Lem}\label{1/2}
If $C$ is a weakly smooth complete intersection, the area of each $2$-cell of $K_{**}$ is $1/2$.
In particular, $K_{XY}$, etc., contain all lattice points in their supports.
\end{Lem}

\begin{proof}
Let $\Delta$ be a $2$-cell of $K_{XY}$, $P=\Newt(f)$ and $Q=\Newt(g)$.
There is a vertex $V=(s, t, 0)$ of $C\cap (XY\setminus (X\cup Y))$ corresponding to $\Delta$.
Since $V\in \mathbf{V}_{\!-}(g)$ and $V$ is a trivalent vertex of $C$, there exist three terms $\alpha_{a_1b_1c_1}x^{a_1}y^{b_1}z^{c_1}$, $\alpha_{a_2b_2c_2}x^{a_2}y^{b_2}z^{c_2}$ and $\alpha_{a_3b_3c_3}x^{a_3}y^{b_3}z^{c_3}$ of $g$ such that $\Delta=\conv(\{(a_1, b_1), (a_2, b_2), (a_3, b_3)\})$, and that for all terms $\alpha_{abc}x^ay^bz^c$ of $g$, we have
\[
v:=\alpha_{a_1b_1c_1}+a_1s+b_1t=\alpha_{a_2b_2c_2}+a_2s+b_2t=\alpha_{a_3b_3c_3}+a_3s+b_3t\geq \alpha_{abc}+as+bt.
\]
Let
\begin{eqnarray*}
T_1&=&\left\{ \begin{pmatrix} 0 \\ 0 \\ 0 \\ 0 \\ 0 \end{pmatrix}, \begin{pmatrix} 0 \\ 0 \\ 1 \\ 0 \\ 0 \end{pmatrix} \right\},\quad T_2=\left\{ \begin{pmatrix} a \\ b \\ c \\ 1 \\ \alpha_{abc} \end{pmatrix} \middle| \alpha_{abc}x^ay^bz^c\text{ is a term of $g$}\right\} \cup \left\{ \begin{pmatrix} 1 \\ 0 \\ 0 \\ 0 \\ 0 \end{pmatrix}, \begin{pmatrix} 0 \\ 1 \\ 0 \\ 0 \\ 0 \end{pmatrix}\right\},\\
\mathbf{p}&=&(-s, -t, 0, v, -1),\quad \mathbf{w}_i=\begin{pmatrix} a_i \\ b_i \\ c_i \\ 1 \\ \alpha_{a_ib_ic_i} \end{pmatrix} (i=1, 2, 3).
\end{eqnarray*}
Then, we have $\mathbf{w}_i\in T_2$ ($i=1, 2, 3$) and
\[
\forall \mathbf{v}\in T_1,\ \forall \mathbf{w}\in T_2,\ 0=\mathbf{p}\cdot \mathbf{v}=\mathbf{p}\cdot \mathbf{w}_1=\mathbf{p}\cdot \mathbf{w}_2=\mathbf{p}\cdot \mathbf{w}_3\leq \mathbf{p}\cdot \mathbf{w}.
\]
Hence there is a four-dimensional cell in the Cayley subdivision of $\Cay(P, Q)$ which contains $\conv \left\{ \begin{pmatrix} 0 \\ 0 \\ 0 \\ 0 \end{pmatrix}, \begin{pmatrix} 0 \\ 0 \\ 1 \\ 0 \end{pmatrix}, \begin{pmatrix} a_1 \\ b_1 \\ c_1 \\ 1 \end{pmatrix}, \begin{pmatrix} a_2 \\ b_2 \\ c_2 \\ 1  \end{pmatrix}, \begin{pmatrix} a_3 \\ b_3 \\ c_3 \\ 1  \end{pmatrix}\right\}$.
Since $C$ is a weakly smooth tropical complete intersection curve, we have $1=\left| \det \begin{pmatrix} 0 & a_1 & a_2 & a_3\\ 0 & b_1 & b_2 &b_3\\ 1 & c_1 & c_2 & c_3\\ 0 & 1 & 1 & 1 \end{pmatrix} \right|=\left| \det \begin{pmatrix} a_2-a_1 & a_3-a_1\\ b_2-b_1 & b_3-b_1 \end{pmatrix} \right|=2\lvert \Delta \rvert$.
Thus we have $\lvert \Delta \rvert=1/2$.
The cases  of $K_{YZ}, \dots, K_{YW}$ and $K_{ZW}$ are similar.
\end{proof}

The following is our main theorem.

\begin{Th}\label{mainth}
Let $f$ and $g$ be tropical polynomials of degrees $d$ and $e$ in three variables.
We assume the following:
\begin{itemize}
\item $\Newt(f)=\conv(\{ (0, 0, 0), (d, 0, 0), (0, d, 0), (0, 0, d)\})$,
\item $\Newt(g)=\conv(\{ (0, 0, 0), (e, 0, 0), (0, e, 0), (0, 0, e)\})$,
\item $C=\mathbf{V}(f)\cap \mathbf{V}(g)$ is a smooth complete intersection curve.
\end{itemize}
Then the skeleton of $C$ is not the lollipop graph of genus $3$.
\end{Th}

From here on, we will prove this theorem.
First, by Lemma \ref{d=4}, we may assume that $(d, e)=(1, 4)$.
We may also assume that $f=x\oplus y\oplus z\oplus 0$ by translation. 

The lollipop graph of genus $3$ has three homologically independent cycles.
Here, by a cycle we will mean a subgraph homeomorphic to $S^1$.
Note that these cycles do not have intersections if the skeleton is the lollipop graph.
Thus if two distinct cycles of $C$ have a nonempty intersection, the skeleton of $C$ is not the lollipop graph.

\begin{notation}
We write $\MM_X(g^h)$ for the set of exponents of monomials of $g^h$ which correspond to the elements of $\pi_0(\mathbf{D}(g))$ which have nonempty intersections with $X$.
We define $\MM_Y(g^h)$, $\MM_Z(g^h)$ and $\MM_W(g^h)$ in a similar way.
When $\MM_X(g^h)=\{(a_0, b_0, c_0, d_0)$, $(a_0-1, b_1, c_1, d_1), \dots$, $(0, b_{a_0}, c_{a_0}, d_{a_0})\}$, we write as follows:
\[
(a_0, b_0, c_0, d_0)\rightarrow (a_0-1, b_1, c_1, d_1)\rightarrow \cdots.
\]
\end{notation}

We write $C_{XY}$ for the restriction of $C$ to $XY$, and similarly for $YZ, \dots, YW$ and $ZW$. 
We also write $ZX$ as $XZ$.
In the following, we consider skeletons of $C_{XY}$ relative to $X\cup Y$, and so on: We do not move the points on $X\cup Y$ when we transform $C_{XY}$ to its skeleton, and similarly for $C_{YZ}, \dots, C_{ZW}$.
After glueing $3$ quarter planes along $X$, for example, we allow the skeleton to move past $X$.
Note that $\lvert C\cap (X\cup Y\cup Z\cup W)\rvert=\lvert \mathbf{V}(g)\cap (X\cup Y\cup Z\cup W)\rvert=4$ by Lemma \ref{koutenyonko}.

\begin{notation}
Let $C'$ denote $C\setminus (X\cup Y\cup Z\cup W)$, and $\overline{C'}$ the graph obtained by adding a vertex to each open edge of $C'$.
We write $b_1(\overline{C'})$ for the first Betti number of $\overline{C'}$, which is equal to the sum of the numbers of interior lattice points in the supports of the complexes $K_{XY}, \dots, K_{YW}$ and $K_{ZW}$.
\end{notation}

\begin{Lem}\label{st}
Let $s$ be the number of connected components of $C'$ and $t=b_1(\overline{C'})$.
Then the pair $(s, t)$ is one of the following: $(6, 0)$, $(7, 1)$, $(8, 2)$ and $(9, 3)$.
\end{Lem}

\begin{proof}
Let $p$ be a point in $C\setminus C'$.
If we glue $\overline{C'}$ along $p$ and write $D$ for the resulting graph, the number of vertices of $D$ is $2$ less than that of $\overline{C'}$, and the number of edges does not change because $C$ is trivalent at the point $p$.
Let $e(G)$ denote the Euler characteristic of a graph $G$.
Then $e(D)$ is $2$ less than $e(\overline{C'})$.
Here, $e(C)=-2$, thus $e(\overline{C'})=6$.
Therefore, we have $s-t=6$.
Since $0\leq t\leq 3$, we have $(s, t)=(6, 0), (7, 1), (8, 2)$ or $(9, 3)$.
\end{proof}

\begin{Lem}\label{saikurunashi}
When none of the supports of the complexes $K_{**}$ contain a lattice point in its interior, i.e., $b_1(\overline{C'})=0$, the skeleton of $C$ is not the lollipop graph.
\end{Lem}

\begin{proof}
Under the assumption, $C'$ does not have cycles.
On the other hand, $C$ has three homologically independent cycles.
Each of them passes through at least two points among $C\cap (X\cup Y\cup Z\cup W)$.
Thus, at least one point of the four points is shared by two cycles in $C$.
Therefore, the skeleton of $C$ cannot be the lollipop graph.
\end{proof}

\begin{Lem}\label{1112}
Assume that $K_{**}$ contains two $0$-simplices $p$ and $q$ in its interior.
If there is a $1$-simplex containing $p$ and $q$ or there are two $2$-simplices $\Delta_1$ containing $p$ and $\Delta_2$ containing $q$ that have a common edge, the skeleton of $C$ is not the lollipop graph of genus $3$.
\end{Lem}

\begin{proof}
We can assume that $K_{XY}$ contains two $0$-simplices $p$ and $q$ in its interior.
Note that two points $p$ and $q$, or more precisely, edges adjacent to them, correspond to cycles of $C_{XY}$.
If there is a $1$-simplex containing $p$ and $q$ as in Figure \ref{p,qnozu} (a), the two cycles have a common edge and the skeleton of $C_{XY}$ contains the shape as in Figure \ref{p,qnozu} (a') and the skeleton of $C$ is not the lollipop graph.
If there are two $2$-simplices $\Delta_1$ containing $p$ and $\Delta_2$ containing $q$ that have a common edge as in Figure \ref{p,qnozu} (b), the two cycles are connected by an edge that does not branch on the way and the skeleton of $C_{XY}$ contains the shape as in Figure \ref{p,qnozu} (b'), where the middle edge does not branch on the way, and the skeleton of $C$ is not the lollipop graph.
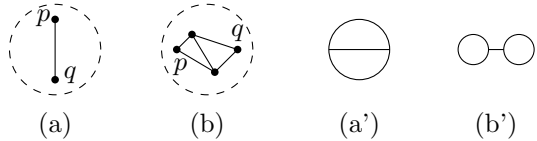
\begin{figure}[H]
\centering
\begin{tikzpicture}
\draw[dashed] (-2,0) circle (0.6);
\draw (-2,0.4)--(-2,-0.4);
\draw[dashed] (0,0) circle (0.6);
\draw (-0.4,0)--(-0.2,0.2)--(0.4,0)--(0.1,-0.3)--cycle;
\draw (-0.2,0.2)--(0.1,-0.3);
\draw (3.5,0) circle (0.2);
\draw (4.1,0) circle (0.2);
\draw (3.7,0)--(3.9,0);
\draw (2,0) circle (0.4);
\draw (1.6,0)--(2.4,0);
\fill[black] (-2,0.4) circle (0.05);
\fill[black] (-2,-0.4) circle (0.05);
\fill[black] (-0.4,0) circle (0.05);
\fill[black] (-0.2,0.2) circle (0.05);
\fill[black] (0.4,0) circle (0.05);
\fill[black] (0.1,-0.3) circle (0.05);
\coordinate [label=below:(a)] (a) at (-2,-0.7);
\coordinate [label=below:(b)] (b) at (0,-0.7);
\coordinate [label=below:(a')] (a) at (2,-0.7);
\coordinate [label=below:(b')] (b) at (3.8,-0.7);
\coordinate [label=below:$p$] (a) at (-2.15,0.65);
\coordinate [label=below:$q$] (b) at (-1.8,-0.1);
\coordinate [label=below:$p$] (a) at (-0.35,0);
\coordinate [label=below:$q$] (b) at (0.4,0.45);
\end{tikzpicture}
\caption{The two points $p$ and $q$ and the cycles of $C_{XY}$ discussed in the proof of Lemma \ref{1112}.}
\label{p,qnozu}
\end{figure}
\end{proof}

\begin{Lem}\label{3in1}
If the support of the complex $K_{**}$ contains the three points $(1, 1)$, $(1, 2)$ and $(2, 1)$ in its interior, the skeleton of $C$ is not the lollipop graph.
\end{Lem}

\begin{proof}
Let $p$ and $q$ be two of the three points $(1, 1)$, $(1, 2)$ and $(2, 1)$.
Here, the edge $pq$ is not in $K_{**}$ by Lemma \ref{1112}.
For example, let $p=(1, 1)$ and $q=(1, 2)$.
Since $p$ is an interior point, $K_{**}$ has a triangle that contains $p$ as its vertex and $(1, 1+\epsilon)$ in its interior for a sufficiently small $\epsilon >0$.
Note that this triangle cannot contain the point $(2, 1)$ by Lemma \ref{1112} and Lemma \ref{1/2}.
There are three candidates for the edge $E$ opposing to $p$: $(0, 1)$-$(2, 2)$, $(0, 2)$-$(3, 0)$ and $(0, 3)$-$(2, 0)$ (See Figure \ref{3sennbunn}).
Note that the edge $(0, 2)$-$(3, 1)$ is not a candidate for $E$ by Lemma \ref{1/2}.
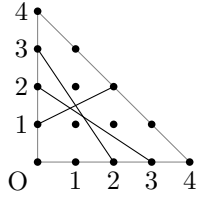
\begin{figure}[H]
\centering
\begin{tikzpicture}
 \coordinate[label=below left:O] (C1) at (6,0);
 \coordinate [label=below:1] (C2) at (6.5,0);
 \coordinate [label=below:2] (C3) at (7,0);
 \coordinate [label=below:3] (C4) at (7.5,0);
 \coordinate [label=below:4] (C5) at (8,0);
 \coordinate [label=left:1] (C6) at (6,0.5);
 \coordinate (C7) at (6.5,0.5);
 \coordinate (C8) at (7,0.5);
 \coordinate (C9) at (7.5,0.5);
 \coordinate [label=left:2] (C10) at (6,1);
 \coordinate (C11) at (6.5,1);
 \coordinate (C12) at (7,1);
 \coordinate [label=left:3] (C13) at (6,1.5);
 \coordinate (C14) at (6.5,1.5);
 \coordinate [label=left:4] (C15) at (6,2);
 \draw [very thin, gray] (C1)--(C5)--(C15)--cycle; 
 \draw (C3)--(C13);
 \draw (C4)--(C10);
 \draw (C6)--(C12);
 \foreach \t in {1,2,...,15} \fill[black] (C\t) circle (0.05);
\end{tikzpicture}
\caption{The three candidates for the edge $E$ discussed in the proof of Lemma \ref{3in1}.}
\label{3sennbunn}
\end{figure}
If $E$ is $(0, 1)$-$(2, 2)$, it is one of the edges of the triangle in $K_{**}$ containing $(1, 2-\epsilon)$.
The point $(1, 2)$ is one of the vertices of this triangle, and hence by Lemma \ref{1112}, the skeleton of $C$ is not the lollipop graph in this case.
We can treat the case where $E$ is $(0, 3)$-$(2, 0)$ in the same way.
If $E$ is $(0, 2)$-$(3, 0)$, it is one of the edges of the triangle in $K_{**}$ containing $(2, 1-\epsilon)$.
The point $(2, 1)$ is one of the vertices of this triangle, and hence by Lemma \ref{1112}, the skeleton of $C$ is not the lollipop graph in this case.
\end{proof}

\begin{Lem}\label{Case2.6}
If the complex $K_{**}$ contains the edge $(2, 0)$-$(1, 3)$ as its $1$-simplex and the points $(1, 1)$ and $(1, 2)$ are in the interior of $\lvert K_{**} \rvert$, the skeleton of $C$ is not the lollipop graph of genus $3$.
\end{Lem}

\begin{proof}
Under the assumption as in Figure \ref{lemma5.14nozu} (a), $K_{**}$ contains a $2$-simplex $\Delta$ containing the point $(1+\epsilon, 2)$ for a sufficiently small $\epsilon >0$.
At least one of the $x$-coordinates of the vertices of $\Delta$ is greater than $1$.
This vertex must be the point $(2, 0)$.
By Lemma \ref{1/2}, the edge of $\Delta$ opposing to $(2, 0)$ must be $(1, 2)$-$(1, 3)$ as in Figure \ref{lemma5.14nozu} (b).
The complex $K_{**}$ does not contain the edge $(1, 1)$-$(1, 2)$ as its $1$-simplex by Lemma \ref{1112}.
Thus the $2$-simplex of $K_{**}$ containing the point $(1, 2-\epsilon)$ is as in Figure \ref{lemma5.14nozu} (c) by Lemma \ref{1/2}.
Again by Lemma \ref{1/2}, $K_{**}$ contains the $2$-simplex containing the point $(1, 1+\epsilon)$ as in Figure \ref{lemma5.14nozu} (d), and hence the skeleton of $C$ is not the lollipop graph by Lemma \ref{1112}.
\begin{figure}[H]
\centering
\begin{tikzpicture}
 \coordinate[label=below left:O] (A1) at (0,0);
 \coordinate [label=below:1] (A2) at (0.5,0);
 \coordinate [label=below:2] (A3) at (1,0);
 \coordinate [label=below:3] (A4) at (1.5,0);
 \coordinate [label=below:4] (A5) at (2,0);
 \coordinate [label=left:1] (A6) at (0,0.5);
 \coordinate (A7) at (0.5,0.5);
 \coordinate (A8) at (1,0.5);
 \coordinate (A9) at (1.5,0.5);
 \coordinate [label=left:2] (A10) at (0,1);
 \coordinate (A11) at (0.5,1);
 \coordinate (A12) at (1,1);
 \coordinate [label=left:3] (A13) at (0,1.5);
 \coordinate (A14) at (0.5,1.5);
 \coordinate [label=left:4] (A15) at (0,2);
 \draw [very thin, gray] (A1)--(A5)--(A15)--cycle; 
 \draw (A3)--(A14);
 \foreach \t in {1,2,...,15} \fill[black] (A\t) circle (0.05);
 \coordinate[label=below left:O] (B1) at (3,0);
 \coordinate [label=below:1] (B2) at (3.5,0);
 \coordinate [label=below:2] (B3) at (4,0);
 \coordinate [label=below:3] (B4) at (4.5,0);
 \coordinate [label=below:4] (B5) at (5,0);
 \coordinate [label=left:1] (B6) at (3,0.5);
 \coordinate (B7) at (3.5,0.5);
 \coordinate (B8) at (4,0.5);
 \coordinate (B9) at (4.5,0.5);
 \coordinate [label=left:2] (B10) at (3,1);
 \coordinate (B11) at (3.5,1);
 \coordinate (B12) at (4,1);
 \coordinate [label=left:3] (B13) at (3,1.5);
 \coordinate (B14) at (3.5,1.5);
 \coordinate [label=left:4] (B15) at (3,2);
 \draw [very thin, gray] (B1)--(B5)--(B15)--cycle; 
 \draw (B3)--(B14);
 \foreach \t in {1,2,...,15} \fill[black] (B\t) circle (0.05);
 \draw (B14)--(B11)--(B3);
 \coordinate[label=below left:O] (C1) at (6,0);
 \coordinate [label=below:1] (C2) at (6.5,0);
 \coordinate [label=below:2] (C3) at (7,0);
 \coordinate [label=below:3] (C4) at (7.5,0);
 \coordinate [label=below:4] (C5) at (8,0);
 \coordinate [label=left:1] (C6) at (6,0.5);
 \coordinate (C7) at (6.5,0.5);
 \coordinate (C8) at (7,0.5);
 \coordinate (C9) at (7.5,0.5);
 \coordinate [label=left:2] (C10) at (6,1);
 \coordinate (C11) at (6.5,1);
 \coordinate (C12) at (7,1);
 \coordinate [label=left:3] (C13) at (6,1.5);
 \coordinate (C14) at (6.5,1.5);
 \coordinate [label=left:4] (C15) at (6,2);
 \draw [very thin, gray] (C1)--(C5)--(C15)--cycle; 
 \draw (C3)--(C14)--(C11)--cycle;
 \draw (C11)--(C13)--(C3);
 \foreach \t in {1,2,...,15} \fill[black] (C\t) circle (0.05); 
 \coordinate[label=below left:O] (D1) at (9,0);
 \coordinate [label=below:1] (D2) at (9.5,0);
 \coordinate [label=below:2] (D3) at (10,0);
 \coordinate [label=below:3] (D4) at (10.5,0);
 \coordinate [label=below:4] (D5) at (11,0);
 \coordinate [label=left:1] (D6) at (9,0.5);
 \coordinate (D7) at (9.5,0.5);
 \coordinate (D8) at (10,0.5);
 \coordinate (D9) at (10.5,0.5);
 \coordinate [label=left:2] (D10) at (9,1);
 \coordinate (D11) at (9.5,1);
 \coordinate (D12) at (10,1);
 \coordinate [label=left:3] (D13) at (9,1.5);
 \coordinate (D14) at (9.5,1.5);
 \coordinate [label=left:4] (D15) at (9,2);
 \draw [very thin, gray] (D1)--(D5)--(D15)--cycle; 
 \draw (D3)--(D14)--(D11)--cycle;
 \draw (D11)--(D13)--(D3);
 \draw (D3)--(D7)--(D13);
 \foreach \t in {1,2,...,15} \fill[black] (D\t) circle (0.05);
 \coordinate [label=below:(a)] (a1) at (1,-0.5);
 \coordinate [label=below:(b)] (a2) at (4,-0.5);
 \coordinate [label=below:(c)] (a3) at (7,-0.5);
 \coordinate [label=below:(d)] (a4) at (10,-0.5);
 \end{tikzpicture}
\caption{The cells of $K_{**}$ discussed in the proof of Lemma \ref{Case2.6}.}
\label{lemma5.14nozu}
\end{figure}
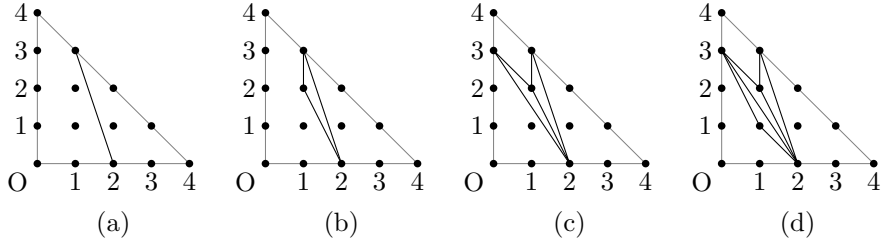
\end{proof}

\begin{Lem}\label{12}
\begin{enumerate}[(1)]
\item Assume that $K_{**}$ contains the points $(1, 1)$ and $(1, 2)$ as its interior points and the line segment $(1, 3)$-$(2, 1)$ as its $1$-simplex as in Figure \ref{lemma5.15nozu} \textup{(a)}.
If the skeleton of $C$ is the lollipop graph of genus $3$, the complex contains the $2$-simplices as in Figure \ref{lemma5.15nozu} \textup{(b)}.
\item Assume that $K_{**}$ contains the points $(1, 1)$ and $(1, 2)$ as its interior points and the line segment $(2, 1)$-$(2, 2)$ as its $1$-simplex as in Figure \ref{lemma5.15nozu} \textup{(c)}.
If the skeleton of $C$ is the lollipop graph of genus $3$, the complex contains the $2$-simplices as in Figure \ref{lemma5.15nozu} \textup{(d)}.
\item Assume that $K_{**}$ contains the points $(1, 1)$ and $(2, 1)$ as its interior points and the line segment $(1, 2)$-$(2, 2)$ as its $1$-simplex as in Figure \ref{lemma5.15nozu} \textup{(e)}.
If the skeleton of $C$ is the lollipop graph of genus $3$, the complex contains the $2$-simplices as in Figure \ref{lemma5.15nozu} \textup{(f)}.
\end{enumerate}
\end{Lem}
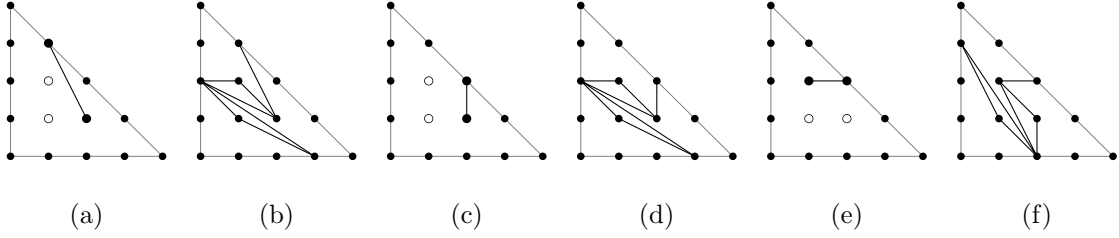
\begin{figure}[H]
\centering
\begin{tikzpicture}
 \coordinate (A1) at (0,0);
 \coordinate (A2) at (0.5,0);
 \coordinate (A3) at (1,0);
 \coordinate (A4) at (1.5,0);
 \coordinate (A5) at (2,0);
 \coordinate (A6) at (0,0.5);
 \coordinate (A7) at (0.5,0.5);
 \coordinate (A8) at (1,0.5);
 \coordinate (A9) at (1.5,0.5);
 \coordinate (A10) at (0,1);
 \coordinate (A11) at (0.5,1);
 \coordinate (A12) at (1,1);
 \coordinate (A13) at (0,1.5);
 \coordinate (A14) at (0.5,1.5);
 \coordinate (A15) at (0,2);
 \draw [very thin, gray] (A1)--(A5)--(A15)--cycle;
 \draw (A8)--(A14);
 \foreach \t in {1,2,...,15} \fill[black] (A\t) circle (0.05);
 \foreach \t in {7,8,11,14} \fill[black] (A\t) circle (0.06);
 \foreach \t in {7,11} \fill[white] (A\t) circle (0.05);
 \coordinate (B1) at (2.5,0);
 \coordinate (B2) at (3,0);
 \coordinate (B3) at (3.5,0);
 \coordinate (B4) at (4,0);
 \coordinate (B5) at (4.5,0);
 \coordinate (B6) at (2.5,0.5);
 \coordinate (B7) at (3,0.5);
 \coordinate (B8) at (3.5,0.5);
 \coordinate (B9) at (4,0.5);
 \coordinate (B10) at (2.5,1);
 \coordinate (B11) at (3,1);
 \coordinate (B12) at (3.5,1);
 \coordinate (B13) at (2.5,1.5);
 \coordinate (B14) at (3,1.5);
 \coordinate (B15) at (2.5,2);
 \draw [very thin, gray] (B1)--(B5)--(B15)--cycle; 
 \draw (B8)--(B14);
 \draw (B4)--(B10)--(B8);
 \draw (B4)--(B7)--(B10);
 \draw (B10)--(B11)--(B8);
 \foreach \t in {1,2,...,15} \fill[black] (B\t) circle (0.05);
 \coordinate (C1) at (5,0);
 \coordinate (C2) at (5.5,0);
 \coordinate (C3) at (6,0);
 \coordinate (C4) at (6.5,0);
 \coordinate (C5) at (7,0);
 \coordinate (C6) at (5,0.5);
 \coordinate (C7) at (5.5,0.5);
 \coordinate (C8) at (6,0.5);
 \coordinate (C9) at (6.5,0.5);
 \coordinate (C10) at (5,1);
 \coordinate (C11) at (5.5,1);
 \coordinate (C12) at (6,1);
 \coordinate (C13) at (5,1.5);
 \coordinate (C14) at (5.5,1.5);
 \coordinate (C15) at (5,2);
 \draw [very thin, gray] (C1)--(C5)--(C15)--cycle;
 \draw (C8)--(C12);
 \foreach \t in {1,2,...,15} \fill[black] (C\t) circle (0.05);
 \foreach \t in {7,8,11,12} \fill[black] (C\t) circle (0.06);
 \foreach \t in {7,11} \fill[white] (C\t) circle (0.05);
 \coordinate (A1) at (7.5,0);
 \coordinate (A2) at (8,0);
 \coordinate (A3) at (8.5,0);
 \coordinate (A4) at (9,0);
 \coordinate (A5) at (9.5,0);
 \coordinate (A6) at (7.5,0.5);
 \coordinate (A7) at (8,0.5);
 \coordinate (A8) at (8.5,0.5);
 \coordinate (A9) at (9,0.5);
 \coordinate (A10) at (7.5,1);
 \coordinate (A11) at (8,1);
 \coordinate (A12) at (8.5,1);
 \coordinate (A13) at (7.5,1.5);
 \coordinate (A14) at (8,1.5);
 \coordinate (A15) at (7.5,2);
 \draw [very thin, gray] (A1)--(A5)--(A15)--cycle;
 \draw (A8)--(A12);
 \draw (A8)--(A10)--(A4);
 \draw (A4)--(A7)--(A10);
 \draw (A10)--(A11)--(A8);
 \foreach \t in {1,2,...,15} \fill[black] (A\t) circle (0.05); 
 \coordinate (B1) at (10,0);
 \coordinate (B2) at (10.5,0);
 \coordinate (B3) at (11,0);
 \coordinate (B4) at (11.5,0);
 \coordinate (B5) at (12,0);
 \coordinate (B6) at (10,0.5);
 \coordinate (B7) at (10.5,0.5);
 \coordinate (B8) at (11,0.5);
 \coordinate (B9) at (11.5,0.5);
 \coordinate (B10) at (10,1);
 \coordinate (B11) at (10.5,1);
 \coordinate (B12) at (11,1);
 \coordinate (B13) at (10,1.5);
 \coordinate (B14) at (10.5,1.5);
 \coordinate (B15) at (10,2);
 \draw [very thin, gray] (B1)--(B5)--(B15)--cycle;
 \draw (B11)--(B12);
 \foreach \t in {1,2,...,15} \fill[black] (B\t) circle (0.05);
 \foreach \t in {7,8,11,12} \fill[black] (B\t) circle (0.06);
 \foreach \t in {7,8} \fill[white] (B\t) circle (0.05);
 \coordinate (C1) at (12.5,0);
 \coordinate (C2) at (13,0);
 \coordinate (C3) at (13.5,0);
 \coordinate (C4) at (14,0);
 \coordinate (C5) at (14.5,0);
 \coordinate (C6) at (12.5,0.5);
 \coordinate (C7) at (13,0.5);
 \coordinate (C8) at (13.5,0.5);
 \coordinate (C9) at (14,0.5);
 \coordinate (C10) at (12.5,1);
 \coordinate (C11) at (13,1);
 \coordinate (C12) at (13.5,1);
 \coordinate (C13) at (12.5,1.5);
 \coordinate (C14) at (13,1.5);
 \coordinate (C15) at (12.5,2);
 \draw [very thin, gray] (C1)--(C5)--(C15)--cycle;
 \draw (C11)--(C12);
 \draw (C13)--(C3)--(C11);
 \draw (C3)--(C7)--(C13);
 \draw (C3)--(C8)--(C11);
 \foreach \t in {1,2,...,15} \fill[black] (C\t) circle (0.05);
 \coordinate [label=below:(a)] (a1) at (1,-0.5);
 \coordinate [label=below:(b)] (a2) at (3.5,-0.5);
 \coordinate [label=below:(c)] (a3) at (6,-0.5);
 \coordinate [label=below:(d)] (a4) at (8.5,-0.5);
 \coordinate [label=below:(e)] (a5) at (11,-0.5);
 \coordinate [label=below:(f)] (a6) at (13.5,-0.5);
\end{tikzpicture}
\caption{The cells of $K_{**}$ discussed in Lemma \ref{12}.}
\label{lemma5.15nozu}
\end{figure}

\begin{proof}
$(1)$ By Lemma \ref{1112}, the two points $(1, 1)$ and $(1, 2)$ are not connected directly.
The complex $K_{**}$ has a triangle $\Delta_1$ that contains $(1, 1)$ as its vertex and $(1, 1+\epsilon)$ in its interior for a sufficiently small positive real number $\epsilon$.
Let $E$ be its edge opposing to $(1, 1)$.
Also, $K_{**}$ has a triangle $\Delta_2$ that contains $(1, 2)$ as its vertex and $(1, 2-\epsilon)$ in its interior.
Let $F$ be its edge opposing to $(1, 2)$.
There are three candidates for $E$ and $F$ as in Figure \ref{proofoflemma5.15nozu} (g).
Here, by Lemma \ref{1112}, we have $E\neq F$.
Therefore, $\Delta_1$ and $\Delta_2$ are as in Figure \ref{lemma5.15nozu} (b).
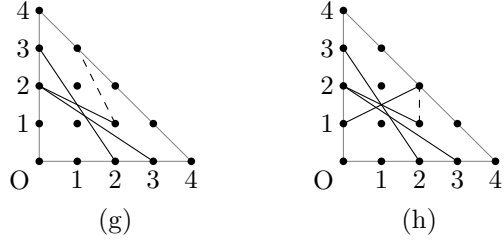
\begin{figure}[H]
\centering
\begin{tikzpicture}
 \coordinate[label=below left:O] (A1) at (0,0);
 \coordinate [label=below:1] (A2) at (0.5,0);
 \coordinate [label=below:2] (A3) at (1,0);
 \coordinate [label=below:3] (A4) at (1.5,0);
 \coordinate [label=below:4] (A5) at (2,0);
 \coordinate [label=left:1] (A6) at (0,0.5);
 \coordinate (A7) at (0.5,0.5);
 \coordinate (A8) at (1,0.5);
 \coordinate (A9) at (1.5,0.5);
 \coordinate [label=left:2] (A10) at (0,1);
 \coordinate (A11) at (0.5,1);
 \coordinate (A12) at (1,1);
 \coordinate [label=left:3] (A13) at (0,1.5);
 \coordinate (A14) at (0.5,1.5);
 \coordinate [label=left:4] (A15) at (0,2);
 \draw [very thin, gray] (A1)--(A5)--(A15)--cycle; 
 \draw [dashed] (A8)--(A14);
 \draw (A3)--(A13);
 \draw (A4)--(A10);
 \draw (A8)--(A10);
 \foreach \t in {1,2,...,15} \fill[black] (A\t) circle (0.05);
 \coordinate[label=below left:O] (B1) at (4,0);
 \coordinate [label=below:1] (B2) at (4.5,0);
 \coordinate [label=below:2] (B3) at (5,0);
 \coordinate [label=below:3] (B4) at (5.5,0);
 \coordinate [label=below:4] (B5) at (6,0);
 \coordinate [label=left:1] (B6) at (4,0.5);
 \coordinate (B7) at (4.5,0.5);
 \coordinate (B8) at (5,0.5);
 \coordinate (B9) at (5.5,0.5);
 \coordinate [label=left:2] (B10) at (4,1);
 \coordinate (B11) at (4.5,1);
 \coordinate (B12) at (5,1);
 \coordinate [label=left:3] (B13) at (4,1.5);
 \coordinate (B14) at (4.5,1.5);
 \coordinate [label=left:4] (B15) at (4,2);
 \draw [very thin, gray] (B1)--(B5)--(B15)--cycle;
 \draw [dashed] (B8)--(B12);
 \draw (B6)--(B12);
 \draw (B3)--(B13);
 \draw (B4)--(B10);
 \draw (B8)--(B10);
 \foreach \t in {1,2,...,15} \fill[black] (B\t) circle (0.05);
 \coordinate [label=below:(g)] (a1) at (1,-0.5);
 \coordinate [label=below:(h)] (a2) at (5,-0.5);
\end{tikzpicture}
\caption{The cells of $K_{**}$ discussed in the proof of Lemma \ref{12}.}
\label{proofoflemma5.15nozu}
\end{figure}
$(2)$ is shown in the same way (see Figure \ref{proofoflemma5.15nozu} (h)).
$(3)$ follows from $(2)$ by symmetry.
\end{proof}

Now we prove Theorem \ref{mainth} by a case-by-case analysis.
Along the way, we also see how the shapes of complete intersection curves, in particular their skeletons, can be described.
As we saw in the previous section, there is a symmetry between $X$, $Y$, $Z$ and $W$.
Therefore, it is sufficient to consider the following $5$ cases as the exponent $(a_0, b_0, c_0, d_0)$ of the monomial of $g^h$ which corresponds to the element of $\pi_0 (\mathbf{D}(g))$ which contains the origin, i.e. the exponent of the term with the largest coefficient.
\begin{eqnarray*}
\text{Case $1$: }(a_0, b_0, c_0, d_0)=(1, 1, 1, 1), \\
\text{Case $2$: }(a_0, b_0, c_0, d_0)=(2, 1, 1, 0), \\
\text{Case $3$: }(a_0, b_0, c_0, d_0)=(2, 2, 0, 0), \\
\text{Case $4$: }(a_0, b_0, c_0, d_0)=(3, 1, 0, 0), \\
\text{Case $5$: }(a_0, b_0, c_0, d_0)=(4, 0, 0, 0).\hspace{0.5mm}
\end{eqnarray*}

\begin{notation}\label{CcapX}
If $|C\cap X|=n$, we write $C\cap X=\{x_1, \dots, x_n\}$.
Here, we assume that for each $i$, the point $x_i$ is closer to the origin than $x_{i+1}$.
If $|C\cap X|=1$, we also write $C\cap X=\{x\}$.
We write the intersection of $C$ and the other half lines in the same way.
\end{notation}

\subsection{Case $1$: $(a_0, b_0, c_0, d_0)=(1, 1, 1, 1)$} 

We look at $\MM_X(g^h)$ and $\MM_Y(g^h)$.
Projecting to the first two components, they are as follows.
\begin{equation*}
\begin{array}{lccrr}
(0, u(0))\leftarrow&(1, 1)&\\
&\downarrow&\\
&(r(0), 0)&
\end{array}
\end{equation*}

These two paths are selected from the dashed and dotted lines in Figure \ref{fc1} (a).
Thus the skeleton of $C_{XY}$ is as in (b).
\begin{figure}[H]
\centering
\begin{tikzpicture}
 \coordinate[label=below left:O] (A1) at (0,0);
 \coordinate [label=below:1] (A2) at (0.5,0);
 \coordinate [label=below:2] (A3) at (1,0);
 \coordinate [label=below:3] (A4) at (1.5,0);
 \coordinate [label=below:4] (A5) at (2,0);
 \coordinate [label=left:1] (A6) at (0,0.5);
 \coordinate (A7) at (0.5,0.5);
 \coordinate (A8) at (1,0.5);
 \coordinate (A9) at (1.5,0.5);
 \coordinate [label=left:2] (A10) at (0,1);
 \coordinate (A11) at (0.5,1);
 \coordinate (A12) at (1,1);
 \coordinate [label=left:3] (A13) at (0,1.5);
 \coordinate (A14) at (0.5,1.5);
 \coordinate [label=left:4] (A15) at (0,2);
 \draw [very thin, gray] (A1)--(A5)--(A15)--cycle; 
 \foreach \P in {1,6,10,13,15} \draw[dashed, blue] (A7)--(A\P);
 \foreach \P in {1,2,3,4,5} \draw[thick, dotted, red] (A7)--(A\P);
 \foreach \t in {1,2,...,15} \fill[black] (A\t) circle (0.05);
 \coordinate[label=above right:O] (O1) at (5.6,2);
 \coordinate (Y1) at (3.6,2);
 \coordinate (Z1) at (5.6,0);
 \coordinate (Y'1) at (4.6,2);
 \coordinate (Z'1) at (5.6,1);
 \draw[thin,->,>=stealth] (O1)--(Y1) node[left] {$X$};
 \draw[thin,->,>=stealth] (O1)--(Z1) node[below] {$Y$};
 \fill[black] (Y'1) circle (0.06);
 \fill[black] (Z'1) circle (0.06);
 \draw (Y'1) to [out=280, in=170] (Z'1);
 \draw (7.2,1.5)--(8.8,1.5)--(8,0.2)--cycle;
 \draw (7.2,1.5)--(8,1)--(8.8,1.5);
 \draw (8,1)--(8,0.2); 
 \coordinate [label=below:(a)] (a1) at (1,-0.5);
 \coordinate [label=below:(b)] (a2) at (4.6,-0.5);
 \coordinate [label=below:(c)] (a3) at (8,-0.5);
\end{tikzpicture}
\caption{Case $1$.}
\label{fc1}
\end{figure}
Thus none of the supports of the complexes $K_{XY}, \dots, K_{YW}$ and $K_{ZW}$ contain a lattice point in its interior, i.e., $b_1(\overline{C'})=0$.
Hence the skeleton of $C$ is not the lollipop graph by Lemma \ref{saikurunashi}.
In fact, it is as in Figure \ref{fc1} (c).

\subsection{Case $2$: $(a_0, b_0, c_0, d_0)=(2, 1, 1, 0)$}  

Since $(b_0, c_0)=(1, 1)$, the skeleton of $C_{YZ}$ is as in Case $1$.
The boundaries of the complexes $K_{YW}, K_{ZW}$ are obtained by connecting any of the dashed lines in Figure \ref{fc2} (a), which correspond to the intersections $C\cap Y$ and $C\cap Z$. 
Thus the skeletons of $C_{YZ}$, $C_{YW}$ and $C_{ZW}$ are as in (b).
Note that they do not contain cycles.
\begin{figure}[H]
\centering
\begin{tikzpicture}
 \coordinate[label=below left:O] (A1) at (0,-3);
 \coordinate [label=below:1] (A2) at (0.5,-3);
 \coordinate [label=below:2] (A3) at (1,-3);
 \coordinate [label=below:3] (A4) at (1.5,-3);
 \coordinate [label=below:4] (A5) at (2,-3);
 \coordinate [label=left:1] (A6) at (0,-2.5);
 \coordinate (A7) at (0.5,-2.5);
 \coordinate (A8) at (1,-2.5);
 \coordinate (A9) at (1.5,-2.5);
 \coordinate [label=left:2] (A10) at (0,-2);
 \coordinate (A11) at (0.5,-2);
 \coordinate (A12) at (1,-2);
 \coordinate [label=left:3] (A13) at (0,-1.5);
 \coordinate (A14) at (0.5,-1.5);
 \coordinate [label=left:4] (A15) at (0,-1);
 \draw [very thin, gray] (A1)--(A5)--(A15)--cycle; 
 \foreach \t in {1,2,...,15} \fill[black] (A\t) circle (0.05);
 \foreach \P in {1,6,10,13,15} \draw[dashed] (A2)--(A\P);
 \coordinate[label=above right:O] (O1) at (5.6,-1.5);
 \coordinate (Y1) at (4.2,-1.5);
 \coordinate (Z1) at (5.6,-2.9);
 \coordinate (Y'1) at (4.9,-1.5);
 \coordinate (Z'1) at (5.6,-2.2);
 \draw[thin,->,>=stealth] (O1)--(Y1) node[left] {$Y$};
 \draw[thin,->,>=stealth] (O1)--(Z1) node[below] {$Z$};
 \fill[black] (Y'1) circle (0.06);
 \fill[black] (Z'1) circle (0.06);
 \draw (Y'1) to [out=280, in=170] (Z'1);
 \coordinate[label=below:O] (O4) at (8,-2.6);
 \coordinate (Y4) at (6.6,-2.6);
 \coordinate (W4) at (8.8,-1.5);
 \coordinate (Y'4) at (7.3,-2.6);
 \draw[thin,->,>=stealth] (O4)--(Y4) node[left] {$Y$};
 \draw[thin,->,>=stealth] (O4)--(W4) node[above right] {$W$};
 \fill[black] (Y'4) circle (0.06);
 \coordinate[label=below:O] (O5) at (11.2,-2.6);
 \coordinate (Z5) at (9.8,-2.6);
 \coordinate (W5) at (12,-1.5);
 \coordinate (Z'5) at (10.5,-2.6);
 \draw[thin,->,>=stealth] (O5)--(Z5) node[left] {$Z$};
 \draw[thin,->,>=stealth] (O5)--(W5) node[above right] {$W$};
 \fill[black] (Z'5) circle (0.06);
 \coordinate [label=below:(a)] (a1) at (1,-3.5);
 \coordinate [label=below:(b)] (a2) at (7.5,-3.5);
\end{tikzpicture}
\caption{Case $2$.}
\label{fc2}
\end{figure}

We look at $\MM_X(g^h)$.
It is sufficient to consider the $6$ cases below by the symmetry between $Y$ and $Z$.
\begin{eqnarray*}
\text{Case $2.1$: }(2, 1, 1, 0)\rightarrow (1, 3, 0, 0), \quad \text{Case $2.2$: }(2, 1, 1, 0)\rightarrow (1, 2, 1, 0), \\
\text{Case $2.3$: }(2, 1, 1, 0)\rightarrow (1, 2, 0, 1), \quad \text{Case $2.4$: }(2, 1, 1, 0)\rightarrow (1, 1, 1, 1), \\
\text{Case $2.5$: }(2, 1, 1, 0)\rightarrow (1, 1, 0, 2), \quad \text{Case $2.6$: }(2, 1, 1, 0)\rightarrow (1, 0, 0, 3).\hspace{0.5mm}
\end{eqnarray*}

\begin{Lem}\label{x1y1}
Let $P=(a, b, c, d)\in \MM_{X}(g^h)$, $a>0$.
Assume that $|K_{XY}|\setminus \{P\}$ is connected.
If $a=a_0$ and $b_0>0$, the points $x_1$ and $y_1$ are connected by a path in $C_{XY}$.
If $a<a_0$, the points $x_{a_0-a}$ and $x_{a_0-a+1}$ are connected by a path in $C_{XY}$.
\end{Lem}

\begin{proof}
Assume that $|K_{XY}|\setminus \{P\}$ is connected, $a=a_0$ and $b_0>0$.
Let $D_P$ be the closure of the domain corresponding to $P$, $E_{x_1}$ the edge of $C_{XY}$ containing $x_1$, and $E_{y_1}$ the edge of $C_{XY}$ containing $y_1$.
Since $a=a_0>0$ and $b=b_0>0$, there is a $t<0$ such that for all $(x, y)\in D_P$, we have $t<x$ and $t<y$.
Thus, $D_P\cap XY$ is bounded.
Then the intersection of $XY$ and the boundary of $D_P$ is a union of line segments since $D_P$ is a convex polyhedral set.
$E_{x_1}$ and $E_{y_1}$ are contained in the above line segments since the corresponding $1$-simplices contain $P$.
Therefore $x_1$ and $y_1$ are connected by a path in $C_{XY}$.

Next, assume that $|K_{XY}|\setminus \{P\}$ is connected and $0<a<a_0$.
Note that since $|K_{XY}|\setminus \{P\}$ is connected and $a>0$, we have $b>0$.
Let $D_P$ be the closure of the domain corresponding to $P$, $E_{a_0-a}$ the edge of $C_{XY}$ containing $x_{a_0-a}$, and $E_{a_0-a+1}$ the edge of $C_{XY}$ containing $x_{a_0-a+1}$.
Since $0<a$ and $0<b$, there is a $t_1<0$ such that for all $(x, y)\in D_P$, we have $t_1<x$ and $t_1<y$.
Assume $D_P\cap Y\neq \emptyset$ and let $y\in D_P\cap Y$.
Then the convex set $D_P$ contains the line segment $y$-$x_{a_0-a}$ and this contradicts the assumption that $|K_{XY}|\setminus \{P\}$ is connected.
Thus, $D_P\cap Y=\emptyset$, and hence there is a $t_2<0$ such that for all $(x, y)\in D_P$, we have $x<t_2$.
Then the intersection of $XY$ and the boundary of $D_P$ is a union of line segments since $D_P$ is a convex polyhedral set.
$E_{a_0-a}$ and $E_{a_0-a+1}$ are contained in the above line segments since the corresponding $1$-simplices contain $P$.
Therefore $x_{a_0-a}$ and $x_{a_0-a+1}$ are connected by a path in $C_{XY}$.
\end{proof}

\begin{Lem}\label{case2hodai}
If $(a_0, b_0, c_0, d_0)=(2, 1, 1, 0)$ and $b_1(\overline{C'})=1$, the skeleton of $C$ is not the lollipop graph.
\end{Lem}

\begin{proof}
Assume $b_1(\overline{C'})=1$.
Then $C'=C\setminus (X\cup Y\cup Z\cup W)$ has exactly $1$ cycle.
We have $C\cap X=\{x_1, x_2\}$, $C\cap Y=\{y\}$ and $C\cap Z=\{z\}$ (see Notation \ref{CcapX}).
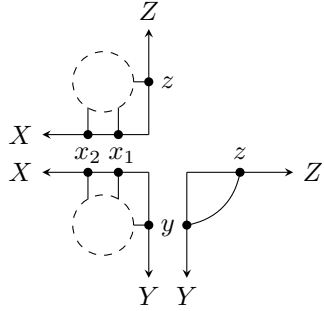
\begin{figure}[H]
\centering
\begin{tikzpicture}
 \coordinate (O1) at (5.6,-1.5);
 \coordinate (Y1) at (4.2,-1.5);
 \coordinate (Z1) at (5.6,-2.9);
 \coordinate (Y'1) at (5.2,-1.5);
 \coordinate (Y2) at (4.8,-1.5);
 \coordinate (Z'1) at (5.6,-2.2);
 \draw[thin,->,>=stealth] (O1)--(Y1) node[left] {$X$};
 \draw[thin,->,>=stealth] (O1)--(Z1) node[below] {$Y$};
 \draw[dashed] (5,-2.2) circle (0.4);
 \draw (Z'1)--(5.4,-2.2);
 \draw (Y2)--(4.8,-1.86);
 \draw (Y'1)--(5.2,-1.86);
 \fill[black] (Y'1) circle (0.06);
 \fill[black] (Z'1) circle (0.06);
 \fill[black] (Y2) circle (0.06);
 \coordinate (O4) at (5.6,-1);
 \coordinate (Y4) at (4.2,-1);
 \coordinate (Y5) at (4.8,-1);
 \coordinate (W5) at (5.6,-0.3);
 \coordinate (W4) at (5.6,0.4);
 \coordinate (Y'4) at (5.2,-1);
 \draw[thin,->,>=stealth] (O4)--(Y4) node[left] {$X$};
 \draw[thin,->,>=stealth] (O4)--(W4) node[above] {$Z$};
 \draw[dashed] (5,-0.3) circle (0.4); 
 \draw (W5)--(5.4,-0.3);
 \draw (Y5)--(4.8,-0.64);
 \draw (Y'4)--(5.2,-0.64);
 \fill[black] (Y'4) circle (0.06);
 \fill[black] (Y5) circle (0.06);
 \fill[black] (W5) circle (0.06);
 \coordinate (O5) at (6.1,-1.5);
 \coordinate (Z5) at (7.5,-1.5);
 \coordinate (W5) at (6.1,-2.9);
 \coordinate (W6) at (6.1,-2.2);
 \coordinate (Z'5) at (6.8,-1.5);
 \draw[thin,->,>=stealth] (O5)--(Z5) node[right] {$Z$};
 \draw[thin,->,>=stealth] (O5)--(W5) node[below] {$Y$};
 \draw (W6) to [out=10, in=260] (Z'5);
 \fill[black] (Z'5) circle (0.06);
 \fill[black] (W6) circle (0.06);
 \coordinate [label=below:$x_1$] (a1) at (5.25,-1.05);
 \coordinate [label=below:$x_2$] (a2) at (4.8,-1.05);
 \coordinate [label=below:$y$] (a1) at (5.85,-2);
 \coordinate [label=below:$z$] (a2) at (6.8,-1.05);
 \coordinate [label=below:$z$] (a1) at (5.85,-0.1);
\end{tikzpicture}
\caption{The skeletons discussed in the proof of Lem \ref{case2hodai}.}
\label{proofoflemma5.18nozu}
\end{figure}
Note that the point $x_1$ is connected with $y$ and $z$ in $XY$ and $XZ$, respectively, by Lemma \ref{x1y1}.
When we connect $C'$ at $y$ and $z$, the number of cycles does not change (see Figure \ref{proofoflemma5.18nozu}).
Therefore, glueing at $X$ adds $2$ homologically independent cycles.
First, when we glue at $x_1$, we have at least $1$ cycle passing through the points $x_1$, $y$ and $z$.
Note that if glueing at $x_1$ adds $2$ cycles, these cycles share the point $x_1$.
Thus we may assume that glueing at $x_1$ adds $1$ cycle.
Next, glueing at $x_2$ adds at least $1$ cycle, and this cycle passes through at least one of the points $x_1$, $y$ and $z$.
Thus, two cycles share one point.
Therefore, the skeleton of $C$ is not the lollipop graph.
\end{proof}

\subsubsection{Case $2. 1$: $(2, 1, 1, 0)\rightarrow (1, 3, 0, 0)$}

The complexes $K_{XY}$, $K_{XZ}$ and $K_{XW}$ are as in Figure \ref{fc2.1} (a).
Assume that the skeleton of $C$ is the lollipop graph of genus $3$.
Then by Lemma \ref{12}, $K_{XY}$ contains the $2$-simplices as in (b).

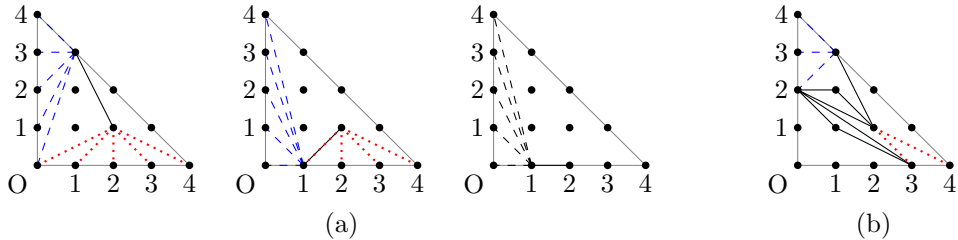
\begin{figure}[H]
\centering
\begin{tikzpicture}
 \coordinate[label=below left:O] (A1) at (0,0);
 \coordinate [label=below:1] (A2) at (0.5,0);
 \coordinate [label=below:2] (A3) at (1,0);
 \coordinate [label=below:3] (A4) at (1.5,0);
 \coordinate [label=below:4] (A5) at (2,0);
 \coordinate [label=left:1] (A6) at (0,0.5);
 \coordinate (A7) at (0.5,0.5);
 \coordinate (A8) at (1,0.5);
 \coordinate (A9) at (1.5,0.5);
 \coordinate [label=left:2] (A10) at (0,1);
 \coordinate (A11) at (0.5,1);
 \coordinate (A12) at (1,1);
 \coordinate [label=left:3] (A13) at (0,1.5);
 \coordinate (A14) at (0.5,1.5);
 \coordinate [label=left:4] (A15) at (0,2);
 \draw [very thin, gray] (A1)--(A5)--(A15)--cycle; 
 \draw (A8)--(A14);
 \foreach \P in {1,6,10,13,15}  \draw[dashed, blue] (A14)--(A\P);
 \foreach \P in {1,2,3,4,5}  \draw[thick, dotted, red] (A8)--(A\P);
 \foreach \t in {1,2,...,15} \fill[black] (A\t) circle (0.05);
 \coordinate[label=below left:O] (B1) at (3,0);
 \coordinate [label=below:1] (B2) at (3.5,0);
 \coordinate [label=below:2] (B3) at (4,0);
 \coordinate [label=below:3] (B4) at (4.5,0);
 \coordinate [label=below:4] (B5) at (5,0);
 \coordinate [label=left:1] (B6) at (3,0.5);
 \coordinate (B7) at (3.5,0.5);
 \coordinate (B8) at (4,0.5);
 \coordinate (B9) at (4.5,0.5);
 \coordinate [label=left:2] (B10) at (3,1);
 \coordinate (B11) at (3.5,1);
 \coordinate (B12) at (4,1);
 \coordinate [label=left:3] (B13) at (3,1.5);
 \coordinate (B14) at (3.5,1.5);
 \coordinate [label=left:4] (B15) at (3,2);
 \draw [very thin, gray] (B1)--(B5)--(B15)--cycle; 
 \draw (B8)--(B2);
 \foreach \P in {1,6,10,13,15}  \draw[dashed, blue] (B2)--(B\P);
 \foreach \P in {2,3,4,5}  \draw[thick, dotted, red] (B8)--(B\P);
 \foreach \t in {1,2,...,15} \fill[black] (B\t) circle (0.05);
 \coordinate[label=below left:O] (C1) at (6,0);
 \coordinate [label=below:1] (C2) at (6.5,0);
 \coordinate [label=below:2] (C3) at (7,0);
 \coordinate [label=below:3] (C4) at (7.5,0);
 \coordinate [label=below:4] (C5) at (8,0);
 \coordinate [label=left:1] (C6) at (6,0.5);
 \coordinate (C7) at (6.5,0.5);
 \coordinate (C8) at (7,0.5);
 \coordinate (C9) at (7.5,0.5);
 \coordinate [label=left:2] (C10) at (6,1);
 \coordinate (C11) at (6.5,1);
 \coordinate (C12) at (7,1);
 \coordinate [label=left:3] (C13) at (6,1.5);
 \coordinate (C14) at (6.5,1.5);
 \coordinate [label=left:4] (C15) at (6,2);
 \draw [very thin, gray] (C1)--(C5)--(C15)--cycle; 
 \draw (C3)--(C2);
 \foreach \t in {1,2,...,15} \fill[black] (C\t) circle (0.05);
 \foreach \P in {1,6,10,13,15}  \draw[dashed] (C2)--(C\P);
 \coordinate[label=below left:O] (D1) at (10,0);
 \coordinate [label=below:1] (D2) at (10.5,0);
 \coordinate [label=below:2] (D3) at (11,0);
 \coordinate [label=below:3] (D4) at (11.5,0);
 \coordinate [label=below:4] (D5) at (12,0);
 \coordinate [label=left:1] (D6) at (10,0.5);
 \coordinate (D7) at (10.5,0.5);
 \coordinate (D8) at (11,0.5);
 \coordinate (D9) at (11.5,0.5);
 \coordinate [label=left:2] (D10) at (10,1);
 \coordinate (D11) at (10.5,1);
 \coordinate (D12) at (11,1);
 \coordinate [label=left:3] (D13) at (10,1.5);
 \coordinate (D14) at (10.5,1.5);
 \coordinate [label=left:4] (D15) at (10,2);
 \draw [very thin, gray] (D1)--(D5)--(D15)--cycle;
 \draw (D8)--(D14);
 \draw (D4)--(D10)--(D8);
 \draw (D4)--(D7)--(D10);
 \draw (D10)--(D11)--(D8);
 \foreach \P in {10,13,15}  \draw[dashed, blue] (D14)--(D\P);
 \foreach \P in {4,5}  \draw[thick, dotted, red] (D8)--(D\P);
 \foreach \t in {1,2,...,15} \fill[black] (D\t) circle (0.05);
 \coordinate [label=below:(a)] (a1) at (4,-0.5);
 \coordinate [label=below:(b)] (a2) at (11,-0.5);
\end{tikzpicture}
\caption{The complexes in Case $2. 1$.}
\label{fc2.1}
\end{figure}

Hence the skeletons of $C_{XY}$, $C_{XZ}$ and $C_{YZ}$ are as in Figure \ref{fc2.1s}.
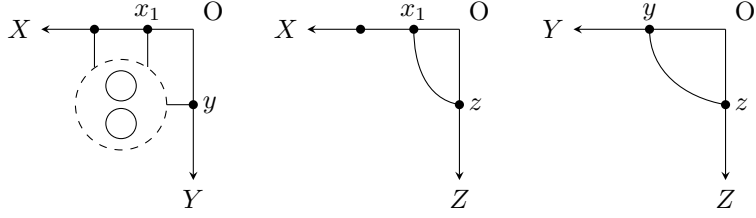
\begin{figure}[H]
\centering
\begin{tikzpicture}
 \coordinate[label=above right:O] (O1) at (-4,-1);
 \coordinate (X1) at (-6,-1);
 \coordinate (Y1) at (-4,-3);
 \coordinate (X2) at (-5.3,-1);
 \coordinate (X3) at (-4.6,-1);
 \coordinate (Y2) at (-4,-2);
 \draw[thin,->,>=stealth] (O1)--(X1) node[left] {$X$};
 \draw[thin,->,>=stealth] (O1)--(Y1) node[below] {$Y$};
 \fill[black] (X2) circle (0.06);
 \fill[black] (X3) circle (0.06);
 \fill[black] (Y2) circle (0.06);
 \draw (X2)--(-5.3,-1.5);
 \draw (X3)--(-4.6,-1.5);
 \draw (Y2)--(-4.35,-2);
 \draw[dashed] (-4.95,-2) circle (0.6);
 \draw (-4.95,-1.75) circle (0.2);
 \draw (-4.95,-2.25) circle (0.2);
 \coordinate[label=above right:O] (O1) at (-0.5,-1);
 \coordinate (X1) at (-2.5,-1);
 \coordinate (Y1) at (-0.5,-3);
 \coordinate (X2) at (-1.8,-1);
 \coordinate (X3) at (-1.1,-1);
 \coordinate (Y2) at (-0.5,-2);
 \draw[thin,->,>=stealth] (O1)--(X1) node[left] {$X$};
 \draw[thin,->,>=stealth] (O1)--(Y1) node[below] {$Z$};
 \fill[black] (X2) circle (0.06);
 \fill[black] (X3) circle (0.06);
 \fill[black] (Y2) circle (0.06);
 \draw (X3) to [out=270, in=170] (Y2);
 \coordinate[label=above right:O] (O2) at (3,-1);
 \coordinate (X4) at (1,-1);
 \coordinate (Z1) at (3,-3);
 \coordinate (X5) at (2,-1);
 \coordinate (Z2) at (3,-2);
 \draw[thin,->,>=stealth] (O2)--(X4) node[left] {$Y$};
 \draw[thin,->,>=stealth] (O2)--(Z1) node[below] {$Z$};
 \fill[black] (X5) circle (0.06);
 \fill[black] (Z2) circle (0.06);
 \draw (X5) to [out=270, in=170] (Z2);
 \coordinate [label=above:$x_1$] (u1) at (-4.6,-1);
 \coordinate [label=above:$x_1$] (u1) at (-1.1,-1);
 \coordinate [label=above:$y$] (u1) at (2,-1);
 \coordinate [label=right:$y$] (u1) at (-4,-2);
 \coordinate [label=right:$z$] (u1) at (-0.5,-2);
 \coordinate [label=right:$z$] (u1) at (3,-2); 
\end{tikzpicture}
\caption{The skeletons of $C_{XY}$, $C_{XZ}$ and $C_{YZ}$ in Case $2. 1$.}
\label{fc2.1s}
\end{figure}

The $1$-simplex $(1, 3)$-$(2, 1)$ of $K_{XY}$ corresponds to the edge of $C_{XY}$ which contains the point $x_1\in C\cap X$.
Any path from $x_1$ to $y$ in $C_{XY}$ intersects the cycle which corresponds to the point $(1, 2)$.
When we glue at $X$, $Y$ and $Z$, we have a cycle which contains the points $x_1$ and $y$.
Thus the two cycles intersect and the skeleton of $C$ is not the lollipop graph.

\subsubsection{Cases $2. 2$ to $2. 6$}
The complexes $K_{XY}$, $K_{XZ}$ and $K_{XW}$ are as in Figure \ref{fc2.2to2.6}.
In Cases $2. 2$ to $2. 5$, $b_1(\overline{C'})$ is at most $1$ and the skeleton of $C$ is not the lollipop graph by Lemma \ref{saikurunashi} and \ref{case2hodai}.
In Case $2. 6$, $K_{XW}$ contains the edge $(1, 3)$-$(2, 0)$ as its $1$-simplex and the skeleton of $C$ is not the lollipop graph by Lemma \ref{Case2.6}.
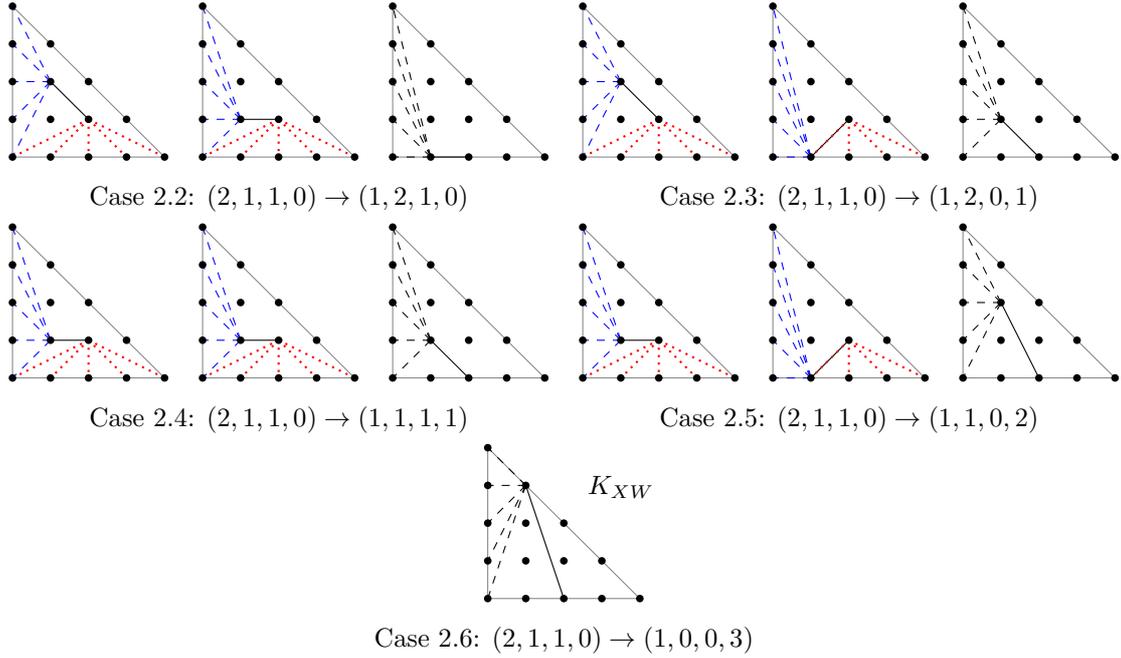
\begin{figure}[H]
\centering
\begin{tikzpicture}
 \coordinate (A1) at (0,0);
 \coordinate (A2) at (0.5,0);
 \coordinate (A3) at (1,0);
 \coordinate (A4) at (1.5,0);
 \coordinate (A5) at (2,0);
 \coordinate (A6) at (0,0.5);
 \coordinate (A7) at (0.5,0.5);
 \coordinate (A8) at (1,0.5);
 \coordinate (A9) at (1.5,0.5);
 \coordinate (A10) at (0,1);
 \coordinate (A11) at (0.5,1);
 \coordinate (A12) at (1,1);
 \coordinate (A13) at (0,1.5);
 \coordinate (A14) at (0.5,1.5);
 \coordinate (A15) at (0,2);
 \draw [very thin, gray] (A1)--(A5)--(A15)--cycle; 
 \draw (A8)--(A11);
 \foreach \P in {1,6,10,13,15}  \draw[dashed, blue] (A11)--(A\P);
 \foreach \P in {1,2,3,4,5}  \draw[thick, dotted, red] (A8)--(A\P);
 \foreach \t in {1,2,...,15} \fill[black] (A\t) circle (0.05);
 \coordinate (B1) at (2.5,0);
 \coordinate (B2) at (3,0);
 \coordinate (B3) at (3.5,0);
 \coordinate (B4) at (4,0);
 \coordinate (B5) at (4.5,0);
 \coordinate (B6) at (2.5,0.5);
 \coordinate (B7) at (3,0.5);
 \coordinate (B8) at (3.5,0.5);
 \coordinate (B9) at (4,0.5);
 \coordinate (B10) at (2.5,1);
 \coordinate (B11) at (3,1);
 \coordinate (B12) at (3.5,1);
 \coordinate (B13) at (2.5,1.5);
 \coordinate (B14) at (3,1.5);
 \coordinate (B15) at (2.5,2);
 \draw [very thin, gray] (B1)--(B5)--(B15)--cycle; 
 \draw (B8)--(B7);
 \foreach \P in {1,6,10,13,15}  \draw[dashed, blue] (B7)--(B\P);
 \foreach \P in {1,2,3,4,5}  \draw[thick, dotted, red] (B8)--(B\P);
 \foreach \t in {1,2,...,15} \fill[black] (B\t) circle (0.05);
 \coordinate (C1) at (5,0);
 \coordinate (C2) at (5.5,0);
 \coordinate (C3) at (6,0);
 \coordinate (C4) at (6.5,0);
 \coordinate (C5) at (7,0);
 \coordinate (C6) at (5,0.5);
 \coordinate (C7) at (5.5,0.5);
 \coordinate (C8) at (6,0.5);
 \coordinate (C9) at (6.5,0.5);
 \coordinate (C10) at (5,1);
 \coordinate (C11) at (5.5,1);
 \coordinate (C12) at (6,1);
 \coordinate (C13) at (5,1.5);
 \coordinate (C14) at (5.5,1.5);
 \coordinate (C15) at (5,2);
 \draw [very thin, gray] (C1)--(C5)--(C15)--cycle; 
 \draw (C3)--(C2);
 \foreach \t in {1,2,...,15} \fill[black] (C\t) circle (0.05);
 \foreach \P in {1,6,10,13,15}  \draw[dashed] (C2)--(C\P);
 \coordinate (A1) at (7.5,0);
 \coordinate (A2) at (8,0);
 \coordinate (A3) at (8.5,0);
 \coordinate (A4) at (9,0);
 \coordinate (A5) at (9.5,0);
 \coordinate (A6) at (7.5,0.5);
 \coordinate (A7) at (8,0.5);
 \coordinate (A8) at (8.5,0.5);
 \coordinate (A9) at (9,0.5);
 \coordinate (A10) at (7.5,1);
 \coordinate (A11) at (8,1);
 \coordinate (A12) at (8.5,1);
 \coordinate (A13) at (7.5,1.5);
 \coordinate (A14) at (8,1.5);
 \coordinate (A15) at (7.5,2);
 \draw [very thin, gray] (A1)--(A5)--(A15)--cycle; 
 \draw (A8)--(A11);
 \foreach \P in {1,6,10,13,15}  \draw[dashed, blue] (A11)--(A\P);
 \foreach \P in {1,2,3,4,5}  \draw[thick, dotted, red] (A8)--(A\P);
 \foreach \t in {1,2,...,15} \fill[black] (A\t) circle (0.05);
 \coordinate (B1) at (10,0);
 \coordinate (B2) at (10.5,0);
 \coordinate (B3) at (11,0);
 \coordinate (B4) at (11.5,0);
 \coordinate (B5) at (12,0);
 \coordinate (B6) at (10,0.5);
 \coordinate (B7) at (10.5,0.5);
 \coordinate (B8) at (11,0.5);
 \coordinate (B9) at (11.5,0.5);
 \coordinate (B10) at (10,1);
 \coordinate (B11) at (10.5,1);
 \coordinate (B12) at (11,1);
 \coordinate (B13) at (10,1.5);
 \coordinate (B14) at (10.5,1.5);
 \coordinate (B15) at (10,2);
 \draw [very thin, gray] (B1)--(B5)--(B15)--cycle; 
 \draw (B8)--(B2);
 \foreach \P in {1,6,10,13,15}  \draw[dashed, blue] (B2)--(B\P);
 \foreach \P in {2,3,4,5}  \draw[thick, dotted, red] (B8)--(B\P);
 \foreach \t in {1,2,...,15} \fill[black] (B\t) circle (0.05);
 \coordinate (C1) at (12.5,0);
 \coordinate (C2) at (13,0);
 \coordinate (C3) at (13.5,0);
 \coordinate (C4) at (14,0);
 \coordinate (C5) at (14.5,0);
 \coordinate (C6) at (12.5,0.5);
 \coordinate (C7) at (13,0.5);
 \coordinate (C8) at (13.5,0.5);
 \coordinate (C9) at (14,0.5);
 \coordinate (C10) at (12.5,1);
 \coordinate (C11) at (13,1);
 \coordinate (C12) at (13.5,1);
 \coordinate (C13) at (12.5,1.5);
 \coordinate (C14) at (13,1.5);
 \coordinate (C15) at (12.5,2);
 \draw [very thin, gray] (C1)--(C5)--(C15)--cycle; 
 \draw (C3)--(C7);
 \foreach \t in {1,2,...,15} \fill[black] (C\t) circle (0.05);
 \foreach \P in {1,6,10,13,15}  \draw[dashed] (C7)--(C\P);
 \coordinate [label=below:$\text{Case $2. 2$: $(2, 1, 1, 0)\rightarrow (1, 2, 1, 0)$}$] (a1) at (3.5,-0.25);
 \coordinate [label=below:$\text{Case $2. 3$: $(2, 1, 1, 0)\rightarrow (1, 2, 0, 1)$}$] (a2) at (11,-0.25);
\end{tikzpicture}
\begin{tikzpicture}
 \coordinate (A1) at (0,0);
 \coordinate (A2) at (0.5,0);
 \coordinate (A3) at (1,0);
 \coordinate (A4) at (1.5,0);
 \coordinate (A5) at (2,0);
 \coordinate (A6) at (0,0.5);
 \coordinate (A7) at (0.5,0.5);
 \coordinate (A8) at (1,0.5);
 \coordinate (A9) at (1.5,0.5);
 \coordinate (A10) at (0,1);
 \coordinate (A11) at (0.5,1);
 \coordinate (A12) at (1,1);
 \coordinate (A13) at (0,1.5);
 \coordinate (A14) at (0.5,1.5);
 \coordinate (A15) at (0,2);
 \draw [very thin, gray] (A1)--(A5)--(A15)--cycle; 
 \draw (A8)--(A7);
 \foreach \P in {1,6,10,13,15}  \draw[dashed, blue] (A7)--(A\P);
 \foreach \P in {1,2,3,4,5}  \draw[thick, dotted, red] (A8)--(A\P);
 \foreach \t in {1,2,...,15} \fill[black] (A\t) circle (0.05);
 \coordinate (B1) at (2.5,0);
 \coordinate (B2) at (3,0);
 \coordinate (B3) at (3.5,0);
 \coordinate (B4) at (4,0);
 \coordinate (B5) at (4.5,0);
 \coordinate (B6) at (2.5,0.5);
 \coordinate (B7) at (3,0.5);
 \coordinate (B8) at (3.5,0.5);
 \coordinate (B9) at (4,0.5);
 \coordinate (B10) at (2.5,1);
 \coordinate (B11) at (3,1);
 \coordinate (B12) at (3.5,1);
 \coordinate (B13) at (2.5,1.5);
 \coordinate (B14) at (3,1.5);
 \coordinate (B15) at (2.5,2);
 \draw [very thin, gray] (B1)--(B5)--(B15)--cycle; 
 \draw (B8)--(B7);
 \foreach \P in {1,6,10,13,15}  \draw[dashed, blue] (B7)--(B\P);
 \foreach \P in {1,2,3,4,5}  \draw[thick, dotted, red] (B8)--(B\P);
 \foreach \t in {1,2,...,15} \fill[black] (B\t) circle (0.05);
 \coordinate (C1) at (5,0);
 \coordinate (C2) at (5.5,0);
 \coordinate (C3) at (6,0);
 \coordinate (C4) at (6.5,0);
 \coordinate (C5) at (7,0);
 \coordinate (C6) at (5,0.5);
 \coordinate (C7) at (5.5,0.5);
 \coordinate (C8) at (6,0.5);
 \coordinate (C9) at (6.5,0.5);
 \coordinate (C10) at (5,1);
 \coordinate (C11) at (5.5,1);
 \coordinate (C12) at (6,1);
 \coordinate (C13) at (5,1.5);
 \coordinate (C14) at (5.5,1.5);
 \coordinate (C15) at (5,2);
 \draw [very thin, gray] (C1)--(C5)--(C15)--cycle; 
 \draw (C3)--(C7);
 \foreach \t in {1,2,...,15} \fill[black] (C\t) circle (0.05);
 \foreach \P in {1,6,10,13,15}  \draw[dashed] (C7)--(C\P);
 \coordinate (A1) at (7.5,0);
 \coordinate (A2) at (8,0);
 \coordinate (A3) at (8.5,0);
 \coordinate (A4) at (9,0);
 \coordinate (A5) at (9.5,0);
 \coordinate (A6) at (7.5,0.5);
 \coordinate (A7) at (8,0.5);
 \coordinate (A8) at (8.5,0.5);
 \coordinate (A9) at (9,0.5);
 \coordinate (A10) at (7.5,1);
 \coordinate (A11) at (8,1);
 \coordinate (A12) at (8.5,1);
 \coordinate (A13) at (7.5,1.5);
 \coordinate (A14) at (8,1.5);
 \coordinate (A15) at (7.5,2);
 \draw [very thin, gray] (A1)--(A5)--(A15)--cycle; 
 \draw (A8)--(A7);
 \foreach \P in {1,6,10,13,15}  \draw[dashed, blue] (A7)--(A\P);
 \foreach \P in {1,2,3,4,5}  \draw[thick, dotted, red] (A8)--(A\P);
 \foreach \t in {1,2,...,15} \fill[black] (A\t) circle (0.05);
 \coordinate (B1) at (10,0);
 \coordinate (B2) at (10.5,0);
 \coordinate (B3) at (11,0);
 \coordinate (B4) at (11.5,0);
 \coordinate (B5) at (12,0);
 \coordinate (B6) at (10,0.5);
 \coordinate (B7) at (10.5,0.5);
 \coordinate (B8) at (11,0.5);
 \coordinate (B9) at (11.5,0.5);
 \coordinate (B10) at (10,1);
 \coordinate (B11) at (10.5,1);
 \coordinate (B12) at (11,1);
 \coordinate (B13) at (10,1.5);
 \coordinate (B14) at (10.5,1.5);
 \coordinate (B15) at (10,2);
 \draw [very thin, gray] (B1)--(B5)--(B15)--cycle; 
 \draw (B8)--(B2);
 \foreach \P in {1,6,10,13,15}  \draw[dashed, blue] (B2)--(B\P);
 \foreach \P in {2,3,4,5}  \draw[thick, dotted, red] (B8)--(B\P);
 \foreach \t in {1,2,...,15} \fill[black] (B\t) circle (0.05);
 \coordinate (C1) at (12.5,0);
 \coordinate (C2) at (13,0);
 \coordinate (C3) at (13.5,0);
 \coordinate (C4) at (14,0);
 \coordinate (C5) at (14.5,0);
 \coordinate (C6) at (12.5,0.5);
 \coordinate (C7) at (13,0.5);
 \coordinate (C8) at (13.5,0.5);
 \coordinate (C9) at (14,0.5);
 \coordinate (C10) at (12.5,1);
 \coordinate (C11) at (13,1);
 \coordinate (C12) at (13.5,1);
 \coordinate (C13) at (12.5,1.5);
 \coordinate (C14) at (13,1.5);
 \coordinate (C15) at (12.5,2);
 \draw [very thin, gray] (C1)--(C5)--(C15)--cycle; 
 \draw (C3)--(C11);
 \foreach \t in {1,2,...,15} \fill[black] (C\t) circle (0.05);
 \foreach \P in {1,6,10,13,15}  \draw[dashed] (C11)--(C\P);
 \coordinate [label=below:$\text{Case $2. 4$: $(2, 1, 1, 0)\rightarrow (1, 1, 1, 1)$}$] (a1) at (3.5,-0.25);
 \coordinate [label=below:$\text{Case $2. 5$: $(2, 1, 1, 0)\rightarrow (1, 1, 0, 2)$}$] (a2) at (11,-0.25);
\end{tikzpicture}
\begin{tikzpicture}
 \coordinate (B1) at (2.5,0);
 \coordinate (B2) at (3,0);
 \coordinate (B3) at (3.5,0);
 \coordinate (B4) at (4,0);
 \coordinate (B5) at (4.5,0);
 \coordinate (B6) at (2.5,0.5);
 \coordinate (B7) at (3,0.5);
 \coordinate (B8) at (3.5,0.5);
 \coordinate (B9) at (4,0.5);
 \coordinate (B10) at (2.5,1);
 \coordinate (B11) at (3,1);
 \coordinate (B12) at (3.5,1);
 \coordinate (B13) at (2.5,1.5);
 \coordinate (B14) at (3,1.5);
 \coordinate (B15) at (2.5,2);
 \draw [very thin, gray] (B1)--(B5)--(B15)--cycle; 
 \draw (B3)--(B14);
 \foreach \P in {1,6,10,13,15}  \draw[dashed] (B14)--(B\P);
 \foreach \t in {1,2,...,15} \fill[black] (B\t) circle (0.05);
 \coordinate [label=below:$\text{Case $2. 6$: $(2, 1, 1, 0)\rightarrow (1, 0, 0, 3)$}$] (a1) at (3.5,-0.25);
 \coordinate[label=below:$K_{XW}$] (a3) at (4.25,1.75);
\end{tikzpicture}
\caption{The complexes in Cases $2. 2$ to $2. 6$.}
\label{fc2.2to2.6}
\end{figure}

Thus in Case $2$, the skeletons of $C$ are not the lollipop graph.

\subsection{Case $3$: $(a_0, b_0, c_0, d_0)=(2, 2, 0, 0)$}

We look at $\MM_X(g^h)$ and $\MM_Y(g^h)$.
Projecting to the first two components, they are as follows.
\begin{equation*}
\begin{array}{lccrr}
(0, u(0))\leftarrow (1, u(1))\leftarrow&(2, 2)&\\
&\downarrow&\\
&(r(1), 1)&\\
&\downarrow&\\
&(r(0), 0)&
\end{array}
\end{equation*}

Since $c_0=d_0=0$, we do not have to consider the glueing along $Z$ and $W$.
The paths which connect the points $(2, 2)$ and $(1, u(1))$, $(r(1), 1)$ are selected from the dashed and dotted lines in Figure \ref{fc3}, and the former is ``above'' or ``equal to'' the latter by Lemma \ref{sakaime}.
\begin{figure}[H]
\centering
\begin{tikzpicture}
 \coordinate[label=below left:O] (A1) at (0,0);
 \coordinate [label=below:1] (A2) at (0.5,0);
 \coordinate [label=below:2] (A3) at (1,0);
 \coordinate [label=below:3] (A4) at (1.5,0);
 \coordinate [label=below:4] (A5) at (2,0);
 \coordinate [label=left:1] (A6) at (0,0.5);
 \coordinate (A7) at (0.5,0.5);
 \coordinate (A8) at (1,0.5);
 \coordinate (A9) at (1.5,0.5);
 \coordinate [label=left:2] (A10) at (0,1);
 \coordinate (A11) at (0.5,1);
 \coordinate (A12) at (1,1);
 \coordinate [label=left:3] (A13) at (0,1.5);
 \coordinate (A14) at (0.5,1.5);
 \coordinate [label=left:4] (A15) at (0,2);
 \draw [very thin, gray] (A1)--(A5)--(A15)--cycle; 
 \foreach \P in {2,7,11,14}  \draw[dashed, blue] (A12)--(A\P);
 \foreach \P in {6,7,8,9}  \draw[thick, dotted, red] (A12)--(A\P);
 \foreach \t in {1,2,...,15} \fill[black] (A\t) circle (0.05);
\end{tikzpicture}
\caption{Case $3$.}
\label{fc3}
\end{figure}

It is sufficient to consider the $8$ cases below as the pairs $(u(1), r(1))$ from the symmetry between $X$ and $Y$.
\begin{eqnarray*}
\text{Case $3. 1$: }(u(1), r(1))=(0, 2), \quad \text{Case $3. 2$: }(u(1), r(1))=(0, 3), \\
\text{Case $3. 3$: }(u(1), r(1))=(1, 1), \quad \text{Case $3. 4$: }(u(1), r(1))=(1, 2), \\
\text{Case $3. 5$: }(u(1), r(1))=(1, 3), \quad \text{Case $3. 6$: }(u(1), r(1))=(2, 2), \\
\text{Case $3. 7$: }(u(1), r(1))=(2, 3), \quad \text{Case $3. 8$: }(u(1), r(1))=(3, 3).\hspace{0.5mm}
\end{eqnarray*}

\begin{Lem}\label{2011}
Let $A_1$, $A_2$ and $A_3$ be three of $X$, $Y$, $Z$ and $W$.
Assume that the projection of $\MM_{A_1}(g^h)$ to the coordinates $A_1$, $A_2$ and $A_3$ is $(2, 0, 0)\rightarrow (1, 1, 0)\rightarrow \cdots$.
Then the skeleton of $C_{A_1A_2}\cup C_{A_1A_3}$ is as in Figure \ref{lemma5.19nozu}.
\begin{figure}[H]
\centering
\begin{tikzpicture}
 \coordinate[label=right:$O$] (O1) at (8.5,2);
 \coordinate (X1) at (6.4,2);
 \coordinate (Z1) at (8.5,1);
 \coordinate (X2) at (7.1,2);
 \coordinate (X3) at (7.8,2);
 \draw[thin,->,>=stealth] (O1)--(X1) node[left] {$A_1$};
 \draw[thin,->,>=stealth] (O1)--(Z1);
 \fill[black] (X2) circle (0.06);
 \fill[black] (X3) circle (0.06);
 \draw (X2) arc (180:360:0.35);
\end{tikzpicture}
\caption{The skeleton of $C_{A_1A_2}\cup C_{A_1A_3}$ in Lemma \ref{2011}.}
\label{lemma5.19nozu}
\end{figure}
\end{Lem}

\begin{proof}
The complexes $K_{A_1A_2}$ and $K_{A_1A_3}$ are as in Figure \ref{proofoflemma5.19nozu} (a).
The skeletons of $C_{A_1A_2}$ and $C_{A_1A_3}$ are as in (b).
Thus the skeleton of $C_{A_1A_2}\cup C_{A_1A_3}$ is the same as that of $C_{A_1A_2}$.
\begin{figure}[H]
\centering
\begin{tikzpicture}
 \coordinate (A1) at (2.5,0);
 \coordinate (A2) at (3,0);
 \coordinate (A3) at (3.5,0);
 \coordinate (A4) at (4,0);
 \coordinate (A5) at (4.5,0);
 \coordinate (A6) at (2.5,0.5);
 \coordinate (A7) at (3,0.5);
 \coordinate (A8) at (3.5,0.5);
 \coordinate (A9) at (4,0.5);
 \coordinate (A10) at (2.5,1);
 \coordinate (A11) at (3,1);
 \coordinate (A12) at (3.5,1);
 \coordinate (A13) at (2.5,1.5);
 \coordinate (A14) at (3,1.5);
 \coordinate (A15) at (2.5,2);
 \draw [very thin, gray] (A1)--(A5)--(A15)--cycle; 
 \draw (A3)--(A7);
 \foreach \P in {1,6,10,13,15} \draw[dashed] (A7)--(A\P);
 \foreach \t in {1,2,...,15} \fill[black] (A\t) circle (0.05);
 \coordinate (A1) at (5,0);
 \coordinate (A2) at (5.5,0);
 \coordinate (A3) at (6,0);
 \coordinate (A4) at (6.5,0);
 \coordinate (A5) at (7,0);
 \coordinate (A6) at (5,0.5);
 \coordinate (A7) at (5.5,0.5);
 \coordinate (A8) at (6,0.5);
 \coordinate (A9) at (6.5,0.5);
 \coordinate (A10) at (5,1);
 \coordinate (A11) at (5.5,1);
 \coordinate (A12) at (6,1);
 \coordinate (A13) at (5,1.5);
 \coordinate (A14) at (5.5,1.5);
 \coordinate (A15) at (5,2);
 \draw [very thin, gray] (A1)--(A5)--(A15)--cycle; 
 \draw (A3)--(A2);
 \foreach \P in {1,6,10,13,15} \draw[dashed] (A2)--(A\P);
 \foreach \t in {1,2,...,15} \fill[black] (A\t) circle (0.05);
 \coordinate (O1) at (9.5,2);
 \coordinate (X1) at (7.5,2);
 \coordinate (Y1) at (9.5,0);
 \coordinate (X2) at (8.1,2);
 \coordinate (X3) at (8.8,2);
 \coordinate (Y2) at (9.5,1.3);
 \coordinate (Y3) at (9.5,0.6);
 \draw[thin,->,>=stealth] (O1)--(X1) node[above] {$A_1$};
 \draw[thin,->,>=stealth] (O1)--(Y1) node[below] {$A_2$};
 \fill[black] (X2) circle (0.06);
 \fill[black] (X3) circle (0.06);
 \draw (X2) arc (180:360:0.35);
 \coordinate (O1) at (12,2);
 \coordinate (X1) at (10,2);
 \coordinate (Y1) at (12,0);
 \coordinate (X2) at (10.6,2);
 \coordinate (X3) at (11.3,2);
 \coordinate (Y2) at (12,1.3);
 \coordinate (Y3) at (12,0.6);
 \draw[thin,->,>=stealth] (O1)--(X1) node[above] {$A_1$};
 \draw[thin,->,>=stealth] (O1)--(Y1) node[below] {$A_3$};
 \fill[black] (X2) circle (0.06);
 \fill[black] (X3) circle (0.06);
 \coordinate[label=below:(a)] (a2) at (4.75,-0.5);
 \coordinate[label=below:(b)] (a4) at (10,-0.5);
\end{tikzpicture}
\caption{The complexes and skeletons in the proof of Lemma \ref{2011}.}
\label{proofoflemma5.19nozu}
\end{figure}
\end{proof}

\subsubsection{Case $3. 1$: $(u(1), r(1))=(0, 2)$ and Case $3. 2$: $(u(1), r(1))=(0, 3)$}

In Case $3. 1$, $K_{XY}$ is as in Figure \ref{fc3.1and3.2} (a) and the skeleton of $C_{XY}$ is as in (b).
The set $\MM_Y(g^h)$ has to contain $(2, 2, 0, 0)\rightarrow (2, 1, 1, 0)$ or $(2, 2, 0, 0)\rightarrow (2, 1, 0, 1)$.
In any case, we have $(2, 0)\rightarrow (1, 1)$ and $(2, 0)\rightarrow (1, 0)$ on $K_{YZ}$ and $K_{YW}$.
Then by Lemma \ref{2011}, the skeleton (c) appears when we glue $C_{XY}$, $C_{YZ}$ and $C_{YW}$ along $Y$.
In Case $3. 2$, $K_{XY}$ is as in (d) and the skeleton of $C_{XY}$ is as in (e).

In Cases $3. 1$ and $3. 2$, we can assume that $\MM_X(g^h)$ contains $(2, 2, 0, 0)\rightarrow (1, 0, 3, 0)$ or $(2, 2, 0, 0)\rightarrow (1, 0, 2, 1)$ by the symmetry between $Z$ and $W$.
In the case $(2, 2, 0, 0)\rightarrow (1, 0, 3, 0)$, $K_{XZ}$ is as in (f) and the skeleton of $C$ is not the lollipop graph by Lemma \ref{Case2.6}.
Hence we assume that $\MM_X(g^h)$ contains $(2, 2, 0, 0)\rightarrow (1, 0, 2, 1)$.
Then $K_{XZ}$ and $K_{XW}$ are as in (g).
By connecting the skeletons of $C_{XZ}$ and others at $X$, we obtain the skeleton as in (h).

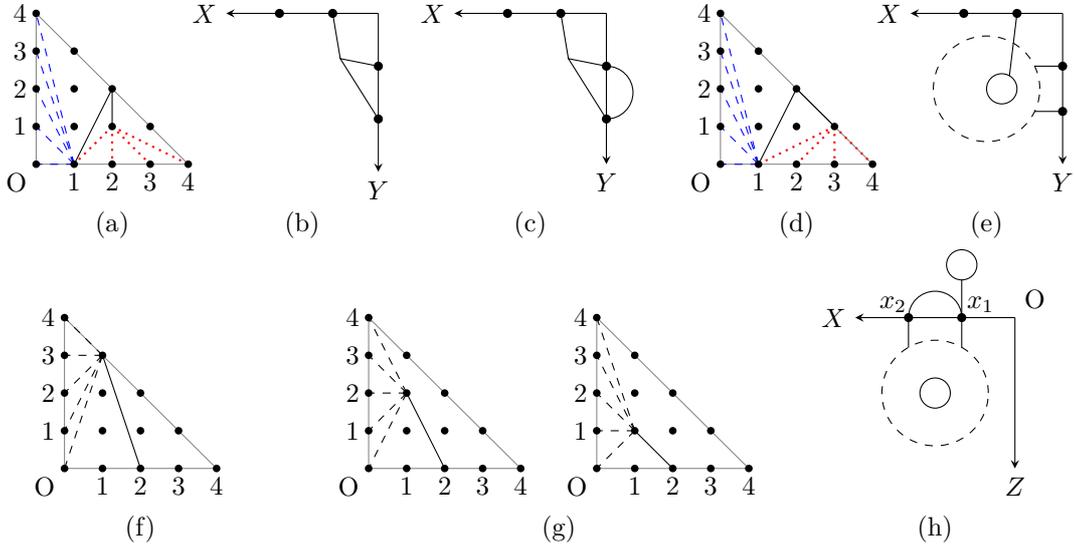
\begin{figure}[H]
\centering
\begin{tikzpicture}
 \coordinate[label=below left:O] (A1) at (0,0);
 \coordinate [label=below:1] (A2) at (0.5,0);
 \coordinate [label=below:2] (A3) at (1,0);
 \coordinate [label=below:3] (A4) at (1.5,0);
 \coordinate [label=below:4] (A5) at (2,0);
 \coordinate [label=left:1] (A6) at (0,0.5);
 \coordinate (A7) at (0.5,0.5);
 \coordinate (A8) at (1,0.5);
 \coordinate (A9) at (1.5,0.5);
 \coordinate [label=left:2] (A10) at (0,1);
 \coordinate (A11) at (0.5,1);
 \coordinate (A12) at (1,1);
 \coordinate [label=left:3] (A13) at (0,1.5);
 \coordinate (A14) at (0.5,1.5);
 \coordinate [label=left:4] (A15) at (0,2);
 \draw [very thin, gray] (A1)--(A5)--(A15)--cycle; 
 \draw (A12)--(A2);
 \draw (A12)--(A8);
 \foreach \P in {1,6,10,13,15}  \draw[dashed, blue] (A2)--(A\P);
 \foreach \P in {2,3,4,5}  \draw[thick, dotted, red] (A8)--(A\P);
 \foreach \t in {1,2,...,15} \fill[black] (A\t) circle (0.05);
 \coordinate (O1) at (4.5,2);
 \coordinate (X1) at (2.5,2);
 \coordinate (Y1) at (4.5,-0.1);
 \coordinate (X2) at (3.2,2);
 \coordinate (X3) at (3.9,2);
 \coordinate (Y2) at (4.5,1.3);
 \coordinate (Y3) at (4.5,0.6);
 \coordinate (P) at (4,1.4);
 \draw[thin,->,>=stealth] (O1)--(X1) node[left] {$X$};
 \draw[thin,->,>=stealth] (O1)--(Y1) node[below] {$Y$};
 \fill[black] (X2) circle (0.06);
 \fill[black] (X3) circle (0.06);
 \fill[black] (Y2) circle (0.06);
 \fill[black] (Y3) circle (0.06);
 \draw (Y2)--(P)--(Y3);
 \draw (X3)--(P);
 \coordinate (O3) at (7.5,2);
 \coordinate (Q1) at (5.5,2);
 \coordinate (W1) at (7.5,0);
 \coordinate (Q2) at (6.2,2);
 \coordinate (Q3) at (6.9,2);
 \coordinate (W2) at (7.5,1.3);
 \coordinate (W3) at (7.5,0.6);
 \coordinate (E) at (7,1.4);
 \draw[thin,->,>=stealth] (O3)--(Q1) node[left] {$X$};
 \draw[thin,->,>=stealth] (O3)--(W1) node[below] {$Y$};
 \fill[black] (Q2) circle (0.06);
 \fill[black] (Q3) circle (0.06);
 \fill[black] (W2) circle (0.06);
 \fill[black] (W3) circle (0.06);
 \draw (W2)--(E)--(W3);
 \draw (Q3)--(E);
 \draw (W2) arc (90:-90:0.35);
 \coordinate [label=below:(a)] (a1) at (1,-0.5);
 \coordinate [label=below:(b)] (a2) at (3.5,-0.5);
 \coordinate [label=below:(c)] (a3) at (6.5,-0.5);
 \coordinate[label=below left:O] (A1) at (9,0);
 \coordinate [label=below:1] (A2) at (9.5,0);
 \coordinate [label=below:2] (A3) at (10,0);
 \coordinate [label=below:3] (A4) at (10.5,0);
 \coordinate [label=below:4] (A5) at (11,0);
 \coordinate [label=left:1] (A6) at (9,0.5);
 \coordinate (A7) at (9.5,0.5);
 \coordinate (A8) at (10,0.5);
 \coordinate (A9) at (10.5,0.5);
 \coordinate [label=left:2] (A10) at (9,1);
 \coordinate (A11) at (9.5,1);
 \coordinate (A12) at (10,1);
 \coordinate [label=left:3] (A13) at (9,1.5);
 \coordinate (A14) at (9.5,1.5);
 \coordinate [label=left:4] (A15) at (9,2);
 \draw [very thin, gray] (A1)--(A5)--(A15)--cycle; 
 \draw (A12)--(A2);
 \draw (A12)--(A9);
 \foreach \P in {1,6,10,13,15}  \draw[dashed, blue] (A2)--(A\P);
 \foreach \P in {2,3,4,5}  \draw[thick, dotted, red] (A9)--(A\P);
 \foreach \t in {1,2,...,15} \fill[black] (A\t) circle (0.05);
 \coordinate (O1) at (13.5,2);
 \coordinate (X1) at (11.5,2);
 \coordinate (Y1) at (13.5,0);
 \coordinate (X2) at (12.2,2);
 \coordinate (X3) at (12.9,2);
 \coordinate (Y2) at (13.5,1.3);
 \coordinate (Y3) at (13.5,0.7);
 \draw[thin,->,>=stealth] (O1)--(X1) node[left] {$X$};
 \draw[thin,->,>=stealth] (O1)--(Y1) node[below] {$Y$};
 \fill[black] (X2) circle (0.06);
 \fill[black] (X3) circle (0.06);
 \fill[black] (Y2) circle (0.06);
 \fill[black] (Y3) circle (0.06);
 \draw (X3)--(12.8,1.17);
 \draw (Y2)--(13.13,1.3);
 \draw (Y3)--(13.13,0.7);
 \draw (12.7,1) circle (0.2);
 \draw[dashed] (12.5,1) circle (0.7);
 \coordinate [label=below:(d)] (a1) at (10,-0.5);
 \coordinate [label=below:(e)] (a2) at (12.5,-0.5);
\end{tikzpicture}
\begin{tikzpicture}
 \coordinate[label=below left:O] (Q1) at (-4,0);
 \coordinate [label=below:1] (Q2) at (-3.5,0);
 \coordinate [label=below:2] (Q3) at (-3,0);
 \coordinate [label=below:3] (Q4) at (-2.5,0);
 \coordinate [label=below:4] (Q5) at (-2,0);
 \coordinate [label=left:1] (Q6) at (-4,0.5);
 \coordinate (Q7) at (-3.5,0.5);
 \coordinate (Q8) at (-3,0.5);
 \coordinate (Q9) at (-2.5,0.5);
 \coordinate [label=left:2] (Q10) at (-4,1);
 \coordinate (Q11) at (-3.5,1);
 \coordinate (Q12) at (-3,1);
 \coordinate [label=left:3] (Q13) at (-4,1.5);
 \coordinate (Q14) at (-3.5,1.5);
 \coordinate [label=left:4] (Q15) at (-4,2);
 \draw [very thin, gray] (Q1)--(Q5)--(Q15)--cycle; 
 \draw (Q3)--(Q14);
 \foreach \t in {1,2,...,15} \fill[black] (Q\t) circle (0.05);
 \foreach \P in {1,6,10,13,15}  \draw[dashed] (Q14)--(Q\P);
 \coordinate[label=below left:O] (A1) at (0,0);
 \coordinate [label=below:1] (A2) at (0.5,0);
 \coordinate [label=below:2] (A3) at (1,0);
 \coordinate [label=below:3] (A4) at (1.5,0);
 \coordinate [label=below:4] (A5) at (2,0);
 \coordinate [label=left:1] (A6) at (0,0.5);
 \coordinate (A7) at (0.5,0.5);
 \coordinate (A8) at (1,0.5);
 \coordinate (A9) at (1.5,0.5);
 \coordinate [label=left:2] (A10) at (0,1);
 \coordinate (A11) at (0.5,1);
 \coordinate (A12) at (1,1);
 \coordinate [label=left:3] (A13) at (0,1.5);
 \coordinate (A14) at (0.5,1.5);
 \coordinate [label=left:4] (A15) at (0,2);
 \draw [very thin, gray] (A1)--(A5)--(A15)--cycle; 
 \draw (A3)--(A11);
 \foreach \t in {1,2,...,15} \fill[black] (A\t) circle (0.05);
 \foreach \P in {1,6,10,13,15}  \draw[dashed] (A11)--(A\P);
 \coordinate[label=below left:O] (B1) at (3,0);
 \coordinate [label=below:1] (B2) at (3.5,0);
 \coordinate [label=below:2] (B3) at (4,0);
 \coordinate [label=below:3] (B4) at (4.5,0);
 \coordinate [label=below:4] (B5) at (5,0);
 \coordinate [label=left:1] (B6) at (3,0.5);
 \coordinate (B7) at (3.5,0.5);
 \coordinate (B8) at (4,0.5);
 \coordinate (B9) at (4.5,0.5);
 \coordinate [label=left:2] (B10) at (3,1);
 \coordinate (B11) at (3.5,1);
 \coordinate (B12) at (4,1);
 \coordinate [label=left:3] (B13) at (3,1.5);
 \coordinate (B14) at (3.5,1.5);
 \coordinate [label=left:4] (B15) at (3,2);
 \draw [very thin, gray] (B1)--(B5)--(B15)--cycle; 
 \draw (B3)--(B7);
 \foreach \t in {1,2,...,15} \fill[black] (B\t) circle (0.05);
 \foreach \P in {1,6,10,13,15}  \draw[dashed] (B7)--(B\P);
 \coordinate[label=above right:O] (O1) at (8.5,2);
 \coordinate (X1) at (6.4,2);
 \coordinate (Z1) at (8.5,0);
 \coordinate (X2) at (7.1,2);
 \coordinate (X3) at (7.8,2);
 \draw[thin,->,>=stealth] (O1)--(X1) node[left] {$X$};
 \draw[thin,->,>=stealth] (O1)--(Z1) node[below] {$Z$};
 \fill[black] (X2) circle (0.06);
 \fill[black] (X3) circle (0.06);
 \draw (X2)--(7.1,1.6);
 \draw(X3)--(7.8,1.6);
 \draw[dashed] (7.45,1) circle (0.7);
 \draw (7.45,1) circle (0.2);
 \draw (X3)--(7.8,2.5);
 \draw (7.8,2.7) circle (0.2);
 \draw (X2) arc (180:0:0.35);
 \coordinate [label=below:(f)] (a1) at (-3,-0.5);
 \coordinate [label=below:(g)] (a2) at (2.5,-0.5);
 \coordinate [label=below:(h)] (a3) at (7.45,-0.5);
 \coordinate [label=below:$x_1$] (u1) at (8.05,2.4);
 \coordinate [label=below:$x_2$] (u2) at (6.9,2.4); 
\end{tikzpicture}
\caption{The complexes and skeletons in Cases $3. 1$ and $3. 2$.}
\label{fc3.1and3.2}
\end{figure}

Since $C_{XZ}$ is connected, we have the third cycle by glueing along $X$: glue reduced paths in $XZ$ and $XW$ from $x_1$ to $x_2$. 
Then any path connecting the upper cycle to the cycle in $XZ$ intersects the third cycle, so the skeleton of $C$ is not the lollipop graph.

Thus in Cases $3. 1$ and $3. 2$, the skeleton of $C$ is not the lollipop graph.

\subsubsection{Case $3. 3$: $(u(1), r(1))=(1, 1)$}

We may assume that $\MM_X(g^h)$ contains $(2, 2, 0, 0)\rightarrow (1, 1, 2, 0)$ or $(2, 2, 0, 0)\rightarrow (1, 1, 1, 1)$ by the symmetry between $Z$ and $W$.
We can consider $\MM_Y(g^h)$ in the same way (but without symmetry).

Case $1$: If $\MM_X(g^h)$ contains $(2, 2, 0, 0)\rightarrow (1, 1, 2, 0)$, $K_{XZ}$ and $K_{XW}$ are as in Figure \ref{fc3.3} (a) and (b).
Then when we glue the skeletons of $C_{XZ}\cup C_{XW}$ to $C_{XY}$, it is as in (c).

Case $2$: If $\MM_X(g^h)$ contains $(2, 2, 0, 0)\rightarrow (1, 1, 1, 1)$, $K_{XZ}$ and $K_{XW}$ are as in (d).
Then when we glue the skeletons of $C_{XZ}\cup C_{XW}$ to $C_{XY}$, it is as in (e).
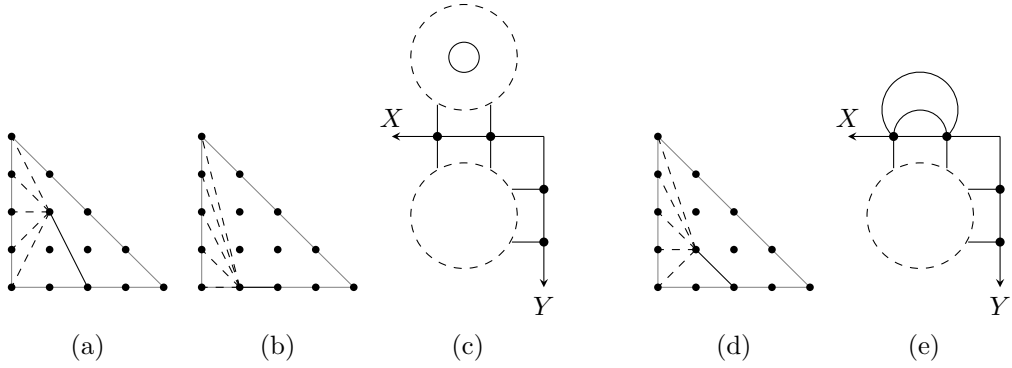
\begin{figure}[H]
\centering
\begin{tikzpicture}
 \coordinate (A1) at (2.5,0);
 \coordinate (A2) at (3,0);
 \coordinate (A3) at (3.5,0);
 \coordinate (A4) at (4,0);
 \coordinate (A5) at (4.5,0);
 \coordinate (A6) at (2.5,0.5);
 \coordinate (A7) at (3,0.5);
 \coordinate (A8) at (3.5,0.5);
 \coordinate (A9) at (4,0.5);
 \coordinate (A10) at (2.5,1);
 \coordinate (A11) at (3,1);
 \coordinate (A12) at (3.5,1);
 \coordinate (A13) at (2.5,1.5);
 \coordinate (A14) at (3,1.5);
 \coordinate (A15) at (2.5,2);
 \draw [very thin, gray] (A1)--(A5)--(A15)--cycle; 
 \draw (A3)--(A11);
 \foreach \P in {1,6,10,13,15} \draw[dashed] (A11)--(A\P);
 \foreach \t in {1,2,...,15} \fill[black] (A\t) circle (0.05);
 \coordinate (A1) at (5,0);
 \coordinate (A2) at (5.5,0);
 \coordinate (A3) at (6,0);
 \coordinate (A4) at (6.5,0);
 \coordinate (A5) at (7,0);
 \coordinate (A6) at (5,0.5);
 \coordinate (A7) at (5.5,0.5);
 \coordinate (A8) at (6,0.5);
 \coordinate (A9) at (6.5,0.5);
 \coordinate (A10) at (5,1);
 \coordinate (A11) at (5.5,1);
 \coordinate (A12) at (6,1);
 \coordinate (A13) at (5,1.5);
 \coordinate (A14) at (5.5,1.5);
 \coordinate (A15) at (5,2);
 \draw [very thin, gray] (A1)--(A5)--(A15)--cycle; 
 \draw (A3)--(A2);
 \foreach \P in {1,6,10,13,15} \draw[dashed] (A2)--(A\P);
 \foreach \t in {1,2,...,15} \fill[black] (A\t) circle (0.05);
 \coordinate (O1) at (9.5,2);
 \coordinate (X1) at (7.5,2);
 \coordinate (Y1) at (9.5,0);
 \coordinate (X2) at (8.1,2);
 \coordinate (X3) at (8.8,2);
 \coordinate (Y2) at (9.5,1.3);
 \coordinate (Y3) at (9.5,0.6);
 \draw[thin,->,>=stealth] (O1)--(X1) node[above] {$X$};
 \draw[thin,->,>=stealth] (O1)--(Y1) node[below] {$Y$};
 \fill[black] (X2) circle (0.06);
 \fill[black] (X3) circle (0.06);
 \fill[black] (Y2) circle (0.06);
 \fill[black] (Y3) circle (0.06);
 \draw (8.1,2.42)--(8.1,1.58);
 \draw (8.8,2.42)--(8.8,1.58);
 \draw[dashed] (8.45,0.95) circle (0.7);
 \draw[dashed] (8.45,3.05) circle (0.7);
 \draw (8.45,3.05) circle (0.2);
 \draw (Y2)--(9.08,1.3);
 \draw (Y3)--(9.08,0.6);
 \coordinate[label=below:(a)] (a2) at (3.5,-0.5);
 \coordinate[label=below:(b)] (a3) at (6,-0.5);
 \coordinate[label=below:(c)] (a4) at (8.5,-0.5);
  \coordinate (A1) at (11,0);
 \coordinate (A2) at (11.5,0);
 \coordinate (A3) at (12,0);
 \coordinate (A4) at (12.5,0);
 \coordinate (A5) at (13,0);
 \coordinate (A6) at (11,0.5);
 \coordinate (A7) at (11.5,0.5);
 \coordinate (A8) at (12,0.5);
 \coordinate (A9) at (12.5,0.5);
 \coordinate (A10) at (11,1);
 \coordinate (A11) at (11.5,1);
 \coordinate (A12) at (12,1);
 \coordinate (A13) at (11,1.5);
 \coordinate (A14) at (11.5,1.5);
 \coordinate (A15) at (11,2);
 \draw [very thin, gray] (A1)--(A5)--(A15)--cycle; 
 \draw (A3)--(A7);
 \foreach \P in {1,6,10,13,15} \draw[dashed] (A7)--(A\P);
 \foreach \t in {1,2,...,15} \fill[black] (A\t) circle (0.05);
 \coordinate (O1) at (15.5,2);
 \coordinate (X1) at (13.5,2);
 \coordinate (Y1) at (15.5,0);
 \coordinate (X2) at (14.1,2);
 \coordinate (X3) at (14.8,2);
 \coordinate (Y2) at (15.5,1.3);
 \coordinate (Y3) at (15.5,0.6);
 \draw[thin,->,>=stealth] (O1)--(X1) node[above] {$X$};
 \draw[thin,->,>=stealth] (O1)--(Y1) node[below] {$Y$};
 \fill[black] (X2) circle (0.06);
 \fill[black] (X3) circle (0.06);
 \fill[black] (Y2) circle (0.06);
 \fill[black] (Y3) circle (0.06);
 \draw (X2) arc (180:0:0.35);
 \draw (X3) arc (-45:225:0.5);
 \draw (X2)--(14.1,1.58);
 \draw (X3)--(14.8,1.58);
 \draw[dashed] (14.45,0.95) circle (0.7);
 \draw (Y2)--(15.08,1.3);
 \draw (Y3)--(15.08,0.6); 
 \coordinate[label=below:(d)] (a3) at (12,-0.5);
 \coordinate[label=below:(e)] (a4) at (14.5,-0.5); 
\end{tikzpicture}
\caption{The complexes and skeletons in Case $3. 3$.}
\label{fc3.3}
\end{figure}
Figure \ref{fc3.3} (c) and (e) will be also used in Case $3. 4$ and $3. 5$.

The complex $K_{XY}$ is as in Figure \ref{fc3.3s} (f) and the skeleton of $C_{XY}$ is as in (g).
We have the third cycle when glue along $X$ and $Y$.
This cycle passes through the points $x_1$ and $x_2$.
Any path connecting the cycle in $C_{XZ}\cup C_{XW}$ and the cycle in $C_{YZ}\cup C_{YW}$ passes through $x_1$ or $x_2$.
Thus, the skeleton of $C$ is not the lollipop graph.
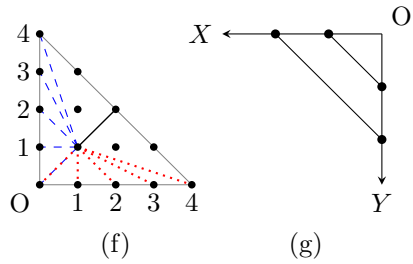
\begin{figure}[H]
\centering
\begin{tikzpicture}
 \coordinate[label=below left:O] (A1) at (1,0);
 \coordinate [label=below:1] (A2) at (1.5,0);
 \coordinate [label=below:2] (A3) at (2,0);
 \coordinate [label=below:3] (A4) at (2.5,0);
 \coordinate [label=below:4] (A5) at (3,0);
 \coordinate [label=left:1] (A6) at (1,0.5);
 \coordinate (A7) at (1.5,0.5);
 \coordinate (A8) at (2,0.5);
 \coordinate (A9) at (2.5,0.5);
 \coordinate [label=left:2] (A10) at (1,1);
 \coordinate (A11) at (1.5,1);
 \coordinate (A12) at (2,1);
 \coordinate [label=left:3] (A13) at (1,1.5);
 \coordinate (A14) at (1.5,1.5);
 \coordinate [label=left:4] (A15) at (1,2);
 \draw [very thin, gray] (A1)--(A5)--(A15)--cycle; 
 \draw (A12)--(A7);
 \draw (A12)--(A7);
 \foreach \P in {1,6,10,13,15}  \draw[dashed, blue] (A7)--(A\P);
 \foreach \P in {1,2,3,4,5}  \draw[thick, dotted, red] (A7)--(A\P);
 \foreach \t in {1,2,...,15} \fill[black] (A\t) circle (0.05);
 \coordinate[label=above right:O] (O1) at (5.5,2);
 \coordinate (X1) at (3.4,2);
 \coordinate (Y1) at (5.5,0);
 \coordinate (X2) at (4.1,2);
 \coordinate (X3) at (4.8,2);
 \coordinate (Y2) at (5.5,1.3);
 \coordinate (Y3) at (5.5,0.6);
 \draw[thin,->,>=stealth] (O1)--(X1) node[left] {$X$};
 \draw[thin,->,>=stealth] (O1)--(Y1) node[below] {$Y$};
 \fill[black] (X2) circle (0.06);
 \fill[black] (X3) circle (0.06);
 \fill[black] (Y2) circle (0.06);
 \fill[black] (Y3) circle (0.06);
 \draw (X2)--(Y3);
 \draw (X3)--(Y2);
 \coordinate[label=below:(f)] (a1) at (2,-0.5);
 \coordinate[label=below:(g)] (a2) at (4.5,-0.5);
\end{tikzpicture}
\caption{The complex $K_{XY}$ and the skeleton of $C_{XY}$ in Case $3. 3$.}
\label{fc3.3s}
\end{figure}

\subsubsection{Case $3. 4$: $(u(1), r(1))=(1, 2)$ and Case $3. 5$: $(u(1), r(1))=(1, 3)$}

In Case $3. 4$, $K_{XY}$ is as in Figure \ref{fc3.4} (a).
The $2$-simplex containing the line segment $(1, 1)$-$(2, 2)$ must be as in (b) or (c) since it has area $1/2$.
The set $\MM_Y(g^h)$ has to contain $(2, 2, 0, 0)\rightarrow (2, 1, 1, 0)$ or $(2, 2, 0, 0)\rightarrow (2, 1, 0, 1)$.
Using Lemma \ref{2011}, we see that the skeleton of $C_{XY}\cup C_{YZ}\cup C_{YW}$ is as in Figure \ref{fc3.4} (d) or (e).
\begin{figure}[H]
\centering
\begin{tikzpicture}
 \coordinate[label=below left:O] (A1) at (0,0);
 \coordinate [label=below:1] (A2) at (0.5,0);
 \coordinate [label=below:2] (A3) at (1,0);
 \coordinate [label=below:3] (A4) at (1.5,0);
 \coordinate [label=below:4] (A5) at (2,0);
 \coordinate [label=left:1] (A6) at (0,0.5);
 \coordinate (A7) at (0.5,0.5);
 \coordinate (A8) at (1,0.5);
 \coordinate (A9) at (1.5,0.5);
 \coordinate [label=left:2] (A10) at (0,1);
 \coordinate (A11) at (0.5,1);
 \coordinate (A12) at (1,1);
 \coordinate [label=left:3] (A13) at (0,1.5);
 \coordinate (A14) at (0.5,1.5);
 \coordinate [label=left:4] (A15) at (0,2);
 \draw [very thin, gray] (A1)--(A5)--(A15)--cycle; 
 \draw (A7)--(A12)--(A8);
 \foreach \P in {1,6,10,13,15}  \draw[dashed, blue] (A7)--(A\P);
 \foreach \P in {1,2,3,4,5}  \draw[thick, dotted, red] (A8)--(A\P);
 \foreach \t in {1,2,...,15} \fill[black] (A\t) circle (0.05); 
 \coordinate[label=below left:O] (B1) at (3,0);
 \coordinate [label=below:1] (B2) at (3.5,0);
 \coordinate [label=below:2] (B3) at (4,0);
 \coordinate [label=below:3] (B4) at (4.5,0);
 \coordinate [label=below:4] (B5) at (5,0);
 \coordinate [label=left:1] (B6) at (3,0.5);
 \coordinate (B7) at (3.5,0.5);
 \coordinate (B8) at (4,0.5);
 \coordinate (B9) at (4.5,0.5);
 \coordinate [label=left:2] (B10) at (3,1);
 \coordinate (B11) at (3.5,1);
 \coordinate (B12) at (4,1);
 \coordinate [label=left:3] (B13) at (3,1.5);
 \coordinate (B14) at (3.5,1.5);
 \coordinate [label=left:4] (B15) at (3,2);
 \draw [very thin, gray] (B1)--(B5)--(B15)--cycle; 
 \draw (B8)--(B12)--(B7);
 \draw (B7)--(B2)--(B12);
 \foreach \P in {1,6,10,13,15}  \draw[dashed, blue] (B7)--(B\P);
 \foreach \P in {2,3,4,5}  \draw[thick, dotted, red] (B8)--(B\P);
 \foreach \t in {1,2,...,15} \fill[black] (B\t) circle (0.05);
 \coordinate[label=below left:O] (C1) at (6,0);
 \coordinate [label=below:1] (C2) at (6.5,0);
 \coordinate [label=below:2] (C3) at (7,0);
 \coordinate [label=below:3] (C4) at (7.5,0);
 \coordinate [label=below:4] (C5) at (8,0);
 \coordinate [label=left:1] (C6) at (6,0.5);
 \coordinate (C7) at (6.5,0.5);
 \coordinate (C8) at (7,0.5);
 \coordinate (C9) at (7.5,0.5);
 \coordinate [label=left:2] (C10) at (6,1);
 \coordinate (C11) at (6.5,1);
 \coordinate (C12) at (7,1);
 \coordinate [label=left:3] (C13) at (6,1.5);
 \coordinate (C14) at (6.5,1.5);
 \coordinate [label=left:4] (C15) at (6,2);
 \draw [very thin, gray] (C1)--(C5)--(C15)--cycle; 
 \draw (C8)--(C12)--(C7)--cycle;
 \foreach \P in {1,6,10,13,15}  \draw[dashed, blue] (C7)--(C\P);
 \foreach \P in {1,2,3,4,5}  \draw[thick, dotted, red] (C8)--(C\P);
 \foreach \t in {1,2,...,15} \fill[black] (C\t) circle (0.05);
 \coordinate[label=below:(a)] (a1) at (1,-0.5);
 \coordinate[label=below:(b)] (a2) at (4,-0.5);
 \coordinate[label=below:(c)] (a3) at (7,-0.5);
 \coordinate[label=below:(d)] (a4) at (10,-0.5);
 \coordinate[label=below:(e)] (a5) at (13,-0.5);
 \coordinate (O1) at (11,2);
 \coordinate (X1) at (9,2);
 \coordinate (Y1) at (11,0);
 \coordinate (X2) at (10.3,2);
 \coordinate (X3) at (9.6,2);
 \coordinate (Y2) at (11,1.3);
 \coordinate (Y3) at (11,0.6);
 \draw[thin,->,>=stealth] (O1)--(X1) node[left] {$X$};
 \draw[thin,->,>=stealth] (O1)--(Y1) node[below] {$Y$};
 \fill[black] (X2) circle (0.06);
 \fill[black] (X3) circle (0.06);
 \fill[black] (Y2) circle (0.06);
 \fill[black] (Y3) circle (0.06); 
 \draw (X3)--(9.95,1.6)--(X2);
 \draw (9.95,1.6) to [out=270, in=180] (10.6,0.95);
 \draw (Y3)--(10.6,0.95)--(Y2);
 \draw (Y2) arc (90:-90:0.35);
 \coordinate (O11) at (14,2);
 \coordinate (X11) at (12,2);
 \coordinate (Y11) at (14,0);
 \coordinate (X22) at (13.3,2);
 \coordinate (X33) at (12.6,2);
 \coordinate (Y22) at (14,1.3);
 \coordinate (Y33) at (14,0.6);
 \draw[thin,->,>=stealth] (O11)--(X11) node[left] {$X$};
 \draw[thin,->,>=stealth] (O11)--(Y11) node[below] {$Y$};
 \fill[black] (X22) circle (0.06);
 \fill[black] (X33) circle (0.06);
 \fill[black] (Y22) circle (0.06);
 \fill[black] (Y33) circle (0.06);
 \draw (X33)--(12.6,0.6)--(Y33);
 \draw (X22)--(13.3,1.3)--(Y22);
 \draw (12.6,0.6)--(13.3,1.3);
 \draw (Y22) arc (90:-90:0.35);
\end{tikzpicture}
\caption{The complexes and skeletons in Case $3. 4$.}
\label{fc3.4}
\end{figure}
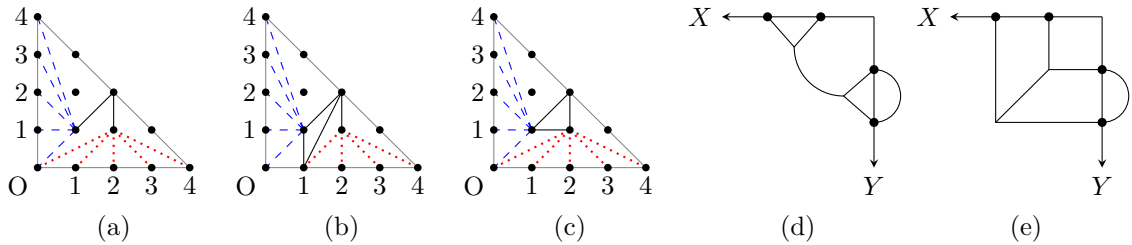

In Case $3. 5$, $K_{XY}$ and the skeleton of $C_{XY}$ are as in Figure \ref{fc3.5}.
\begin{figure}[H]
\centering
\begin{tikzpicture}
 \coordinate[label=below left:O] (A1) at (0,0);
 \coordinate [label=below:1] (A2) at (0.5,0);
 \coordinate [label=below:2] (A3) at (1,0);
 \coordinate [label=below:3] (A4) at (1.5,0);
 \coordinate [label=below:4] (A5) at (2,0);
 \coordinate [label=left:1] (A6) at (0,0.5);
 \coordinate (A7) at (0.5,0.5);
 \coordinate (A8) at (1,0.5);
 \coordinate (A9) at (1.5,0.5);
 \coordinate [label=left:2] (A10) at (0,1);
 \coordinate (A11) at (0.5,1);
 \coordinate (A12) at (1,1);
 \coordinate [label=left:3] (A13) at (0,1.5);
 \coordinate (A14) at (0.5,1.5);
 \coordinate [label=left:4] (A15) at (0,2);
 \draw [very thin, gray] (A1)--(A5)--(A15)--cycle; 
 \draw (A12)--(A7);
 \draw (A12)--(A9);
 \foreach \P in {1,6,10,13,15}  \draw[dashed, blue] (A7)--(A\P);
 \foreach \P in {1,2,3,4,5}  \draw[thick, dotted, red] (A9)--(A\P);
 \foreach \t in {1,2,...,15} \fill[black] (A\t) circle (0.05);
 \coordinate (O11) at (5,2);
 \coordinate (X11) at (3,2);
 \coordinate (Y11) at (5,0);
 \coordinate (X22) at (4.3,2);
 \coordinate (X33) at (3.6,2);
 \coordinate (Y22) at (5,1.3);
 \coordinate (Y33) at (5,0.6);
 \draw[thin,->,>=stealth] (O11)--(X11) node[left] {$X$};
 \draw[thin,->,>=stealth] (O11)--(Y11) node[below] {$Y$};
 \fill[black] (X22) circle (0.06);
 \fill[black] (X33) circle (0.06);
 \fill[black] (Y22) circle (0.06);
 \fill[black] (Y33) circle (0.06);
 \draw (X33)--(3.6,1.58);
 \draw (X22)--(4.3,1.58);
 \draw (Y33)--(4.55,0.6);
 \draw (Y22)--(4.55,1.3);
 \draw (3.95,0.95) circle (0.3);
 \draw[dashed] (3.95,0.95) circle (0.7);
\end{tikzpicture}
\caption{The complex $K_{XY}$ and the skeleton of $C_{XY}$ in Case $3. 5$.}
\label{fc3.5}
\end{figure}
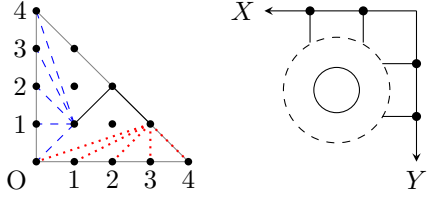

In Case $3. 4$ and $3. 5$, the skeleton of $C_{XY}\cup C_{YZ}\cup C_{YW}$ has $1$ cycle.
On the other hand, since $u(1)=1$, the skeleton of $C_{XZ}\cup C_{XW}$ is as in Case $3. 3$ and has $1$ cycle (see Figure \ref{fc3.3} (c) and (e)).
When we glue these along $X$, we have another cycle passing through $x_1$ and $x_2$.
Any path connecting the above cycles in $C_{XY}\cup C_{YZ}\cup C_{YW}$ and $C_{XZ}\cup C_{XW}$ passes through $x_1$ or $x_2$.
Therefore in Case $3. 4$ and $3. 5$, the skeleton of $C$ is not the lollipop graph.

\subsubsection{Case $3. 6$: $(u(1), r(1))=(2, 2)$}

The complex $K_{XY}$ is as in Figure \ref{fc3.6} (a).
The set $\MM_X(g^h)$ contains $(2, 2, 0, 0)\rightarrow (1, 2, 1, 0)$ or $(2, 2, 0, 0)\rightarrow (1, 2, 0, 1)$, and similarly for $\MM_Y(g^h)$.
By Lemma \ref{2011}, when we glue along $X$ and $Y$, we have the shape as in (b).
There is one cycle in the dashed circle part.
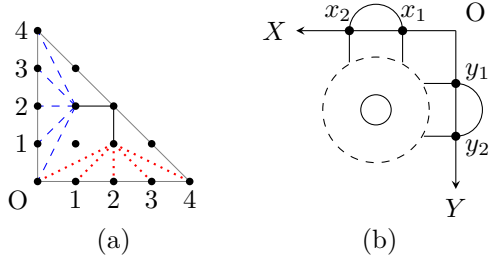
\begin{figure}[H]
\centering
\begin{tikzpicture}
 \coordinate[label=below left:O] (A1) at (-6,-3);
 \coordinate [label=below:1] (A2) at (-5.5,-3);
 \coordinate [label=below:2] (A3) at (-5,-3);
 \coordinate [label=below:3] (A4) at (-4.5,-3);
 \coordinate [label=below:4] (A5) at (-4,-3);
 \coordinate [label=left:1] (A6) at (-6,-2.5);
 \coordinate (A7) at (-5.5,-2.5);
 \coordinate (A8) at (-5,-2.5);
 \coordinate (A9) at (-4.5,-2.5);
 \coordinate [label=left:2] (A10) at (-6,-2);
 \coordinate (A11) at (-5.5,-2);
 \coordinate (A12) at (-5,-2);
 \coordinate [label=left:3] (A13) at (-6,-1.5);
 \coordinate (A14) at (-5.5,-1.5);
 \coordinate [label=left:4] (A15) at (-6,-1);
 \draw [very thin, gray] (A1)--(A5)--(A15)--cycle; 
 \draw (A12)--(A8);
 \draw (A12)--(A11);
 \foreach \P in {1,6,10,13,15}  \draw[dashed, blue] (A11)--(A\P);
 \foreach \P in {1,2,3,4,5}  \draw[thick, dotted, red] (A8)--(A\P);
 \foreach \t in {1,2,...,15} \fill[black] (A\t) circle (0.05);
 \coordinate[label=above right:O] (O1) at (-0.5,-1);
 \coordinate (X1) at (-2.6,-1);
 \coordinate (Y1) at (-0.5,-3.1);
 \coordinate (X2) at (-1.9,-1);
 \coordinate (X3) at (-1.2,-1);
 \coordinate (Y2) at (-0.5,-1.7);
 \coordinate (Y3) at (-0.5,-2.4);
 \draw[thin,->,>=stealth] (O1)--(X1) node[left] {$X$};
 \draw[thin,->,>=stealth] (O1)--(Y1) node[below] {$Y$};
 \fill[black] (X2) circle (0.06);
 \fill[black] (X3) circle (0.06);
 \fill[black] (Y2) circle (0.06);
 \fill[black] (Y3) circle (0.06);
 \draw (X2)--(-1.9,-1.42);
 \draw (X3)--(-1.2,-1.42);
 \draw[dashed] (-1.55,-2.05) circle (0.7);
 \draw (-1.55,-2.05) circle (0.2);
 \draw (Y2)--(-0.92,-1.7);
 \draw (Y3)--(-0.92,-2.4); 
 \draw (Y2) arc (90:-90:0.35);
 \draw (X2) arc (180:0:0.35);
 \coordinate[label=below:(a)] (a1) at (-5,-3.5);
 \coordinate[label=below:(b)] (a2) at (-1.5,-3.5);
 \coordinate [label=above:$x_1$] (u1) at (-1.05,-1);
 \coordinate [label=above:$x_2$] (u2) at (-2.05,-1); 
 \coordinate [label=right:$y_1$] (u1) at (-0.5,-1.55);
 \coordinate [label=right:$y_2$] (u2) at (-0.5,-2.55); 
\end{tikzpicture}
\caption{The complex $K_{XY}$ and skeleton in Case $3. 6$.}
\label{fc3.6}
\end{figure}
Note that the edge $(2, 2)$-$(1, 2)$ of $K_{XY}$ corresponds to the edge of $C_{XY}$ which passes through the point $x_1$ and that the edge $(2, 2)$-$(2, 1)$ of $K_{XY}$ corresponds to the edge of $C_{XY}$ which passes through the point $y_1$.
As the third vertex of the $2$-simplex containing the line segment $(1, 2)$-$(2, 2)$, there are $3$ cases: $(2, 1)$, $(1, 1)$ and $(0, 1)$.
It is sufficient to consider the following $3$ cases from Figure \ref{fc3.6s} (c) to (e).
In any case, the figure (b) shows that the vertices $(1, 2)$ and $(2, 1)$ of $K_{XY}$ correspond to $2$ cycles of $C$ (i.e. the corresponding domains are enclosed by a cycle in $C$).
In the case (c), the vertices $(1, 2)$ and $(2, 1)$ are connected directly.
In the case (d), the vertices $(1, 1)$ and $(1, 2)$ are connected directly.
In the case (e), there are two $2$-simplices $\Delta_1$ containing $(1, 2)$ and $\Delta_2$ containing $(1, 1)$ that have a common edge $(0, 1)$-$(2, 2)$.
In any case, the skeleton of $C$ is not the lollipop graph as in Lemma \ref{1112}.
\begin{figure}[H]
\centering
\begin{tikzpicture}
 \coordinate[label=below left:O] (C1) at (6,0);
 \coordinate [label=below:1] (C2) at (6.5,0);
 \coordinate [label=below:2] (C3) at (7,0);
 \coordinate [label=below:3] (C4) at (7.5,0);
 \coordinate [label=below:4] (C5) at (8,0);
 \coordinate [label=left:1] (C6) at (6,0.5);
 \coordinate (C7) at (6.5,0.5);
 \coordinate (C8) at (7,0.5);
 \coordinate (C9) at (7.5,0.5);
 \coordinate [label=left:2] (C10) at (6,1);
 \coordinate (C11) at (6.5,1);
 \coordinate (C12) at (7,1);
 \coordinate [label=left:3] (C13) at (6,1.5);
 \coordinate (C14) at (6.5,1.5);
 \coordinate [label=left:4] (C15) at (6,2);
 \draw [very thin, gray] (C1)--(C5)--(C15)--cycle; 
 \draw (C12)--(C8);
 \draw (C12)--(C11);
 \draw[very thick] (C11)--(C8);
 \foreach \P in {1,6,10,13,15}  \draw[dashed, blue] (C11)--(C\P);
 \foreach \P in {1,2,3,4,5}  \draw[thick, dotted, red] (C8)--(C\P);
 \foreach \t in {1,2,...,15} \fill[black] (C\t) circle (0.05);
 \coordinate[label=below left:O] (D1) at (9,0);
 \coordinate [label=below:1] (D2) at (9.5,0);
 \coordinate [label=below:2] (D3) at (10,0);
 \coordinate [label=below:3] (D4) at (10.5,0);
 \coordinate [label=below:4] (D5) at (11,0);
 \coordinate [label=left:1] (D6) at (9,0.5);
 \coordinate (D7) at (9.5,0.5);
 \coordinate (D8) at (10,0.5);
 \coordinate (D9) at (10.5,0.5);
 \coordinate [label=left:2] (D10) at (9,1);
 \coordinate (D11) at (9.5,1);
 \coordinate (D12) at (10,1);
 \coordinate [label=left:3] (D13) at (9,1.5);
 \coordinate (D14) at (9.5,1.5);
 \coordinate [label=left:4] (D15) at (9,2);
 \draw [very thin, gray] (D1)--(D5)--(D15)--cycle; 
 \draw (D12)--(D8);
 \draw (D12)--(D11);
 \draw (D12)--(D7)--(D11);
 \draw[very thick] (D7)--(D11);
 \foreach \P in {1,6,10,13,15}  \draw[dashed, blue] (D11)--(D\P);
 \foreach \P in {1,2,3,4,5}  \draw[thick, dotted, red] (D8)--(D\P);
 \foreach \t in {1,2,...,15} \fill[black] (D\t) circle (0.05);
 \coordinate[label=below left:O] (E1) at (12,0);
 \coordinate [label=below:1] (E2) at (12.5,0);
 \coordinate [label=below:2] (E3) at (13,0);
 \coordinate [label=below:3] (E4) at (13.5,0);
 \coordinate [label=below:4] (E5) at (14,0);
 \coordinate [label=left:1] (E6) at (12,0.5);
 \coordinate (E7) at (12.5,0.5);
 \coordinate (E8) at (13,0.5);
 \coordinate (E9) at (13.5,0.5);
 \coordinate [label=left:2] (E10) at (12,1);
 \coordinate (E11) at (12.5,1);
 \coordinate (E12) at (13,1);
 \coordinate [label=left:3] (E13) at (12,1.5);
 \coordinate (E14) at (12.5,1.5);
 \coordinate [label=left:4] (E15) at (12,2);
 \draw [very thin, gray] (E1)--(E5)--(E15)--cycle; 
 \draw (E12)--(E8);
 \draw (E12)--(E11);
 \draw (E12)--(E6)--(E11);
 \draw (E12)--(E7)--(E6);
 \draw[very thick] (E6)--(E12);
 \foreach \P in {6,10,13,15} \draw[dashed, blue] (E11)--(E\P);
 \foreach \P in {1,2,3,4,5} \draw[thick, dotted, red] (E8)--(E\P);
 \foreach \t in {1,2,...,15} \fill[black] (E\t) circle (0.05);
 \coordinate[label=below:(c)] (a3) at (7,-0.5);
 \coordinate[label=below:(d)] (a4) at (10,-0.5);
 \coordinate[label=below:(e)] (a5) at (13,-0.5);
\end{tikzpicture}
\caption{The complexes in Case $3. 6$.}
\label{fc3.6s}
\end{figure}
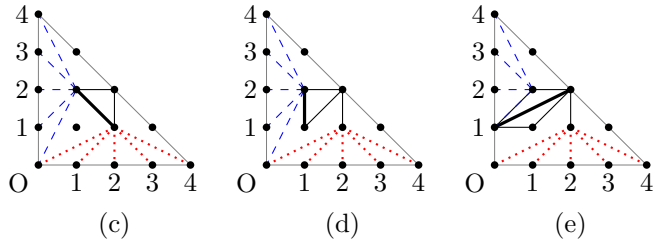

\subsubsection{Case $3. 7$: $(u(1), r(1))=(2, 3)$}

The complex $K_{XY}$ is as in Figure \ref{fc3.7} (a).
Assume that the skeleton of $C$ is the lollipop graph of genus $3$.
Then by Lemma \ref{12}, $K_{XY}$ contains the $2$-simplices as in (b).
The set $\MM_X(g^h)$ contains $(2, 2, 0, 0)\rightarrow (1, 2, 1, 0)$ or $(2, 2, 0, 0)\rightarrow (1, 2, 0, 1)$.
The glueing of the skeletons of $C_{XY}$, $C_{XZ}$ and $C_{XW}$ is as in (c).
\begin{figure}[H]
\centering
\begin{tikzpicture}
 \coordinate[label=below left:O] (A1) at (0,0);
 \coordinate [label=below:1] (A2) at (0.5,0);
 \coordinate [label=below:2] (A3) at (1,0);
 \coordinate [label=below:3] (A4) at (1.5,0);
 \coordinate [label=below:4] (A5) at (2,0);
 \coordinate [label=left:1] (A6) at (0,0.5);
 \coordinate (A7) at (0.5,0.5);
 \coordinate (A8) at (1,0.5);
 \coordinate (A9) at (1.5,0.5);
 \coordinate [label=left:2] (A10) at (0,1);
 \coordinate (A11) at (0.5,1);
 \coordinate (A12) at (1,1);
 \coordinate [label=left:3] (A13) at (0,1.5);
 \coordinate (A14) at (0.5,1.5);
 \coordinate [label=left:4] (A15) at (0,2);
 \draw [very thin, gray] (A1)--(A5)--(A15)--cycle; 
 \draw (A12)--(A11);
 \draw (A12)--(A9);
 \foreach \P in {1,6,10,13,15}  \draw[dashed, blue] (A11)--(A\P);
 \foreach \P in {1,2,3,4,5}  \draw[thick, dotted, red] (A9)--(A\P);
 \foreach \t in {1,2,...,15} \fill[black] (A\t) circle (0.05);
 \coordinate[label=below left:O] (D1) at (4,0);
 \coordinate [label=below:1] (D2) at (4.5,0);
 \coordinate [label=below:2] (D3) at (5,0);
 \coordinate [label=below:3] (D4) at (5.5,0);
 \coordinate [label=below:4] (D5) at (6,0);
 \coordinate [label=left:1] (D6) at (4,0.5);
 \coordinate (D7) at (4.5,0.5);
 \coordinate (D8) at (5,0.5);
 \coordinate (D9) at (5.5,0.5);
 \coordinate [label=left:2] (D10) at (4,1);
 \coordinate (D11) at (4.5,1);
 \coordinate (D12) at (5,1);
 \coordinate [label=left:3] (D13) at (4,1.5);
 \coordinate (D14) at (4.5,1.5);
 \coordinate [label=left:4] (D15) at (4,2);
 \draw [very thin, gray] (D1)--(D5)--(D15)--cycle;
 \foreach \P in {13,15}  \draw[dashed, blue] (D11)--(D\P);
 \foreach \P in {3,4,5}  \draw[thick, dotted, red] (D9)--(D\P);
 \draw (D9)--(D12);
 \draw (D11)--(D12);
 \draw (D13)--(D3)--(D11);
 \draw (D3)--(D7)--(D13);
 \draw (D3)--(D8)--(D11);
 \draw[very thick] (D8)--(D11);
 \foreach \t in {1,2,...,15} \fill[black] (D\t) circle (0.05);
 \coordinate[label=above right:O] (O1) at (10,2);
 \coordinate (X1) at (8,2);
 \coordinate (Y1) at (10,0);
 \coordinate (X2) at (8.6,2);
 \coordinate (X3) at (9.3,2);
 \coordinate (Y2) at (10,1.3);
 \coordinate (Y3) at (10,0.7);
 \draw[thin,->,>=stealth] (O1)--(X1) node[left] {$X$};
 \draw[thin,->,>=stealth] (O1)--(Y1) node[below] {$Y$};
 \fill[black] (X2) circle (0.06);
 \fill[black] (X3) circle (0.06);
 \fill[black] (Y2) circle (0.06);
 \fill[black] (Y3) circle (0.06);
 \draw (X2) arc (180:0:0.35);
 \draw (8.65,1) circle (0.2);
 \draw (9.25,1) circle (0.2);
 \draw [dashed] (8.95,1) circle (0.6);
 \draw (X2)--(8.6,1.5)--(9.08,1.12);
 \draw (X3)--(9.3,1.5)--(9.25,1.2);
 \draw (Y2)--(9.5,1.3);
 \draw (Y3)--(9.5,0.7);
 \draw[very thick] (9.25,1.2) arc (90:145:0.2);
 \coordinate [label=below:(a)] (a1) at (1,-0.5);
 \coordinate [label=below:(b)] (a2) at (5,-0.5); 
 \coordinate [label=below:(c)] (a3) at (9,-0.5); 
\end{tikzpicture}
\caption{The complexes and skeletons in Case $3. 7$.}
\label{fc3.7}
\end{figure}
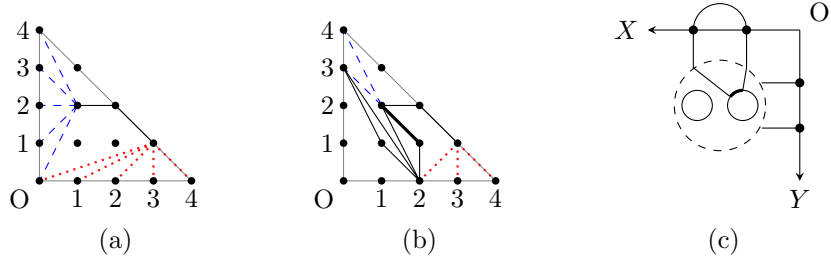

The $1$-simplex $(1, 2)$-$(2, 2)$ of $K_{XY}$ corresponds to the edge of $C_{XY}$ which contains the point $x_1\in C\cap X$.
The path from $x_1$ to $x_2$ in $C_{XY}$, corresponding to edges in $K_{XY}$ adjacent to $(1, 2)$, intersects the cycle which corresponds to the point $(2, 1)$ in the edge corresponding to $(1, 2)$-$(2, 1)$.
When we glue at $X$, we have a cycle which contains the points $x_1$ and $x_2$.
Thus the two cycles intersect and the skeleton of $C$ is not the lollipop graph.

\subsubsection{Case $3. 8$: $(u(1), r(1))=(3, 3)$}

The complex $K_{XY}$ is as in Figure \ref{fc3.8} and its support contains three lattice points in its interior.
Hence the skeleton of $C$ is not the lollipop graph by Lemma \ref{3in1}.

\begin{figure}[H]
\centering
\begin{tikzpicture}
 \coordinate[label=below left:O] (A1) at (0,0);
 \coordinate [label=below:1] (A2) at (0.5,0);
 \coordinate [label=below:2] (A3) at (1,0);
 \coordinate [label=below:3] (A4) at (1.5,0);
 \coordinate [label=below:4] (A5) at (2,0);
 \coordinate [label=left:1] (A6) at (0,0.5);
 \coordinate (A7) at (0.5,0.5);
 \coordinate (A8) at (1,0.5);
 \coordinate (A9) at (1.5,0.5);
 \coordinate [label=left:2] (A10) at (0,1);
 \coordinate (A11) at (0.5,1);
 \coordinate (A12) at (1,1);
 \coordinate [label=left:3] (A13) at (0,1.5);
 \coordinate (A14) at (0.5,1.5);
 \coordinate [label=left:4] (A15) at (0,2);
 \draw [very thin, gray] (A1)--(A5)--(A15)--cycle;
 \draw (A12)--(A14);
 \draw (A12)--(A9);
 \foreach \P in {1,6,10,13,15}  \draw[dashed, blue] (A14)--(A\P);
 \foreach \P in {1,2,3,4,5}  \draw[thick, dotted, red] (A9)--(A\P);
 \foreach \t in {1,2,...,15} \fill[black] (A\t) circle (0.05);
\end{tikzpicture}
\caption{The complex $K_{XY}$ in Case $3. 8$.}
\label{fc3.8}
\end{figure}
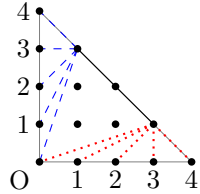

From the above, in  Case $3$, the skeleton of $C$ is not the lollipop graph.

\subsection{Case $4$: $(a_0, b_0, c_0, d_0)=(3, 1, 0, 0)$ and Case $5$: $(a_0, b_0, c_0, d_0)=(4, 0, 0, 0)$}

In Case $5$, it is sufficient to glue the skeletons of $C_{XY}$, $C_{XZ}$ and $C_{XW}$ to obtain the skeleton of $C$.
We can assume that $\MM_X(g^h)$ contains $(4, 0, 0, 0)\rightarrow (3, 1, 0, 0)$ from the symmetry between $Y$, $Z$ and $W$.
The complexes $K_{XY}$, $K_{XZ}$ and $K_{XW}$ are as in Figure \ref{fc5}.
\begin{figure}[H]
\centering
\begin{tikzpicture}
 \coordinate[label=below left:O] (A1) at (0,0);
 \coordinate [label=below:1] (A2) at (0.5,0);
 \coordinate [label=below:2] (A3) at (1,0);
 \coordinate [label=below:3] (A4) at (1.5,0);
 \coordinate [label=below:4] (A5) at (2,0);
 \coordinate [label=left:1] (A6) at (0,0.5);
 \coordinate (A7) at (0.5,0.5);
 \coordinate (A8) at (1,0.5);
 \coordinate (A9) at (1.5,0.5);
 \coordinate [label=left:2] (A10) at (0,1);
 \coordinate (A11) at (0.5,1);
 \coordinate (A12) at (1,1);
 \coordinate [label=left:3] (A13) at (0,1.5);
 \coordinate (A14) at (0.5,1.5);
 \coordinate [label=left:4] (A15) at (0,2);
 \draw [very thin, gray] (A1)--(A5)--(A15)--cycle; 
 \draw (A5)--(A9);
 \foreach \t in {1,2,...,15} \fill[black] (A\t) circle (0.05);
 \foreach \P in {3,8,12}  \draw[dashed] (A9)--(A\P);
 \coordinate[label=below left:O] (B1) at (3,0);
 \coordinate [label=below:1] (B2) at (3.5,0);
 \coordinate [label=below:2] (B3) at (4,0);
 \coordinate [label=below:3] (B4) at (4.5,0);
 \coordinate [label=below:4] (B5) at (5,0);
 \coordinate [label=left:1] (B6) at (3,0.5);
 \coordinate (B7) at (3.5,0.5);
 \coordinate (B8) at (4,0.5);
 \coordinate (B9) at (4.5,0.5);
 \coordinate [label=left:2] (B10) at (3,1);
 \coordinate (B11) at (3.5,1);
 \coordinate (B12) at (4,1);
 \coordinate [label=left:3] (B13) at (3,1.5);
 \coordinate (B14) at (3.5,1.5);
 \coordinate [label=left:4] (B15) at (3,2);
 \draw [very thin, gray] (B1)--(B5)--(B15)--cycle; 
 \draw (B5)--(B4);
 \foreach \t in {1,2,...,15} \fill[black] (B\t) circle (0.05);
 \foreach \P in {3,8,12} \draw[dashed] (B4)--(B\P);
 \coordinate[label=below left:O] (C1) at (6,0);
 \coordinate [label=below:1] (C2) at (6.5,0);
 \coordinate [label=below:2] (C3) at (7,0);
 \coordinate [label=below:3] (C4) at (7.5,0);
 \coordinate [label=below:4] (C5) at (8,0);
 \coordinate [label=left:1] (C6) at (6,0.5);
 \coordinate (C7) at (6.5,0.5);
 \coordinate (C8) at (7,0.5);
 \coordinate (C9) at (7.5,0.5);
 \coordinate [label=left:2] (C10) at (6,1);
 \coordinate (C11) at (6.5,1);
 \coordinate (C12) at (7,1);
 \coordinate [label=left:3] (C13) at (6,1.5);
 \coordinate (C14) at (6.5,1.5);
 \coordinate [label=left:4] (C15) at (6,2);
 \draw [very thin, gray] (C1)--(C5)--(C15)--cycle; 
 \draw (C5)--(C4);
 \foreach \t in {1,2,...,15} \fill[black] (C\t) circle (0.05);
 \foreach \P in {3,8,12}  \draw[dashed] (C4)--(C\P);
\end{tikzpicture}
\caption{The complexes $K_{XY}$, $K_{XZ}$ and $K_{XW}$ in Case $5$.}
\label{fc5}
\end{figure}
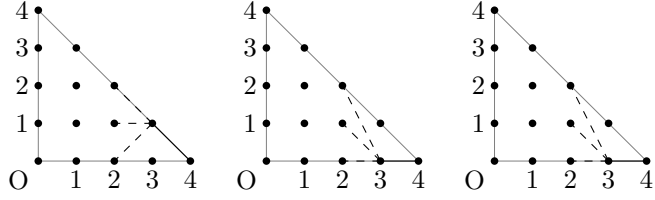

If we have the lollipop graph of genus $3$ in Case $5$, so do we also in Case $4$ with $\MM_Y(g^h)$ being $(3, 1, 0, 0)\rightarrow (4, 0, 0, 0)$ as in Figure \ref{fccomp4to5}.
\begin{figure}[H]
\centering
\begin{tikzpicture}
 \coordinate (A1) at (0,0);
 \coordinate (A2) at (0.5,0);
 \coordinate (A3) at (1,0);
 \coordinate (A4) at (1.5,0);
 \coordinate (A5) at (2,0);
 \coordinate (A6) at (0,0.5);
 \coordinate (A7) at (0.5,0.5);
 \coordinate (A8) at (1,0.5);
 \coordinate (A9) at (1.5,0.5);
 \coordinate (A10) at (0,1);
 \coordinate (A11) at (0.5,1);
 \coordinate (A12) at (1,1);
 \coordinate (A13) at (0,1.5);
 \coordinate (A14) at (0.5,1.5);
 \coordinate (A15) at (0,2);
 \draw [very thin, gray] (A1)--(A5)--(A15)--cycle; 
 \draw (A9)--(A5);
 \foreach \P in {3,8,12}  \draw[dashed] (A9)--(A\P);
 \foreach \t in {1,2,...,15} \fill[black] (A\t) circle (0.05);
 \coordinate (B1) at (2.5,0);
 \coordinate (B2) at (3,0);
 \coordinate (B3) at (3.5,0);
 \coordinate (B4) at (4,0);
 \coordinate (B5) at (4.5,0);
 \coordinate (B6) at (2.5,0.5);
 \coordinate (B7) at (3,0.5);
 \coordinate (B8) at (3.5,0.5);
 \coordinate (B9) at (4,0.5);
 \coordinate (B10) at (2.5,1);
 \coordinate (B11) at (3,1);
 \coordinate (B12) at (3.5,1);
 \coordinate (B13) at (2.5,1.5);
 \coordinate (B14) at (3,1.5);
 \coordinate (B15) at (2.5,2);
 \draw [very thin, gray] (B1)--(B5)--(B15)--cycle; 
 \foreach \P in {3,8,12} \draw[dashed] (B4)--(B\P);
 \foreach \t in {1,2,...,15} \fill[black] (B\t) circle (0.05);
 \coordinate (C1) at (5,0);
 \coordinate (C2) at (5.5,0);
 \coordinate (C3) at (6,0);
 \coordinate (C4) at (6.5,0);
 \coordinate (C5) at (7,0);
 \coordinate (C6) at (5,0.5);
 \coordinate (C7) at (5.5,0.5);
 \coordinate (C8) at (6,0.5);
 \coordinate (C9) at (6.5,0.5);
 \coordinate (C10) at (5,1);
 \coordinate (C11) at (5.5,1);
 \coordinate (C12) at (6,1);
 \coordinate (C13) at (5,1.5);
 \coordinate (C14) at (5.5,1.5);
 \coordinate (C15) at (5,2);
 \draw [very thin, gray] (C1)--(C5)--(C15)--cycle; 
 \foreach \t in {1,2,...,15} \fill[black] (C\t) circle (0.05);
 \foreach \P in {3,8,12} \draw[dashed] (C4)--(C\P);
 \coordinate (A1) at (7.5,0);
 \coordinate (A2) at (8,0);
 \coordinate (A3) at (8.5,0);
 \coordinate (A4) at (9,0);
 \coordinate (A5) at (9.5,0);
 \coordinate (A6) at (7.5,0.5);
 \coordinate (A7) at (8,0.5);
 \coordinate (A8) at (8.5,0.5);
 \coordinate (A9) at (9,0.5);
 \coordinate (A10) at (7.5,1);
 \coordinate (A11) at (8,1);
 \coordinate (A12) at (8.5,1);
 \coordinate (A13) at (7.5,1.5);
 \coordinate (A14) at (8,1.5);
 \coordinate (A15) at (7.5,2);
 \draw [very thin, gray] (A1)--(A5)--(A15)--cycle; 
 \draw (A2)--(A1);
 \foreach \t in {1,2,...,15} \fill[black] (A\t) circle (0.05);
 \coordinate (B1) at (10,0);
 \coordinate (B2) at (10.5,0);
 \coordinate (B3) at (11,0);
 \coordinate (B4) at (11.5,0);
 \coordinate (B5) at (12,0);
 \coordinate (B6) at (10,0.5);
 \coordinate (B7) at (10.5,0.5);
 \coordinate (B8) at (11,0.5);
 \coordinate (B9) at (11.5,0.5);
 \coordinate (B10) at (10,1);
 \coordinate (B11) at (10.5,1);
 \coordinate (B12) at (11,1);
 \coordinate (B13) at (10,1.5);
 \coordinate (B14) at (10.5,1.5);
 \coordinate (B15) at (10,2);
 \draw [very thin, gray] (B1)--(B5)--(B15)--cycle; 
 \draw (B2)--(B1);
 \foreach \t in {1,2,...,15} \fill[black] (B\t) circle (0.05);
 \coordinate[label=below:$K_{XY}$] (a3) at (1.65,1.75);
 \coordinate[label=below:$K_{XZ}$] (a3) at (4.15,1.75);
 \coordinate[label=below:$K_{XW}$] (a3) at (6.65,1.75);
 \coordinate[label=below:$K_{YZ}$] (a3) at (9.15,1.75);
 \coordinate[label=below:$K_{YW}$] (a3) at (11.65,1.75);
\end{tikzpicture}
\caption{Case $4$ with $\MM_Y(g^h)$ being $(3, 1, 0, 0)\rightarrow (4, 0, 0, 0)$.}
\label{fccomp4to5}
\end{figure}
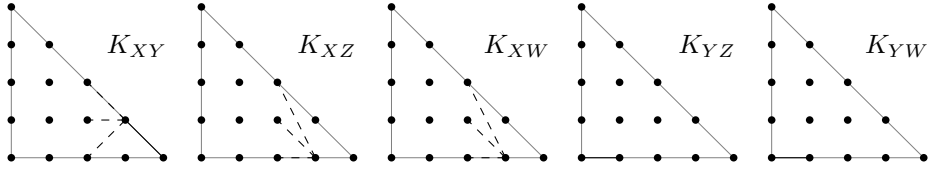
Thus, the tropical curves are as in Figure \ref{fcred4to5}.
Here, the intersection points $x_1$, $x_2$, $x_3$ and $x_4$ are shifted to $y$, $x_1$, $x_2$ and $x_3$, but the skeletons are the same.
\begin{figure}[H]
\centering
\begin{tikzpicture}
 \coordinate (O1) at (8,2);
 \coordinate (X1) at (5.25,2);
 \coordinate (Y1) at (8,0);
 \coordinate (X2) at (6.5,2);
 \coordinate (X4) at (7.5,2);
 \draw[thin,->,>=stealth] (O1)--(X1) node[left] {$X$};
 \draw[thin,->,>=stealth] (O1)--(Y1) node[below] {$Y$};
 \fill[black] (X2) circle (0.06);
 \fill[black] (X4) circle (0.06);
 \draw[dashed] (6.75,0.85) circle (0.7);
 \draw[dashed] (6.1,2.4) circle (0.2);
 \draw[dashed] (6.65,2.55) circle (0.2);
 \draw (X2)--(6.5,1.5);
 \draw (X4)--(7,1.5);
 \draw (X2)--(6.25,2.26);
 \draw (X2)--(6.6,2.36);
 \draw (X4)--(7.4,2.36);
 \draw (X4)--(7.7,2.36);
 \draw[dotted] (5.75,2.25)--(5.25,2.25);
 \draw[dotted] (5.75,1.75)--(5.25,1.75);
 \coordinate [label=below:$x_2$] (a1) at (6.75,2);
 \coordinate [label=below:$x_1$] (a2) at (7.6,2);
 \coordinate [label=below:$\leadsto$] (a2) at (8.75,1.25);
 \coordinate [label=below:$\text{Case $5$}$] (a2) at (6.5,-0.5);
 \coordinate (O1) at (12,2);
 \coordinate (X1) at (9.75,2);
 \coordinate (Y1) at (12,0);
 \coordinate (X2) at (11,2);
 \coordinate (X4) at (12,1.5);
 \draw[thin,->,>=stealth] (O1)--(X1) node[left] {$X$};
 \draw[thin,->,>=stealth] (O1)--(Y1) node[below] {$Y$};
 \fill[black] (X2) circle (0.06);
 \fill[black] (X4) circle (0.06);
 \draw[dashed] (10.95,0.85) circle (0.7);
 \draw[dashed] (10.6,2.4) circle (0.2);
 \draw[dashed] (11.15,2.55) circle (0.2);
 \draw (X2)--(11,1.53);
 \draw (X4)--(11.6,1.1);
 \draw (X2)--(10.75,2.26);
 \draw (X2)--(11.1,2.36);
 \draw (X4)--(12.5,1.75);
 \draw (X4)--(12.5,1.25);
 \draw[dotted] (10.25,2.25)--(9.75,2.25);
 \draw[dotted] (10.25,1.75)--(9.75,1.75);
 \coordinate [label=below:$x_1$] (a1) at (11.25,2);
 \coordinate [label=above:$y$] (a2) at (11.8,1.5);
 \coordinate [label=below:$\text{Case $4$}$] (a2) at (11,-0.5);
\end{tikzpicture}
\caption{The reduction from Case $5$ to Case $4$.}
\label{fcred4to5}
\end{figure}
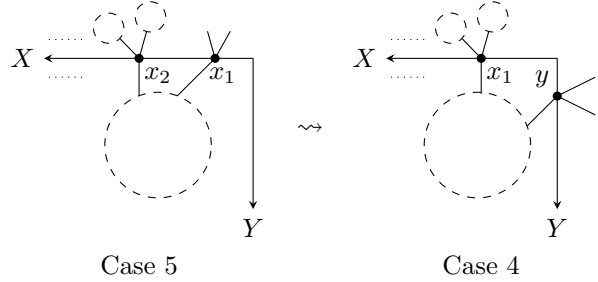
Hence it is sufficient to look at Case $4$.

In the following, we consider Case $4$: $(a_0, b_0, c_0, d_0)=(3, 1, 0, 0)$.
Since the skeletons of the restrictions of $C$ to $YZ$, $YW$ and $ZW$ are as in Figure \ref{fc4}, it is sufficient to consider the skeletons of $C_{XY}$, $C_{XZ}$ and $C_{XW}$ when studying the skeleton of $C$.
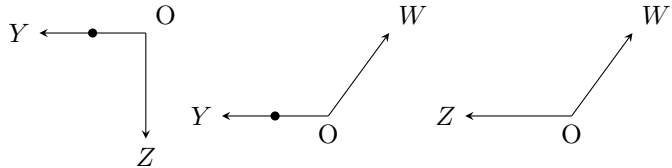
\begin{figure}[H]
\centering
\begin{tikzpicture}
 \coordinate[label=above right:O] (O1) at (5,0);
 \coordinate (Y1) at (3.6,0);
 \coordinate (Z1) at (5,-1.4);
 \coordinate (Y'1) at (4.3,0);
 \draw[thin,->,>=stealth] (O1)--(Y1) node[left] {$Y$};
 \draw[thin,->,>=stealth] (O1)--(Z1) node[below] {$Z$};
 \fill[black] (Y'1) circle (0.06);
 \coordinate[label=below:O] (O4) at (7.4,-1.1);
 \coordinate (Y4) at (6,-1.1);
 \coordinate (W4) at (8.2,0);
 \coordinate (Y'4) at (6.7,-1.1);
 \draw[thin,->,>=stealth] (O4)--(Y4) node[left] {$Y$};
 \draw[thin,->,>=stealth] (O4)--(W4) node[above right] {$W$};
 \fill[black] (Y'4) circle (0.06);
 \coordinate[label=below:O] (O5) at (10.6,-1.1);
 \coordinate (Z5) at (9.2,-1.1);
 \coordinate (W5) at (11.4,0);
 \draw[thin,->,>=stealth] (O5)--(Z5) node[left] {$Z$};
 \draw[thin,->,>=stealth] (O5)--(W5) node[above right] {$W$};
\end{tikzpicture}
\caption{The skeletons of the restrictions of $C$ to $YZ$, $YW$ and $ZW$ in Case $4$.}
\label{fc4}
\end{figure}

We look at $\MM_X(g^h)$.
It is sufficient to consider the following $4$ cases by the symmetry between $Z$ and $W$.
\begin{eqnarray*}
\text{Case $4.1$: }(3, 1, 0, 0)\rightarrow (2, 2, 0, 0), \quad \text{Case $4.2$: }(3, 1, 0, 0)\rightarrow (2, 1, 1, 0), \\
\text{Case $4.3$: }(3, 1, 0, 0)\rightarrow (2, 0, 2, 0), \quad \text{Case $4.4$: }(3, 1, 0, 0)\rightarrow (2, 0, 1, 1).\hspace{0.5mm}
\end{eqnarray*}

\begin{Lem}\label{case4hodai}
If $(a_0, b_0, c_0, d_0)=(3, 1, 0, 0)$ and $b_1(\overline{C'})=1$, the skeleton of $C$ is not the lollipop graph.
\end{Lem}

\begin{proof}
Assume $b_1(\overline{C'})=1$.
When we glue $\overline{C'}$ along $Y$, the number of cycles does not change.
Therefore, glueing along $X$ adds $2$ homologically independent cycles.
We need at least two points on $X$ to make a cycle.
Therefore, when we add two cycles, at least one of the three points on $X$ is contained in the two cycles.
Thus, two cycles share one point, and the skeleton of $C$ is not the lollipop graph.
\end{proof}

\subsubsection{Case $4. 1$: $(3, 1, 0, 0)\rightarrow (2, 2, 0, 0)$}

We divide into the following cases.
\begin{eqnarray*}
\text{Case $4. 1. 1$: }(2, 2, 0, 0)\rightarrow (1, 3, 0, 0), \quad \text{Case $4. 1. 2$: }(2, 2, 0, 0)\rightarrow (1, 2, 1, 0), \\
\text{Case $4. 1. 3$: }(2, 2, 0, 0)\rightarrow (1, 1, 2, 0), \quad \text{Case $4. 1. 4$: }(2, 2, 0, 0)\rightarrow (1, 1, 1, 1), \\
\text{Case $4. 1. 5$: }(2, 2, 0, 0)\rightarrow (1, 0, 3, 0), \quad \text{Case $4. 1. 6$: }(2, 2, 0, 0)\rightarrow (1, 0, 2, 1). \hspace{0.5mm}
\end{eqnarray*}

Case $4. 1. 1$: $(3, 1, 0, 0)\rightarrow (2, 2, 0, 0)\rightarrow (1, 3, 0, 0)$

The complex $K_{XY}$ is as in Figure \ref{fc4.1.1} and the the skeleton of $C$ is not the lollipop graph by Lemma \ref{3in1}.
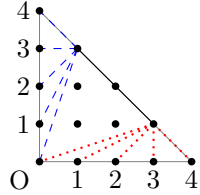
\begin{figure}[H]
\centering
\begin{tikzpicture}
 \coordinate[label=below left:O] (A1) at (0,0);
 \coordinate [label=below:1] (A2) at (0.5,0);
 \coordinate [label=below:2] (A3) at (1,0);
 \coordinate [label=below:3] (A4) at (1.5,0);
 \coordinate [label=below:4] (A5) at (2,0);
 \coordinate [label=left:1] (A6) at (0,0.5);
 \coordinate (A7) at (0.5,0.5);
 \coordinate (A8) at (1,0.5);
 \coordinate (A9) at (1.5,0.5);
 \coordinate [label=left:2] (A10) at (0,1);
 \coordinate (A11) at (0.5,1);
 \coordinate (A12) at (1,1);
 \coordinate [label=left:3] (A13) at (0,1.5);
 \coordinate (A14) at (0.5,1.5);
 \coordinate [label=left:4] (A15) at (0,2);
 \draw [very thin, gray] (A1)--(A5)--(A15)--cycle;
 \draw (A12)--(A14);
 \draw (A12)--(A9);
 \foreach \P in {1,6,10,13,15}  \draw[dashed, blue] (A14)--(A\P);
 \foreach \P in {1,2,3,4,5}  \draw[thick, dotted, red] (A9)--(A\P);
 \foreach \t in {1,2,...,15} \fill[black] (A\t) circle (0.05);
\end{tikzpicture}
\caption{The complex $K_{XY}$ in Case $4. 1. 1$.}
\label{fc4.1.1}
\end{figure}

Case $4. 1. 2$: $(3, 1, 0, 0)\rightarrow (2, 2, 0, 0)\rightarrow (1, 2, 1, 0)$

The complexes $K_{XY}$ and $K_{XZ}$ are as in Figure \ref{fc4.1.2} (a).
Assume that the skeleton of $C$ is the lollipop graph of genus $3$.
Then by Lemma \ref{12}, $K_{XY}$ contains the $2$-simplices as in (b), and the skeletons of $C_{XY}$ and $C_{XZ}$ are as in (c).
\begin{figure}[H]
\centering
\begin{tikzpicture}
 \coordinate[label=below left:O] (B1) at (3,0);
 \coordinate [label=below:1] (B2) at (3.5,0);
 \coordinate [label=below:2] (B3) at (4,0);
 \coordinate [label=below:3] (B4) at (4.5,0);
 \coordinate [label=below:4] (B5) at (5,0);
 \coordinate [label=left:1] (B6) at (3,0.5);
 \coordinate (B7) at (3.5,0.5);
 \coordinate (B8) at (4,0.5);
 \coordinate (B9) at (4.5,0.5);
 \coordinate [label=left:2] (B10) at (3,1);
 \coordinate (B11) at (3.5,1);
 \coordinate (B12) at (4,1);
 \coordinate [label=left:3] (B13) at (3,1.5);
 \coordinate (B14) at (3.5,1.5);
 \coordinate [label=left:4] (B15) at (3,2);
 \draw [very thin, gray] (B1)--(B5)--(B15)--cycle;
 \draw (B9)--(B12)--(B11);
 \foreach \P in {1,6,10,13,15}  \draw[dashed, blue] (B11)--(B\P);
 \foreach \P in {1,2,3,4,5}  \draw[thick, dotted, red] (B9)--(B\P);
 \foreach \t in {1,2,...,15} \fill[black] (B\t) circle (0.05);
 \coordinate[label=below left:O] (C1) at (6,0);
 \coordinate [label=below:1] (C2) at (6.5,0);
 \coordinate [label=below:2] (C3) at (7,0);
 \coordinate [label=below:3] (C4) at (7.5,0);
 \coordinate [label=below:4] (C5) at (8,0);
 \coordinate [label=left:1] (C6) at (6,0.5);
 \coordinate (C7) at (6.5,0.5);
 \coordinate (C8) at (7,0.5);
 \coordinate (C9) at (7.5,0.5);
 \coordinate [label=left:2] (C10) at (6,1);
 \coordinate (C11) at (6.5,1);
 \coordinate (C12) at (7,1);
 \coordinate [label=left:3] (C13) at (6,1.5);
 \coordinate (C14) at (6.5,1.5);
 \coordinate [label=left:4] (C15) at (6,2);
 \draw [very thin, gray] (C1)--(C5)--(C15)--cycle;
 \draw (C4)--(C3)--(C7);
 \foreach \t in {1,2,...,15} \fill[black] (C\t) circle (0.05);
 \foreach \P in {1,6,10,13,15}  \draw[dashed] (C7)--(C\P);
 \coordinate[label=below left:O] (D1) at (9.25,0);
 \coordinate [label=below:1] (D2) at (9.75,0);
 \coordinate [label=below:2] (D3) at (10.25,0);
 \coordinate [label=below:3] (D4) at (10.75,0);
 \coordinate [label=below:4] (D5) at (11.25,0);
 \coordinate [label=left:1] (D6) at (9.25,0.5);
 \coordinate (D7) at (9.75,0.5);
 \coordinate (D8) at (10.25,0.5);
 \coordinate (D9) at (10.75,0.5);
 \coordinate [label=left:2] (D10) at (9.25,1);
 \coordinate (D11) at (9.75,1);
 \coordinate (D12) at (10.25,1);
 \coordinate [label=left:3] (D13) at (9.25,1.5);
 \coordinate (D14) at (9.75,1.5);
 \coordinate [label=left:4] (D15) at (9.25,2);
 \draw [very thin, gray] (D1)--(D5)--(D15)--cycle;
 \foreach \P in {13,15}  \draw[dashed, blue] (D11)--(D\P);
 \foreach \P in {3,4,5}  \draw[thick, dotted, red] (D9)--(D\P);
 \draw (D9)--(D12);
 \draw (D11)--(D12);
 \draw (D13)--(D3)--(D11);
 \draw (D3)--(D7)--(D13);
 \draw (D3)--(D8)--(D11);
 \draw[very thick] (D8)--(D11);
 \foreach \t in {1,2,...,15} \fill[black] (D\t) circle (0.05);
 \coordinate [label=below:(a)] (a1) at (5.5,-0.5);
 \coordinate [label=below:(b)] (a2) at (10.25,-0.5); 
 \coordinate (O1) at (14,2);
 \coordinate (X1) at (12,2);
 \coordinate (Y1) at (14,0);
 \coordinate (X2) at (12.5,2);
 \coordinate (X3) at (13,2);
 \coordinate (X4) at (13.5,2);
 \coordinate (Y2) at (14,1);
 \draw[thin,->,>=stealth] (O1)--(X1) node[left] {$X$};
 \draw[thin,->,>=stealth] (O1)--(Y1) node[below] {$Y$};
 \fill[black] (X2) circle (0.06);
 \fill[black] (X3) circle (0.06);
 \fill[black] (X4) circle (0.06);
 \fill[black] (Y2) circle (0.06);
 \draw (12.7,1) circle (0.2);
 \draw (13.3,1) circle (0.2);
 \draw [dashed] (13,1) circle (0.7);
 \draw (X2)--(12.5,1.48)--(13.158,1.142);
 \draw (X3)--(13,1.7)--(13.442,1.142);
 \draw (X4)--(13.5,1.48);
 \draw (Y2)--(13.7,1);
 \draw[very thick] (13.158,1.142) arc (135:45:0.2);
 \coordinate (O2) at (17,2);
 \coordinate (X5) at (15,2);
 \coordinate (Z1) at (17,0);
 \coordinate (X6) at (15.5,2);
 \coordinate (X7) at (16,2);
 \coordinate (X8) at (16.5,2);
 \draw[thin,->,>=stealth] (O2)--(X5) node[left] {$X$};
 \draw[thin,->,>=stealth] (O2)--(Z1) node[below] {$Z$};
 \fill[black] (X6) circle (0.06);
 \fill[black] (X7) circle (0.06);
 \fill[black] (X8) circle (0.06);
 \draw (X6) arc (180:360:0.25);
 \coordinate [label=above:$x_2$] (a1) at (13,2);
 \coordinate [label=above:$x_3$] (a2) at (12.5,2); 
 \coordinate [label=above:$x_2$] (a1) at (16,2);
 \coordinate [label=above:$x_3$] (a2) at (15.5,2);
 \coordinate [label=below:(c)] (a3) at (14.75,-0.5); 
\end{tikzpicture}
\caption{The complexes and skeletons in Case $4. 1. 2$.}
\label{fc4.1.2}
\end{figure}
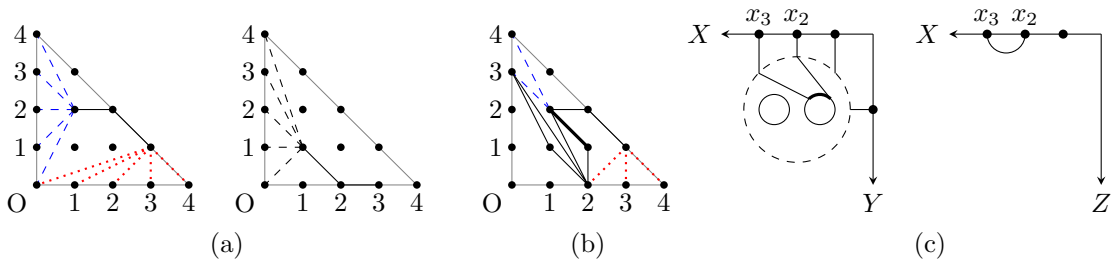

The $1$-simplex $(1, 2)$-$(2, 2)$ of $K_{XY}$ corresponds to the edge of $C_{XY}$ which contains the point $x_2\in C\cap X$.
The path from $x_2$ to $x_3$ in $C_{XY}$ intersects the cycle which corresponds to the point $(2, 1)$.
When we glue at $X$, we have a cycle which contains the points $x_2$ and $x_3$.
Thus the two cycles intersect and the skeleton of $C$ is not the lollipop graph.

\medbreak
Case $4. 1. 3$: $(3, 1, 0, 0)\rightarrow (2, 2, 0, 0)\rightarrow (1, 1, 2, 0)$

The complexes $K_{XY}$, $K_{XZ}$ and the skeletons $C_{XY}$ and $C_{XZ}$ are as in Figure \ref{fc4.1.3}.
\begin{figure}[H]
\centering
\begin{tikzpicture}
 \coordinate[label=below left:O] (A1) at (0,0);
 \coordinate [label=below:1] (A2) at (0.5,0);
 \coordinate [label=below:2] (A3) at (1,0);
 \coordinate [label=below:3] (A4) at (1.5,0);
 \coordinate [label=below:4] (A5) at (2,0);
 \coordinate [label=left:1] (A6) at (0,0.5);
 \coordinate (A7) at (0.5,0.5);
 \coordinate (A8) at (1,0.5);
 \coordinate (A9) at (1.5,0.5);
 \coordinate [label=left:2] (A10) at (0,1);
 \coordinate (A11) at (0.5,1);
 \coordinate (A12) at (1,1);
 \coordinate [label=left:3] (A13) at (0,1.5);
 \coordinate (A14) at (0.5,1.5);
 \coordinate [label=left:4] (A15) at (0,2);
 \draw [very thin, gray] (A1)--(A5)--(A15)--cycle; 
 \draw (A9)--(A12)--(A7);
 \foreach \P in {1,6,10,13,15}  \draw[dashed, blue] (A7)--(A\P);
 \foreach \P in {1,2,3,4,5}  \draw[thick, dotted, red] (A9)--(A\P);
 \foreach \t in {1,2,...,15} \fill[black] (A\t) circle (0.05);
 \coordinate[label=below left:O] (B1) at (3,0);
 \coordinate [label=below:1] (B2) at (3.5,0);
 \coordinate [label=below:2] (B3) at (4,0);
 \coordinate [label=below:3] (B4) at (4.5,0);
 \coordinate [label=below:4] (B5) at (5,0);
 \coordinate [label=left:1] (B6) at (3,0.5);
 \coordinate (B7) at (3.5,0.5);
 \coordinate (B8) at (4,0.5);
 \coordinate (B9) at (4.5,0.5);
 \coordinate [label=left:2] (B10) at (3,1);
 \coordinate (B11) at (3.5,1);
 \coordinate (B12) at (4,1);
 \coordinate [label=left:3] (B13) at (3,1.5);
 \coordinate (B14) at (3.5,1.5);
 \coordinate [label=left:4] (B15) at (3,2);
 \draw [very thin, gray] (B1)--(B5)--(B15)--cycle; 
 \draw (B4)--(B3)--(B11);
 \foreach \t in {1,2,...,15} \fill[black] (B\t) circle (0.05);
 \foreach \P in {1,6,10,13,15}  \draw[dashed] (B11)--(B\P);
 \coordinate[label=above right:O] (O1) at (8,2);
 \coordinate (X1) at (6,2);
 \coordinate (Y1) at (8,0);
 \coordinate (X2) at (6.5,2);
 \coordinate (X3) at (7,2);
 \coordinate (X4) at (7.5,2);
 \coordinate (Y2) at (8,1);
 \draw[thin,->,>=stealth] (O1)--(X1) node[left] {$X$};
 \draw[thin,->,>=stealth] (O1)--(Y1) node[below] {$Y$};
 \fill[black] (X2) circle (0.06);
 \fill[black] (X3) circle (0.06);
 \fill[black] (X4) circle (0.06);
 \fill[black] (Y2) circle (0.06);
 \draw[dashed] (7,1) circle (0.7);
 \draw (7,1) circle (0.2);
 \draw (X2)--(6.5,1.48);
 \draw (X3)--(7,1.7);
 \draw (X4)--(7.5,1.48);
 \draw (Y2)--(7.7,1); 
 \coordinate[label=above right:O] (O2) at (11.5,2);
 \coordinate (X5) at (9.5,2);
 \coordinate (Z1) at (11.5,0);
 \coordinate (X6) at (10,2);
 \coordinate (X7) at (10.5,2);
 \coordinate (X8) at (11,2);
 \draw[thin,->,>=stealth] (O2)--(X5) node[left] {$X$};
 \draw[thin,->,>=stealth] (O2)--(Z1) node[below] {$Z$};
 \fill[black] (X6) circle (0.06);
 \fill[black] (X7) circle (0.06);
 \fill[black] (X8) circle (0.06);
 \draw[dashed] (10.25,1) circle (0.7);
 \draw (10.25,1) circle (0.2);
 \draw (X6)--(10,1.65);
 \draw (X7)--(10.5,1.65);
 \coordinate [label=above:$x_3$] (a1) at (6.5,2);
 \coordinate [label=above:$x_2$] (a2) at (7,2);
 \coordinate [label=above:$x_3$] (a2) at (10,2);
 \coordinate [label=above:$x_2$] (a2) at (10.5,2);   
\end{tikzpicture}
\caption{The complexes and skeletons in Case $4. 1. 3$.}
\label{fc4.1.3}
\end{figure}
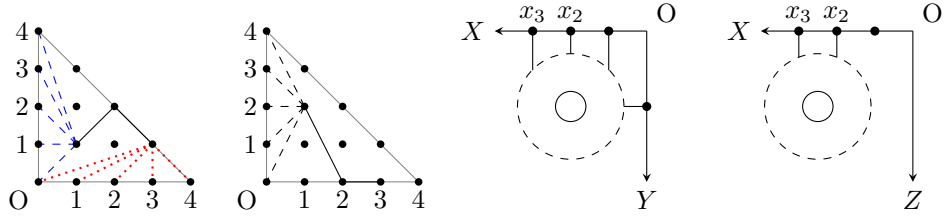

When we glue at $X$, we have the third cycle passing through the points $x_2$ and $x_3$.
Then any path connecting the cycle in $XY$ to the cycle in $XZ$ intersects the third cycle, so the skeleton of $C$ is not the lollipop graph.

\medbreak
Cases $4. 1. 4$ and $4. 1. 5$:

The complexes $K_{XY}$, $K_{XZ}$ and $K_{XW}$ are as in Figure \ref{fc4.1.4and4.1.5}.
In Case $4. 1. 4$, $b_1(\overline{C'})=1$.
In Case $4. 1. 5$, $K_{XZ}$ contains the edge $(1, 3)$-$(2, 0)$ as its $1$-simplex.
In each case, the skeleton of $C$ is not the lollipop graph by Lemma \ref{Case2.6} and \ref{case4hodai}.

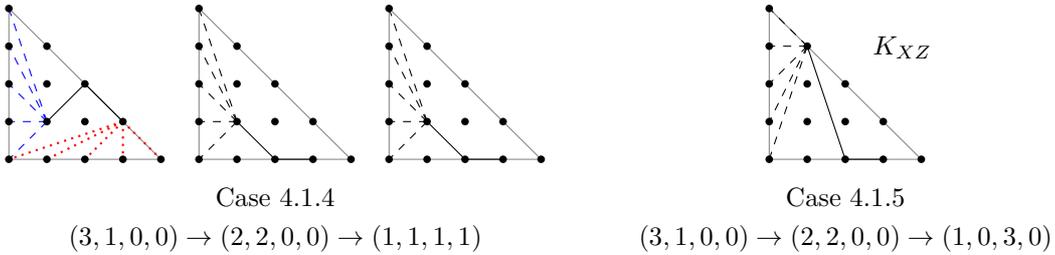
\begin{figure}[H]
\center
\begin{tikzpicture}
 \coordinate (A1) at (0,0);
 \coordinate (A2) at (0.5,0);
 \coordinate (A3) at (1,0);
 \coordinate (A4) at (1.5,0);
 \coordinate (A5) at (2,0);
 \coordinate (A6) at (0,0.5);
 \coordinate (A7) at (0.5,0.5);
 \coordinate (A8) at (1,0.5);
 \coordinate (A9) at (1.5,0.5);
 \coordinate (A10) at (0,1);
 \coordinate (A11) at (0.5,1);
 \coordinate (A12) at (1,1);
 \coordinate (A13) at (0,1.5);
 \coordinate (A14) at (0.5,1.5);
 \coordinate (A15) at (0,2);
 \draw [very thin, gray] (A1)--(A5)--(A15)--cycle;
 \draw (A9)--(A12)--(A7);
 \foreach \P in {1,6,10,13,15}  \draw[dashed, blue] (A7)--(A\P);
 \foreach \P in {1,2,3,4,5}  \draw[thick, dotted, red] (A9)--(A\P);
 \foreach \t in {1,2,...,15} \fill[black] (A\t) circle (0.05); 
 \coordinate (B1) at (2.5,0);
 \coordinate (B2) at (3,0);
 \coordinate (B3) at (3.5,0);
 \coordinate (B4) at (4,0);
 \coordinate (B5) at (4.5,0);
 \coordinate (B6) at (2.5,0.5);
 \coordinate (B7) at (3,0.5);
 \coordinate (B8) at (3.5,0.5);
 \coordinate (B9) at (4,0.5);
 \coordinate (B10) at (2.5,1);
 \coordinate (B11) at (3,1);
 \coordinate (B12) at (3.5,1);
 \coordinate (B13) at (2.5,1.5);
 \coordinate (B14) at (3,1.5);
 \coordinate (B15) at (2.5,2);
 \draw [very thin, gray] (B1)--(B5)--(B15)--cycle;
 \draw (B4)--(B3)--(B7);
 \foreach \t in {1,2,...,15} \fill[black] (B\t) circle (0.05);
 \foreach \P in {1,6,10,13,15}  \draw[dashed] (B7)--(B\P);
 \coordinate (C1) at (5,0);
 \coordinate (C2) at (5.5,0);
 \coordinate (C3) at (6,0);
 \coordinate (C4) at (6.5,0);
 \coordinate (C5) at (7,0);
 \coordinate (C6) at (5,0.5);
 \coordinate (C7) at (5.5,0.5);
 \coordinate (C8) at (6,0.5);
 \coordinate (C9) at (6.5,0.5);
 \coordinate (C10) at (5,1);
 \coordinate (C11) at (5.5,1);
 \coordinate (C12) at (6,1);
 \coordinate (C13) at (5,1.5);
 \coordinate (C14) at (5.5,1.5);
 \coordinate (C15) at (5,2);
 \draw [very thin, gray] (C1)--(C5)--(C15)--cycle;
 \draw (C4)--(C3)--(C7);
 \foreach \t in {1,2,...,15} \fill[black] (C\t) circle (0.05);
 \foreach \P in {1,6,10,13,15}  \draw[dashed] (C7)--(C\P);
 \coordinate (B1) at (10,0);
 \coordinate (B2) at (10.5,0);
 \coordinate (B3) at (11,0);
 \coordinate (B4) at (11.5,0);
 \coordinate (B5) at (12,0);
 \coordinate (B6) at (10,0.5);
 \coordinate (B7) at (10.5,0.5);
 \coordinate (B8) at (11,0.5);
 \coordinate (B9) at (11.5,0.5);
 \coordinate (B10) at (10,1);
 \coordinate (B11) at (10.5,1);
 \coordinate (B12) at (11,1);
 \coordinate (B13) at (10,1.5);
 \coordinate (B14) at (10.5,1.5);
 \coordinate (B15) at (10,2);
 \draw [very thin, gray] (B1)--(B5)--(B15)--cycle;
 \draw (B4)--(B3)--(B14);
 \foreach \t in {1,2,...,15} \fill[black] (B\t) circle (0.05);
 \foreach \P in {1,6,10,13,15}  \draw[dashed] (B14)--(B\P);
 \coordinate [label=below:$\text{Case $4. 1. 4$}$] (a1) at (3.5,-0.25);
 \coordinate [label=below:$\text{$(3, 1, 0, 0)\rightarrow (2, 2, 0, 0)\rightarrow (1, 1, 1, 1)$}$] (a1) at (3.5,-0.75);
 \coordinate [label=below:$\text{Case $4. 1. 5$}$] (a2) at (11,-0.25);
 \coordinate [label=below:$\text{$(3, 1, 0, 0)\rightarrow (2, 2, 0, 0)\rightarrow (1, 0, 3, 0)$}$] (a2) at (11,-0.75);
 \coordinate[label=below:$K_{XZ}$] (a3) at (11.75,1.75);
\end{tikzpicture}
\caption{The complexes in Cases $4. 1. 4$ and $4. 1. 5$.}
\label{fc4.1.4and4.1.5}
\end{figure}

Case $4. 1. 6$: $(3, 1, 0, 0)\rightarrow (2, 2, 0, 0)\rightarrow (1, 0, 2, 1)$

The complexes $K_{XY}$, $K_{XZ}$ and $K_{XW}$ are as in Figure \ref{fc4.1.6}.
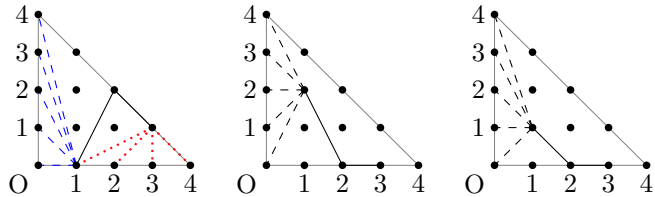
\begin{figure}[H]
\centering
\begin{tikzpicture}
 \coordinate[label=below left:O] (A1) at (0,0);
 \coordinate [label=below:1] (A2) at (0.5,0);
 \coordinate [label=below:2] (A3) at (1,0);
 \coordinate [label=below:3] (A4) at (1.5,0);
 \coordinate [label=below:4] (A5) at (2,0);
 \coordinate [label=left:1] (A6) at (0,0.5);
 \coordinate (A7) at (0.5,0.5);
 \coordinate (A8) at (1,0.5);
 \coordinate (A9) at (1.5,0.5);
 \coordinate [label=left:2] (A10) at (0,1);
 \coordinate (A11) at (0.5,1);
 \coordinate (A12) at (1,1);
 \coordinate [label=left:3] (A13) at (0,1.5);
 \coordinate (A14) at (0.5,1.5);
 \coordinate [label=left:4] (A15) at (0,2);
 \draw [very thin, gray] (A1)--(A5)--(A15)--cycle; 
 \draw (A9)--(A12)--(A2);
 \foreach \P in {1,6,10,13,15} \draw[dashed, blue] (A2)--(A\P);
 \foreach \P in {2,3,4,5} \draw[thick, dotted, red] (A9)--(A\P);
 \foreach \t in {1,2,...,15} \fill[black] (A\t) circle (0.05);
 \coordinate[label=below left:O] (B1) at (3,0);
 \coordinate [label=below:1] (B2) at (3.5,0);
 \coordinate [label=below:2] (B3) at (4,0);
 \coordinate [label=below:3] (B4) at (4.5,0);
 \coordinate [label=below:4] (B5) at (5,0);
 \coordinate [label=left:1] (B6) at (3,0.5);
 \coordinate (B7) at (3.5,0.5);
 \coordinate (B8) at (4,0.5);
 \coordinate (B9) at (4.5,0.5);
 \coordinate [label=left:2] (B10) at (3,1);
 \coordinate (B11) at (3.5,1);
 \coordinate (B12) at (4,1);
 \coordinate [label=left:3] (B13) at (3,1.5);
 \coordinate (B14) at (3.5,1.5);
 \coordinate [label=left:4] (B15) at (3,2);
 \draw [very thin, gray] (B1)--(B5)--(B15)--cycle; 
 \draw (B4)--(B3)--(B11);
 \foreach \t in {1,2,...,15} \fill[black] (B\t) circle (0.05);
 \foreach \P in {1,6,10,13,15}  \draw[dashed] (B11)--(B\P);
 \coordinate[label=below left:O] (C1) at (6,0);
 \coordinate [label=below:1] (C2) at (6.5,0);
 \coordinate [label=below:2] (C3) at (7,0);
 \coordinate [label=below:3] (C4) at (7.5,0);
 \coordinate [label=below:4] (C5) at (8,0);
 \coordinate [label=left:1] (C6) at (6,0.5);
 \coordinate (C7) at (6.5,0.5);
 \coordinate (C8) at (7,0.5);
 \coordinate (C9) at (7.5,0.5);
 \coordinate [label=left:2] (C10) at (6,1);
 \coordinate (C11) at (6.5,1);
 \coordinate (C12) at (7,1);
 \coordinate [label=left:3] (C13) at (6,1.5);
 \coordinate (C14) at (6.5,1.5);
 \coordinate [label=left:4] (C15) at (6,2);
 \draw [very thin, gray] (C1)--(C5)--(C15)--cycle; 
 \draw (C4)--(C3)--(C7);
 \foreach \t in {1,2,...,15} \fill[black] (C\t) circle (0.05);
 \foreach \P in {1,6,10,13,15}  \draw[dashed] (C7)--(C\P);
\end{tikzpicture}
\caption{The complexes in Case $4. 1. 6$.}
\label{fc4.1.6}
\end{figure}

The skeletons of $C_{XY}$, $C_{XZ}$ and $C_{XW}$ are as in Figure \ref{fc4.1.6s}.
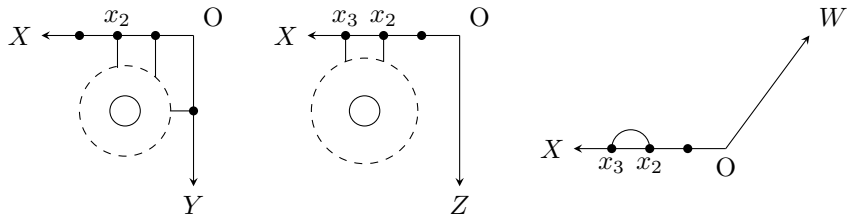
\begin{figure}[H]
\centering
\begin{tikzpicture}
 \coordinate[label=above right:O] (O1) at (-0.5,-1);
 \coordinate (X1) at (-2.5,-1);
 \coordinate (Y1) at (-0.5,-3);
 \coordinate (X2) at (-2,-1);
 \coordinate (X3) at (-1.5,-1);
 \coordinate (X4) at (-1,-1);
 \coordinate (Y2) at (-0.5,-2);
 \draw[thin,->,>=stealth] (O1)--(X1) node[left] {$X$};
 \draw[thin,->,>=stealth] (O1)--(Y1) node[below] {$Y$};
 \fill[black] (X2) circle (0.06);
 \fill[black] (X3) circle (0.06);
 \fill[black] (X4) circle (0.06);
 \fill[black] (Y2) circle (0.06);
 \draw[dashed] (-1.4,-2) circle (0.6);
 \draw (-1.4,-2) circle (0.2);
 \draw (X3)--(-1.5,-1.43);
 \draw (X4)--(-1,-1.55);
 \draw (Y2)--(-0.81,-2); 
 \coordinate[label=above right:O] (O2) at (3,-1);
 \coordinate (X5) at (1,-1);
 \coordinate (Z1) at (3,-3);
 \coordinate (X6) at (1.5,-1);
 \coordinate (X7) at (2,-1);
 \coordinate (X8) at (2.5,-1);
 \draw[thin,->,>=stealth] (O2)--(X5) node[left] {$X$};
 \draw[thin,->,>=stealth] (O2)--(Z1) node[below] {$Z$};
 \fill[black] (X6) circle (0.06);
 \fill[black] (X7) circle (0.06);
 \fill[black] (X8) circle (0.06);
 \draw[dashed] (1.75,-2) circle (0.7);
 \draw (1.75,-2) circle (0.2);
 \draw (X6)--(1.5,-1.35);
 \draw (X7)--(2,-1.35);
 \coordinate[label=below:O] (O3) at (6.5,-2.5);
 \coordinate (X9) at (4.5,-2.5);
 \coordinate (W1) at (7.6,-1);
 \coordinate (X10) at (5,-2.5);
 \coordinate (X11) at (5.5,-2.5);
 \coordinate (X12) at (6,-2.5);
 \draw[thin,->,>=stealth] (O3)--(X9) node[left] {$X$};
 \draw[thin,->,>=stealth] (O3)--(W1) node[above right] {$W$};
 \fill[black] (X10) circle (0.06);
 \fill[black] (X11) circle (0.06);
 \fill[black] (X12) circle (0.06); 
 \draw (X10) arc (180:0:0.25);
 \coordinate [label=above:$x_3$] (a1) at (1.5,-1);
 \coordinate [label=above:$x_2$] (a2) at (-1.5,-1);
 \coordinate [label=above:$x_2$] (a2) at (2,-1);
 \coordinate [label=below:$x_3$] (a1) at (5,-2.5);
 \coordinate [label=below:$x_2$] (a2) at (5.5,-2.5);   
\end{tikzpicture}
\caption{The skeletons of $C_{XY}$, $C_{XZ}$ and $C_{XW}$ in Case $4. 1. 6$.}
\label{fc4.1.6s}
\end{figure}

When we glue at $X$, we have the third cycle passing through the points $x_2$ and $x_3$.
Then any path connecting the cycle in $XY$ to the cycle in $XZ$ passes through the point $x_2$ and intersects the third cycle, so the skeleton of $C$ is not the lollipop graph.

From the above, in Case $4. 1$, the skeleton of $C$ is not the lollipop graph. 

\subsubsection{Case $4. 2$: $(3, 1, 0, 0)\rightarrow (2, 1, 1, 0)$}

We divide into the following cases.
\begin{eqnarray*}
&\text{Case $4. 2. 1$: }(2, 1, 1, 0)\rightarrow (1, 3, 0, 0),& \quad \text{Case $4. 2. 2$: }(2, 1, 1, 0)\rightarrow (1, 2, 1, 0),\\
&\text{Case $4. 2. 3$: }(2, 1, 1, 0)\rightarrow (1, 2, 0, 1),& \quad \text{Case $4. 2. 4$: }(2, 1, 1, 0)\rightarrow (1, 1, 2, 0),\\
&\text{Case $4. 2. 5$: }(2, 1, 1, 0)\rightarrow (1, 1, 1, 1),& \quad \text{Case $4. 2. 6$: }(2, 1, 1, 0)\rightarrow (1, 1, 0, 2),\\
&\text{Case $4. 2. 7$: }(2, 1, 1, 0)\rightarrow (1, 0, 3, 0),& \quad \text{Case $4. 2. 8$: }(2, 1, 1, 0)\rightarrow (1, 0, 2, 1),\\
&\text{Case $4. 2. 9$: }(2, 1, 1, 0)\rightarrow (1, 0, 1, 2),& \quad \text{Case $4. 2. 10$: }(2, 1, 1, 0)\rightarrow (1, 0, 0, 3).
\end{eqnarray*}

Case $4. 2. 1$: $(3, 1, 0, 0)\rightarrow (2, 1, 1, 0)\rightarrow (1, 3, 0, 0)$

The complexes $K_{XY}$ and $K_{XZ}$ are as in Figure \ref{fc4.2.1} (a).
Assume that the skeleton of $C$ is the lollipop graph of genus $3$.
Then by Lemma \ref{12}, $K_{XY}$ contains the $2$-simplices as in (b), and the skeletons of $C_{XY}$ and $C_{XZ}$ are as in (c).
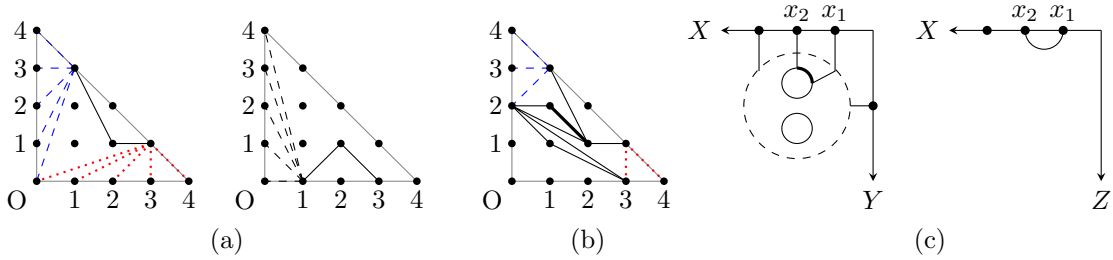
\begin{figure}[H]
\centering
\begin{tikzpicture}
 \coordinate[label=below left:O] (B1) at (3,0);
 \coordinate [label=below:1] (B2) at (3.5,0);
 \coordinate [label=below:2] (B3) at (4,0);
 \coordinate [label=below:3] (B4) at (4.5,0);
 \coordinate [label=below:4] (B5) at (5,0);
 \coordinate [label=left:1] (B6) at (3,0.5);
 \coordinate (B7) at (3.5,0.5);
 \coordinate (B8) at (4,0.5);
 \coordinate (B9) at (4.5,0.5);
 \coordinate [label=left:2] (B10) at (3,1);
 \coordinate (B11) at (3.5,1);
 \coordinate (B12) at (4,1);
 \coordinate [label=left:3] (B13) at (3,1.5);
 \coordinate (B14) at (3.5,1.5);
 \coordinate [label=left:4] (B15) at (3,2);
 \draw [very thin, gray] (B1)--(B5)--(B15)--cycle;
 \draw (B9)--(B8)--(B14);
 \foreach \P in {1,6,10,13,15} \draw[dashed, blue] (B14)--(B\P);
 \foreach \P in {1,2,3,4,5} \draw[thick, dotted, red] (B9)--(B\P);
 \foreach \t in {1,2,...,15} \fill[black] (B\t) circle (0.05);
 \coordinate[label=below left:O] (C1) at (6,0);
 \coordinate [label=below:1] (C2) at (6.5,0);
 \coordinate [label=below:2] (C3) at (7,0);
 \coordinate [label=below:3] (C4) at (7.5,0);
 \coordinate [label=below:4] (C5) at (8,0);
 \coordinate [label=left:1] (C6) at (6,0.5);
 \coordinate (C7) at (6.5,0.5);
 \coordinate (C8) at (7,0.5);
 \coordinate (C9) at (7.5,0.5);
 \coordinate [label=left:2] (C10) at (6,1);
 \coordinate (C11) at (6.5,1);
 \coordinate (C12) at (7,1);
 \coordinate [label=left:3] (C13) at (6,1.5);
 \coordinate (C14) at (6.5,1.5);
 \coordinate [label=left:4] (C15) at (6,2);
 \draw [very thin, gray] (C1)--(C5)--(C15)--cycle; 
 \draw (C4)--(C8)--(C2);
 \foreach \t in {1,2,...,15} \fill[black] (C\t) circle (0.05);
 \foreach \P in {1,6,10,13,15}  \draw[dashed] (C2)--(C\P);
 \coordinate[label=below left:O] (D1) at (9.25,0);
 \coordinate [label=below:1] (D2) at (9.75,0);
 \coordinate [label=below:2] (D3) at (10.25,0);
 \coordinate [label=below:3] (D4) at (10.75,0);
 \coordinate [label=below:4] (D5) at (11.25,0);
 \coordinate [label=left:1] (D6) at (9.25,0.5);
 \coordinate (D7) at (9.75,0.5);
 \coordinate (D8) at (10.25,0.5);
 \coordinate (D9) at (10.75,0.5);
 \coordinate [label=left:2] (D10) at (9.25,1);
 \coordinate (D11) at (9.75,1);
 \coordinate (D12) at (10.25,1);
 \coordinate [label=left:3] (D13) at (9.25,1.5);
 \coordinate (D14) at (9.75,1.5);
 \coordinate [label=left:4] (D15) at (9.25,2);
 \draw [very thin, gray] (D1)--(D5)--(D15)--cycle;
 \foreach \P in {10,13,15}  \draw[dashed, blue] (D14)--(D\P);
 \foreach \P in {4,5}  \draw[thick, dotted, red] (D9)--(D\P);
 \draw (D8)--(D14);
 \draw (D4)--(D10)--(D8);
 \draw (D4)--(D7)--(D10);
 \draw (D10)--(D11)--(D8);
 \draw (D8)--(D9);
 \draw[very thick] (D8)--(D11);
 \foreach \t in {1,2,...,15} \fill[black] (D\t) circle (0.05);
 \coordinate [label=below:(a)] (a1) at (5.5,-0.5);
 \coordinate [label=below:(b)] (a2) at (10.25,-0.5);
 \coordinate (O1) at (14,2);
 \coordinate (X1) at (12,2);
 \coordinate (Y1) at (14,0);
 \coordinate (X2) at (12.5,2);
 \coordinate (X3) at (13,2);
 \coordinate (X4) at (13.5,2);
 \coordinate (Y2) at (14,1);
 \draw[thin,->,>=stealth] (O1)--(X1) node[left] {$X$};
 \draw[thin,->,>=stealth] (O1)--(Y1) node[below] {$Y$};
 \fill[black] (X2) circle (0.06);
 \fill[black] (X3) circle (0.06);
 \fill[black] (X4) circle (0.06);
 \fill[black] (Y2) circle (0.06);
 \draw[dashed] (13,1) circle (0.7);
 \draw (13,0.7) circle (0.2);
 \draw (13,1.3) circle (0.2);
 \draw (X2)--(12.5,1.46);
 \draw (X3)--(13,1.5);
 \draw (X4)--(13.5,1.46)--(13.2,1.3);
 \draw (Y2)--(13.7,1);
 \draw[very thick] (13,1.5) arc (90:0:0.2);
 \coordinate (O2) at (17,2);
 \coordinate (X5) at (15,2);
 \coordinate (Z1) at (17,0);
 \coordinate (X6) at (15.5,2);
 \coordinate (X7) at (16,2);
 \coordinate (X8) at (16.5,2);
 \draw[thin,->,>=stealth] (O2)--(X5) node[left] {$X$};
 \draw[thin,->,>=stealth] (O2)--(Z1) node[below] {$Z$};
 \fill[black] (X6) circle (0.06);
 \fill[black] (X7) circle (0.06);
 \fill[black] (X8) circle (0.06);
 \draw (X7) arc (180:360:0.25);
 \coordinate [label=above:$x_2$] (a1) at (13,2);
 \coordinate [label=above:$x_1$] (a2) at (13.5,2); 
 \coordinate [label=above:$x_2$] (a1) at (16,2);
 \coordinate [label=above:$x_1$] (a2) at (16.5,2);
 \coordinate [label=below:(c)] (a3) at (14.75,-0.5); 
\end{tikzpicture}
\caption{The complexes and skeletons in Case $4. 2. 1$.}
\label{fc4.2.1}
\end{figure}

The path from $x_1$ to $x_2$ in $C_{XY}$ corresponding to the union of simplices of $K_{XY}$ which orbits around the point $(2, 1)$ intersects the cycle corresponding to the point $(1, 2)$.
When we glue at $X$, we have a cycle which contains the points $x_1$ and $x_2$.
Thus the two cycles intersect and the skeleton of $C$ is not the lollipop graph.

\medbreak
Cases $4. 2. 2$ to $4. 2. 6$:

The complexes $K_{XY}$, $K_{XZ}$ and $K_{XW}$ are as in Figure \ref{fc4.2.2to4.2.6}.
In each case, $b_1(\overline{C'})$ is at most $1$.
Hence the skeleton of $C$ is not the lollipop graph by Lemma \ref{saikurunashi} and \ref{case4hodai}.

\begin{figure}[H]
\centering
\begin{tikzpicture}
 \coordinate (A1) at (0,0);
 \coordinate (A2) at (0.5,0);
 \coordinate (A3) at (1,0);
 \coordinate (A4) at (1.5,0);
 \coordinate (A5) at (2,0);
 \coordinate (A6) at (0,0.5);
 \coordinate (A7) at (0.5,0.5);
 \coordinate (A8) at (1,0.5);
 \coordinate (A9) at (1.5,0.5);
 \coordinate (A10) at (0,1);
 \coordinate (A11) at (0.5,1);
 \coordinate (A12) at (1,1);
 \coordinate (A13) at (0,1.5);
 \coordinate (A14) at (0.5,1.5);
 \coordinate (A15) at (0,2);
 \draw [very thin, gray] (A1)--(A5)--(A15)--cycle;
 \draw (A9)--(A8)--(A11);
 \foreach \P in {1,6,10,13,15}  \draw[dashed, blue] (A11)--(A\P);
 \foreach \P in {1,2,3,4,5}  \draw[thick, dotted, red] (A9)--(A\P);
 \foreach \t in {1,2,...,15} \fill[black] (A\t) circle (0.05);
 \coordinate (B1) at (2.5,0);
 \coordinate (B2) at (3,0);
 \coordinate (B3) at (3.5,0);
 \coordinate (B4) at (4,0);
 \coordinate (B5) at (4.5,0);
 \coordinate (B6) at (2.5,0.5);
 \coordinate (B7) at (3,0.5);
 \coordinate (B8) at (3.5,0.5);
 \coordinate (B9) at (4,0.5);
 \coordinate (B10) at (2.5,1);
 \coordinate (B11) at (3,1);
 \coordinate (B12) at (3.5,1);
 \coordinate (B13) at (2.5,1.5);
 \coordinate (B14) at (3,1.5);
 \coordinate (B15) at (2.5,2);
 \draw [very thin, gray] (B1)--(B5)--(B15)--cycle;
 \draw (B4)--(B8)--(B7);
 \foreach \t in {1,2,...,15} \fill[black] (B\t) circle (0.05);
 \foreach \P in {1,6,10,13,15}  \draw[dashed] (B7)--(B\P);
 \coordinate (C1) at (5,0);
 \coordinate (C2) at (5.5,0);
 \coordinate (C3) at (6,0);
 \coordinate (C4) at (6.5,0);
 \coordinate (C5) at (7,0);
 \coordinate (C6) at (5,0.5);
 \coordinate (C7) at (5.5,0.5);
 \coordinate (C8) at (6,0.5);
 \coordinate (C9) at (6.5,0.5);
 \coordinate (C10) at (5,1);
 \coordinate (C11) at (5.5,1);
 \coordinate (C12) at (6,1);
 \coordinate (C13) at (5,1.5);
 \coordinate (C14) at (5.5,1.5);
 \coordinate (C15) at (5,2);
 \draw [very thin, gray] (C1)--(C5)--(C15)--cycle;
 \draw (C4)--(C3)--(C2);
 \foreach \t in {1,2,...,15} \fill[black] (C\t) circle (0.05);
 \foreach \P in {1,6,10,13,15}  \draw[dashed] (C2)--(C\P);
 \coordinate (A1) at (7.5,0);
 \coordinate (A2) at (8,0);
 \coordinate (A3) at (8.5,0);
 \coordinate (A4) at (9,0);
 \coordinate (A5) at (9.5,0);
 \coordinate (A6) at (7.5,0.5);
 \coordinate (A7) at (8,0.5);
 \coordinate (A8) at (8.5,0.5);
 \coordinate (A9) at (9,0.5);
 \coordinate (A10) at (7.5,1);
 \coordinate (A11) at (8,1);
 \coordinate (A12) at (8.5,1);
 \coordinate (A13) at (7.5,1.5);
 \coordinate (A14) at (8,1.5);
 \coordinate (A15) at (7.5,2);
 \draw [very thin, gray] (A1)--(A5)--(A15)--cycle;
 \draw (A9)--(A8)--(A11);
 \foreach \P in {1,6,10,13,15}  \draw[dashed, blue] (A11)--(A\P);
 \foreach \P in {1,2,3,4,5}  \draw[thick, dotted, red] (A9)--(A\P);
 \foreach \t in {1,2,...,15} \fill[black] (A\t) circle (0.05);
 \coordinate (B1) at (10,0);
 \coordinate (B2) at (10.5,0);
 \coordinate (B3) at (11,0);
 \coordinate (B4) at (11.5,0);
 \coordinate (B5) at (12,0);
 \coordinate (B6) at (10,0.5);
 \coordinate (B7) at (10.5,0.5);
 \coordinate (B8) at (11,0.5);
 \coordinate (B9) at (11.5,0.5);
 \coordinate (B10) at (10,1);
 \coordinate (B11) at (10.5,1);
 \coordinate (B12) at (11,1);
 \coordinate (B13) at (10,1.5);
 \coordinate (B14) at (10.5,1.5);
 \coordinate (B15) at (10,2);
 \draw [very thin, gray] (B1)--(B5)--(B15)--cycle;
 \draw (B4)--(B8)--(B2);
 \foreach \t in {1,2,...,15} \fill[black] (B\t) circle (0.05);
 \foreach \P in {1,6,10,13,15}  \draw[dashed] (B2)--(B\P);
 \coordinate (C1) at (12.5,0);
 \coordinate (C2) at (13,0);
 \coordinate (C3) at (13.5,0);
 \coordinate (C4) at (14,0);
 \coordinate (C5) at (14.5,0);
 \coordinate (C6) at (12.5,0.5);
 \coordinate (C7) at (13,0.5);
 \coordinate (C8) at (13.5,0.5);
 \coordinate (C9) at (14,0.5);
 \coordinate (C10) at (12.5,1);
 \coordinate (C11) at (13,1);
 \coordinate (C12) at (13.5,1);
 \coordinate (C13) at (12.5,1.5);
 \coordinate (C14) at (13,1.5);
 \coordinate (C15) at (12.5,2);
 \draw [very thin, gray] (C1)--(C5)--(C15)--cycle;
 \draw (C4)--(C3)--(C7);
 \foreach \t in {1,2,...,15} \fill[black] (C\t) circle (0.05);
 \foreach \P in {1,6,10,13,15}  \draw[dashed] (C7)--(C\P);
 \coordinate [label=below:$\text{Case $4. 2. 2$}$] (a1) at (3.5,-0.25);
 \coordinate [label=below:$\text{$(3, 1, 0, 0)\rightarrow (2, 1, 1, 0)\rightarrow (1, 2, 1, 0)$}$] (a1) at (3.5,-0.75);
 \coordinate [label=below:$\text{Case $4. 2. 3$}$] (a2) at (11,-0.25);
 \coordinate [label=below:$\text{$(3, 1, 0, 0)\rightarrow (2, 1, 1, 0)\rightarrow (1, 2, 0, 1)$}$] (a2) at (11,-0.75);
\end{tikzpicture}
\begin{tikzpicture}
 \coordinate (A1) at (0,0);
 \coordinate (A2) at (0.5,0);
 \coordinate (A3) at (1,0);
 \coordinate (A4) at (1.5,0);
 \coordinate (A5) at (2,0);
 \coordinate (A6) at (0,0.5);
 \coordinate (A7) at (0.5,0.5);
 \coordinate (A8) at (1,0.5);
 \coordinate (A9) at (1.5,0.5);
 \coordinate (A10) at (0,1);
 \coordinate (A11) at (0.5,1);
 \coordinate (A12) at (1,1);
 \coordinate (A13) at (0,1.5);
 \coordinate (A14) at (0.5,1.5);
 \coordinate (A15) at (0,2);
 \draw [very thin, gray] (A1)--(A5)--(A15)--cycle;
 \draw (A9)--(A8)--(A7);
 \foreach \P in {1,6,10,13,15}  \draw[dashed, blue] (A7)--(A\P);
 \foreach \P in {1,2,3,4,5}  \draw[thick, dotted, red] (A9)--(A\P);
 \foreach \t in {1,2,...,15} \fill[black] (A\t) circle (0.05);
 \coordinate (B1) at (2.5,0);
 \coordinate (B2) at (3,0);
 \coordinate (B3) at (3.5,0);
 \coordinate (B4) at (4,0);
 \coordinate (B5) at (4.5,0);
 \coordinate (B6) at (2.5,0.5);
 \coordinate (B7) at (3,0.5);
 \coordinate (B8) at (3.5,0.5);
 \coordinate (B9) at (4,0.5);
 \coordinate (B10) at (2.5,1);
 \coordinate (B11) at (3,1);
 \coordinate (B12) at (3.5,1);
 \coordinate (B13) at (2.5,1.5);
 \coordinate (B14) at (3,1.5);
 \coordinate (B15) at (2.5,2);
 \draw [very thin, gray] (B1)--(B5)--(B15)--cycle;
 \draw (B4)--(B8)--(B11);
 \foreach \t in {1,2,...,15} \fill[black] (B\t) circle (0.05);
 \foreach \P in {1,6,10,13,15}  \draw[dashed] (B11)--(B\P);
 \coordinate (C1) at (5,0);
 \coordinate (C2) at (5.5,0);
 \coordinate (C3) at (6,0);
 \coordinate (C4) at (6.5,0);
 \coordinate (C5) at (7,0);
 \coordinate (C6) at (5,0.5);
 \coordinate (C7) at (5.5,0.5);
 \coordinate (C8) at (6,0.5);
 \coordinate (C9) at (6.5,0.5);
 \coordinate (C10) at (5,1);
 \coordinate (C11) at (5.5,1);
 \coordinate (C12) at (6,1);
 \coordinate (C13) at (5,1.5);
 \coordinate (C14) at (5.5,1.5);
 \coordinate (C15) at (5,2);
 \draw [very thin, gray] (C1)--(C5)--(C15)--cycle;
 \draw (C4)--(C3)--(C2);
 \foreach \t in {1,2,...,15} \fill[black] (C\t) circle (0.05);
 \foreach \P in {1,6,10,13,15}  \draw[dashed] (C2)--(C\P);
 \coordinate (A1) at (7.5,0);
 \coordinate (A2) at (8,0);
 \coordinate (A3) at (8.5,0);
 \coordinate (A4) at (9,0);
 \coordinate (A5) at (9.5,0);
 \coordinate (A6) at (7.5,0.5);
 \coordinate (A7) at (8,0.5);
 \coordinate (A8) at (8.5,0.5);
 \coordinate (A9) at (9,0.5);
 \coordinate (A10) at (7.5,1);
 \coordinate (A11) at (8,1);
 \coordinate (A12) at (8.5,1);
 \coordinate (A13) at (7.5,1.5);
 \coordinate (A14) at (8,1.5);
 \coordinate (A15) at (7.5,2);
 \draw [very thin, gray] (A1)--(A5)--(A15)--cycle;
 \draw (A9)--(A8)--(A7);
 \foreach \P in {1,6,10,13,15}  \draw[dashed, blue] (A7)--(A\P);
 \foreach \P in {1,2,3,4,5}  \draw[thick, dotted, red] (A9)--(A\P);
 \foreach \t in {1,2,...,15} \fill[black] (A\t) circle (0.05);
 \coordinate (B1) at (10,0);
 \coordinate (B2) at (10.5,0);
 \coordinate (B3) at (11,0);
 \coordinate (B4) at (11.5,0);
 \coordinate (B5) at (12,0);
 \coordinate (B6) at (10,0.5);
 \coordinate (B7) at (10.5,0.5);
 \coordinate (B8) at (11,0.5);
 \coordinate (B9) at (11.5,0.5);
 \coordinate (B10) at (10,1);
 \coordinate (B11) at (10.5,1);
 \coordinate (B12) at (11,1);
 \coordinate (B13) at (10,1.5);
 \coordinate (B14) at (10.5,1.5);
 \coordinate (B15) at (10,2);
 \draw [very thin, gray] (B1)--(B5)--(B15)--cycle;
 \draw (B4)--(B8)--(B7);
 \foreach \t in {1,2,...,15} \fill[black] (B\t) circle (0.05);
 \foreach \P in {1,6,10,13,15}  \draw[dashed] (B7)--(B\P);
 \coordinate (C1) at (12.5,0);
 \coordinate (C2) at (13,0);
 \coordinate (C3) at (13.5,0);
 \coordinate (C4) at (14,0);
 \coordinate (C5) at (14.5,0);
 \coordinate (C6) at (12.5,0.5);
 \coordinate (C7) at (13,0.5);
 \coordinate (C8) at (13.5,0.5);
 \coordinate (C9) at (14,0.5);
 \coordinate (C10) at (12.5,1);
 \coordinate (C11) at (13,1);
 \coordinate (C12) at (13.5,1);
 \coordinate (C13) at (12.5,1.5);
 \coordinate (C14) at (13,1.5);
 \coordinate (C15) at (12.5,2);
 \draw [very thin, gray] (C1)--(C5)--(C15)--cycle;
 \draw (C4)--(C3)--(C7);
 \foreach \t in {1,2,...,15} \fill[black] (C\t) circle (0.05);
 \foreach \P in {1,6,10,13,15}  \draw[dashed] (C7)--(C\P);
 \coordinate [label=below:$\text{Case $4. 2. 4$}$] (a1) at (3.5,-0.25);
 \coordinate [label=below:$\text{$(3, 1, 0, 0)\rightarrow (2, 1, 1, 0)\rightarrow (1, 1, 2, 0)$}$] (a1) at (3.5,-0.75);
 \coordinate [label=below:$\text{Case $4. 2. 5$}$] (a2) at (11,-0.25);
 \coordinate [label=below:$\text{$(3, 1, 0, 0)\rightarrow (2, 1, 1, 0)\rightarrow (1, 1, 1, 1)$}$] (a2) at (11,-0.75);
\end{tikzpicture}
\begin{tikzpicture}
 \coordinate (A1) at (0,0);
 \coordinate (A2) at (0.5,0);
 \coordinate (A3) at (1,0);
 \coordinate (A4) at (1.5,0);
 \coordinate (A5) at (2,0);
 \coordinate (A6) at (0,0.5);
 \coordinate (A7) at (0.5,0.5);
 \coordinate (A8) at (1,0.5);
 \coordinate (A9) at (1.5,0.5);
 \coordinate (A10) at (0,1);
 \coordinate (A11) at (0.5,1);
 \coordinate (A12) at (1,1);
 \coordinate (A13) at (0,1.5);
 \coordinate (A14) at (0.5,1.5);
 \coordinate (A15) at (0,2);
 \draw [very thin, gray] (A1)--(A5)--(A15)--cycle;
 \draw (A9)--(A8)--(A7);
 \foreach \P in {1,6,10,13,15}  \draw[dashed, blue] (A7)--(A\P);
 \foreach \P in {1,2,3,4,5}  \draw[thick, dotted, red] (A9)--(A\P);
 \foreach \t in {1,2,...,15} \fill[black] (A\t) circle (0.05);
 \coordinate (B1) at (2.5,0);
 \coordinate (B2) at (3,0);
 \coordinate (B3) at (3.5,0);
 \coordinate (B4) at (4,0);
 \coordinate (B5) at (4.5,0);
 \coordinate (B6) at (2.5,0.5);
 \coordinate (B7) at (3,0.5);
 \coordinate (B8) at (3.5,0.5);
 \coordinate (B9) at (4,0.5);
 \coordinate (B10) at (2.5,1);
 \coordinate (B11) at (3,1);
 \coordinate (B12) at (3.5,1);
 \coordinate (B13) at (2.5,1.5);
 \coordinate (B14) at (3,1.5);
 \coordinate (B15) at (2.5,2);
 \draw [very thin, gray] (B1)--(B5)--(B15)--cycle;
 \draw (B4)--(B8)--(B2);
 \foreach \t in {1,2,...,15} \fill[black] (B\t) circle (0.05);
 \foreach \P in {1,6,10,13,15}  \draw[dashed] (B2)--(B\P);
 \coordinate (C1) at (5,0);
 \coordinate (C2) at (5.5,0);
 \coordinate (C3) at (6,0);
 \coordinate (C4) at (6.5,0);
 \coordinate (C5) at (7,0);
 \coordinate (C6) at (5,0.5);
 \coordinate (C7) at (5.5,0.5);
 \coordinate (C8) at (6,0.5);
 \coordinate (C9) at (6.5,0.5);
 \coordinate (C10) at (5,1);
 \coordinate (C11) at (5.5,1);
 \coordinate (C12) at (6,1);
 \coordinate (C13) at (5,1.5);
 \coordinate (C14) at (5.5,1.5);
 \coordinate (C15) at (5,2);
 \draw [very thin, gray] (C1)--(C5)--(C15)--cycle;
 \draw (C4)--(C3)--(C11);
 \foreach \t in {1,2,...,15} \fill[black] (C\t) circle (0.05);
 \foreach \P in {1,6,10,13,15}  \draw[dashed] (C11)--(C\P);
 \coordinate [label=below:$\text{Case $4. 2. 6$: $(3, 1, 0, 0)\rightarrow (2, 1, 1, 0)\rightarrow (1, 1, 0, 2)$}$] (a1) at (3.5,-0.25);
\end{tikzpicture}
\caption{The complexes in Cases $4. 2. 2$ to $4. 2. 6$.}
\label{fc4.2.2to4.2.6}
\end{figure}
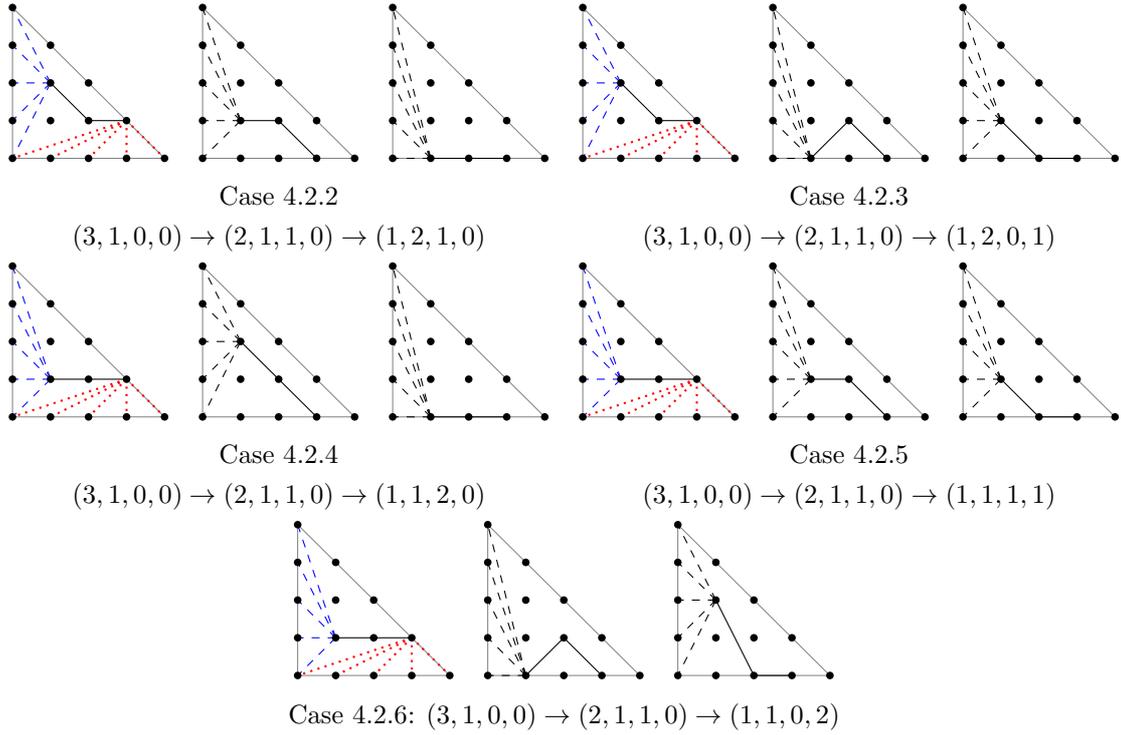

Case $4. 2. 7$: $(3, 1, 0, 0)\rightarrow (2, 1, 1, 0)\rightarrow (1, 0, 3, 0)$

The complexes $K_{XY}$ and $K_{XZ}$ are as in Figure \ref{fc4.2.7} (a).
Assume that the skeleton of $C$ is the lollipop graph of genus $3$.
Then by Lemma \ref{12}, $K_{XZ}$ contains the $2$-simplices as in (b), and the skeletons of $C_{XY}$ and $C_{XZ}$ are as in (c).
\begin{figure}[H]
\centering
\begin{tikzpicture}
 \coordinate[label=below left:O] (B1) at (3,0);
 \coordinate [label=below:1] (B2) at (3.5,0);
 \coordinate [label=below:2] (B3) at (4,0);
 \coordinate [label=below:3] (B4) at (4.5,0);
 \coordinate [label=below:4] (B5) at (5,0);
 \coordinate [label=left:1] (B6) at (3,0.5);
 \coordinate (B7) at (3.5,0.5);
 \coordinate (B8) at (4,0.5);
 \coordinate (B9) at (4.5,0.5);
 \coordinate [label=left:2] (B10) at (3,1);
 \coordinate (B11) at (3.5,1);
 \coordinate (B12) at (4,1);
 \coordinate [label=left:3] (B13) at (3,1.5);
 \coordinate (B14) at (3.5,1.5);
 \coordinate [label=left:4] (B15) at (3,2);
 \draw [very thin, gray] (B1)--(B5)--(B15)--cycle;
 \draw (B9)--(B8)--(B2);
 \foreach \P in {1,6,10,13,15}  \draw[dashed, blue] (B2)--(B\P);
 \foreach \P in {2,3,4,5}  \draw[thick, dotted, red] (B9)--(B\P);
 \foreach \t in {1,2,...,15} \fill[black] (B\t) circle (0.05);
 \coordinate[label=below left:O] (C1) at (6,0);
 \coordinate [label=below:1] (C2) at (6.5,0);
 \coordinate [label=below:2] (C3) at (7,0);
 \coordinate [label=below:3] (C4) at (7.5,0);
 \coordinate [label=below:4] (C5) at (8,0);
 \coordinate [label=left:1] (C6) at (6,0.5);
 \coordinate (C7) at (6.5,0.5);
 \coordinate (C8) at (7,0.5);
 \coordinate (C9) at (7.5,0.5);
 \coordinate [label=left:2] (C10) at (6,1);
 \coordinate (C11) at (6.5,1);
 \coordinate (C12) at (7,1);
 \coordinate [label=left:3] (C13) at (6,1.5);
 \coordinate (C14) at (6.5,1.5);
 \coordinate [label=left:4] (C15) at (6,2);
 \draw [very thin, gray] (C1)--(C5)--(C15)--cycle;
 \draw (C4)--(C8)--(C14);
 \foreach \t in {1,2,...,15} \fill[black] (C\t) circle (0.05);
 \foreach \P in {1,6,10,13,15}  \draw[dashed] (C14)--(C\P);
 \coordinate[label=below left:O] (D1) at (9.25,0);
 \coordinate [label=below:1] (D2) at (9.75,0);
 \coordinate [label=below:2] (D3) at (10.25,0);
 \coordinate [label=below:3] (D4) at (10.75,0);
 \coordinate [label=below:4] (D5) at (11.25,0);
 \coordinate [label=left:1] (D6) at (9.25,0.5);
 \coordinate (D7) at (9.75,0.5);
 \coordinate (D8) at (10.25,0.5);
 \coordinate (D9) at (10.75,0.5);
 \coordinate [label=left:2] (D10) at (9.25,1);
 \coordinate (D11) at (9.75,1);
 \coordinate (D12) at (10.25,1);
 \coordinate [label=left:3] (D13) at (9.25,1.5);
 \coordinate (D14) at (9.75,1.5);
 \coordinate [label=left:4] (D15) at (9.25,2);
 \draw [very thin, gray] (D1)--(D5)--(D15)--cycle;
 \foreach \P in {10,13,15}  \draw[dashed] (D14)--(D\P);
 \draw (D8)--(D14);
 \draw (D4)--(D10)--(D8)--cycle;
 \draw (D4)--(D7)--(D10);
 \draw (D10)--(D11)--(D8);
 \draw[very thick] (D8)--(D11);
 \foreach \t in {1,2,...,15} \fill[black] (D\t) circle (0.05);
 \coordinate [label=below:(a)] (a1) at (5.5,-0.5);
 \coordinate [label=below:(b)] (a2) at (10.25,-0.5);
 \coordinate (O1) at (17,2);
 \coordinate (X1) at (15,2);
 \coordinate (Y1) at (17,0);
 \coordinate (X2) at (15.5,2);
 \coordinate (X3) at (16,2);
 \coordinate (X4) at (16.5,2);
 \coordinate (Y2) at (17,1);
 \draw[thin,->,>=stealth] (O1)--(X1) node[left] {$X$};
 \draw[thin,->,>=stealth] (O1)--(Y1) node[below] {$Z$};
 \fill[black] (X2) circle (0.06);
 \fill[black] (X3) circle (0.06);
 \fill[black] (X4) circle (0.06);
 \draw[dashed] (16,1) circle (0.6);
 \draw (16,0.75) circle (0.2);
 \draw (16,1.25) circle (0.2);
 \draw (X2)--(15.5,1.34);
 \draw (X3)--(16,1.45);
 \draw (16.18,1.34)--(16.5,1.34)--(X4);
 \coordinate (O2) at (14,2);
 \coordinate (X5) at (12,2);
 \coordinate (Z1) at (14,0);
 \coordinate (Z2) at (14,1);
 \coordinate (X6) at (12.5,2);
 \coordinate (X7) at (13,2);
 \coordinate (X8) at (13.5,2);
 \draw[thin,->,>=stealth] (O2)--(X5) node[left] {$X$};
 \draw[thin,->,>=stealth] (O2)--(Z1) node[below] {$Y$};
 \fill[black] (X6) circle (0.06);
 \fill[black] (X7) circle (0.06);
 \fill[black] (X8) circle (0.06);
 \fill[black] (Z2) circle (0.06);
 \draw (X7)--(13.25,1.5)--(X8);
 \draw (Z2)--(13.25,1.5);
 \draw[very thick] (16,1.45) arc (90:25:0.2);
 \coordinate [label=above:$x_1$] (a1) at (16.5,2);
 \coordinate [label=above:$x_2$] (a2) at (16,2); 
 \coordinate [label=above:$x_1$] (a1) at (13.5,2);
 \coordinate [label=above:$x_2$] (a2) at (13,2); 
 \coordinate [label=below:(c)] (a2) at (14.75,-0.5); 
\end{tikzpicture}
\caption{The complexes and skeletons in Case $4. 2. 7$.}
\label{fc4.2.7}
\end{figure}
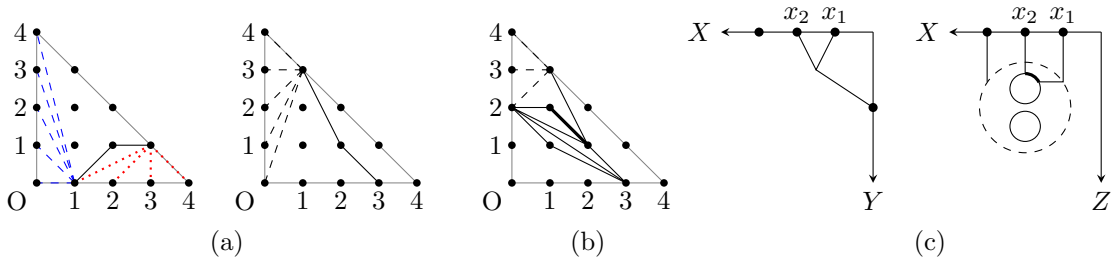

The path from $x_1$ to $x_2$ in $C_{XZ}$ corresponding to the union of simplices of $K_{XY}$ which orbits around the point $(2, 1)$ intersects the cycle corresponding to the point $(1, 2)$.
When we glue at $X$, we have a cycle which contains the points $x_1$ and $x_2$.
Thus the two cycles intersect and the skeleton of $C$ is not the lollipop graph.

\medbreak
Cases $4. 2. 8$ to $4. 2. 10$:

The complexes $K_{XY}$, $K_{XZ}$ and $K_{XW}$ are as in Figure \ref{fc4.2.8to4.2.10}.
In Cases $4. 2. 8$ and $4. 2. 9$, $b_1(\overline{C'})=1$ and the skeleton of $C$ is not the lollipop graph by Lemma \ref{case4hodai}.
In Case $4. 2. 10$, $K_{XW}$ contains the edge $(1, 3)$-$(2, 0)$ as its $1$-simplex and the skeleton of $C$ is not the lollipop graph by Lemma \ref{Case2.6}.
\begin{figure}[H]
\centering
\begin{tikzpicture}
 \coordinate (A1) at (0,0);
 \coordinate (A2) at (0.5,0);
 \coordinate (A3) at (1,0);
 \coordinate (A4) at (1.5,0);
 \coordinate (A5) at (2,0);
 \coordinate (A6) at (0,0.5);
 \coordinate (A7) at (0.5,0.5);
 \coordinate (A8) at (1,0.5);
 \coordinate (A9) at (1.5,0.5);
 \coordinate (A10) at (0,1);
 \coordinate (A11) at (0.5,1);
 \coordinate (A12) at (1,1);
 \coordinate (A13) at (0,1.5);
 \coordinate (A14) at (0.5,1.5);
 \coordinate (A15) at (0,2);
 \draw [very thin, gray] (A1)--(A5)--(A15)--cycle;
 \draw (A9)--(A8)--(A2);
 \foreach \P in {1,6,10,13,15}  \draw[dashed, blue] (A2)--(A\P);
 \foreach \P in {2,3,4,5}  \draw[thick, dotted, red] (A9)--(A\P);
 \foreach \t in {1,2,...,15} \fill[black] (A\t) circle (0.05);
 \coordinate (B1) at (2.5,0);
 \coordinate (B2) at (3,0);
 \coordinate (B3) at (3.5,0);
 \coordinate (B4) at (4,0);
 \coordinate (B5) at (4.5,0);
 \coordinate (B6) at (2.5,0.5);
 \coordinate (B7) at (3,0.5);
 \coordinate (B8) at (3.5,0.5);
 \coordinate (B9) at (4,0.5);
 \coordinate (B10) at (2.5,1);
 \coordinate (B11) at (3,1);
 \coordinate (B12) at (3.5,1);
 \coordinate (B13) at (2.5,1.5);
 \coordinate (B14) at (3,1.5);
 \coordinate (B15) at (2.5,2);
 \draw [very thin, gray] (B1)--(B5)--(B15)--cycle;
 \draw (B4)--(B8)--(B11);
 \foreach \t in {1,2,...,15} \fill[black] (B\t) circle (0.05);
 \foreach \P in {1,6,10,13,15}  \draw[dashed] (B11)--(B\P);
 \coordinate (C1) at (5,0);
 \coordinate (C2) at (5.5,0);
 \coordinate (C3) at (6,0);
 \coordinate (C4) at (6.5,0);
 \coordinate (C5) at (7,0);
 \coordinate (C6) at (5,0.5);
 \coordinate (C7) at (5.5,0.5);
 \coordinate (C8) at (6,0.5);
 \coordinate (C9) at (6.5,0.5);
 \coordinate (C10) at (5,1);
 \coordinate (C11) at (5.5,1);
 \coordinate (C12) at (6,1);
 \coordinate (C13) at (5,1.5);
 \coordinate (C14) at (5.5,1.5);
 \coordinate (C15) at (5,2);
 \draw [very thin, gray] (C1)--(C5)--(C15)--cycle;
 \draw (C4)--(C3)--(C7);
 \foreach \t in {1,2,...,15} \fill[black] (C\t) circle (0.05);
 \foreach \P in {1,6,10,13,15}  \draw[dashed] (C7)--(C\P);
 \coordinate (A1) at (7.5,0);
 \coordinate (A2) at (8,0);
 \coordinate (A3) at (8.5,0);
 \coordinate (A4) at (9,0);
 \coordinate (A5) at (9.5,0);
 \coordinate (A6) at (7.5,0.5);
 \coordinate (A7) at (8,0.5);
 \coordinate (A8) at (8.5,0.5);
 \coordinate (A9) at (9,0.5);
 \coordinate (A10) at (7.5,1);
 \coordinate (A11) at (8,1);
 \coordinate (A12) at (8.5,1);
 \coordinate (A13) at (7.5,1.5);
 \coordinate (A14) at (8,1.5);
 \coordinate (A15) at (7.5,2);
 \draw [very thin, gray] (A1)--(A5)--(A15)--cycle;
 \draw (A9)--(A8)--(A2);
 \foreach \P in {1,6,10,13,15}  \draw[dashed, blue] (A2)--(A\P);
 \foreach \P in {2,3,4,5}  \draw[thick, dotted, red] (A9)--(A\P);
 \foreach \t in {1,2,...,15} \fill[black] (A\t) circle (0.05);
 \coordinate (B1) at (10,0);
 \coordinate (B2) at (10.5,0);
 \coordinate (B3) at (11,0);
 \coordinate (B4) at (11.5,0);
 \coordinate (B5) at (12,0);
 \coordinate (B6) at (10,0.5);
 \coordinate (B7) at (10.5,0.5);
 \coordinate (B8) at (11,0.5);
 \coordinate (B9) at (11.5,0.5);
 \coordinate (B10) at (10,1);
 \coordinate (B11) at (10.5,1);
 \coordinate (B12) at (11,1);
 \coordinate (B13) at (10,1.5);
 \coordinate (B14) at (10.5,1.5);
 \coordinate (B15) at (10,2);
 \draw [very thin, gray] (B1)--(B5)--(B15)--cycle;
 \draw (B4)--(B8)--(B7);
 \foreach \t in {1,2,...,15} \fill[black] (B\t) circle (0.05);
 \foreach \P in {1,6,10,13,15}  \draw[dashed] (B7)--(B\P);
 \coordinate (C1) at (12.5,0);
 \coordinate (C2) at (13,0);
 \coordinate (C3) at (13.5,0);
 \coordinate (C4) at (14,0);
 \coordinate (C5) at (14.5,0);
 \coordinate (C6) at (12.5,0.5);
 \coordinate (C7) at (13,0.5);
 \coordinate (C8) at (13.5,0.5);
 \coordinate (C9) at (14,0.5);
 \coordinate (C10) at (12.5,1);
 \coordinate (C11) at (13,1);
 \coordinate (C12) at (13.5,1);
 \coordinate (C13) at (12.5,1.5);
 \coordinate (C14) at (13,1.5);
 \coordinate (C15) at (12.5,2);
 \draw [very thin, gray] (C1)--(C5)--(C15)--cycle;
 \draw (C4)--(C3)--(C11);
 \foreach \t in {1,2,...,15} \fill[black] (C\t) circle (0.05);
 \foreach \P in {1,6,10,13,15}  \draw[dashed] (C11)--(C\P);
 \coordinate [label=below:$\text{Case $4. 2. 8$}$] (a1) at (3.5,-0.25);
 \coordinate [label=below:$\text{$(3, 1, 0, 0)\rightarrow (2, 1, 1, 0)\rightarrow (1, 0, 2, 1)$}$] (a1) at (3.5,-0.75);
 \coordinate [label=below:$\text{Case $4. 2. 9$}$] (a2) at (11,-0.25);
 \coordinate [label=below:$\text{$(3, 1, 0, 0)\rightarrow (2, 1, 1, 0)\rightarrow (1, 0, 1, 2)$}$] (a2) at (11,-0.75);
\end{tikzpicture}
\begin{tikzpicture}
 \coordinate (B1) at (2.5,0);
 \coordinate (B2) at (3,0);
 \coordinate (B3) at (3.5,0);
 \coordinate (B4) at (4,0);
 \coordinate (B5) at (4.5,0);
 \coordinate (B6) at (2.5,0.5);
 \coordinate (B7) at (3,0.5);
 \coordinate (B8) at (3.5,0.5);
 \coordinate (B9) at (4,0.5);
 \coordinate (B10) at (2.5,1);
 \coordinate (B11) at (3,1);
 \coordinate (B12) at (3.5,1);
 \coordinate (B13) at (2.5,1.5);
 \coordinate (B14) at (3,1.5);
 \coordinate (B15) at (2.5,2);
 \draw [very thin, gray] (B1)--(B5)--(B15)--cycle;
 \draw (B4)--(B3)--(B14);
 \foreach \P in {1,6,10,13,15}  \draw[dashed] (B14)--(B\P);
 \foreach \t in {1,2,...,15} \fill[black] (B\t) circle (0.05);
 \coordinate [label=below:$\text{Case $4. 2. 10$: $(3, 1, 0, 0)\rightarrow (2, 1, 1, 0)\rightarrow (1, 0, 0, 3)$}$] (a1) at (3.5,-0.25);
 \coordinate[label=below:$K_{XW}$] (a3) at (4.25,1.75);
\end{tikzpicture}
\caption{The complexes in Cases $4. 2. 8$ to $4. 2. 10$.}
\label{fc4.2.8to4.2.10}
\end{figure}
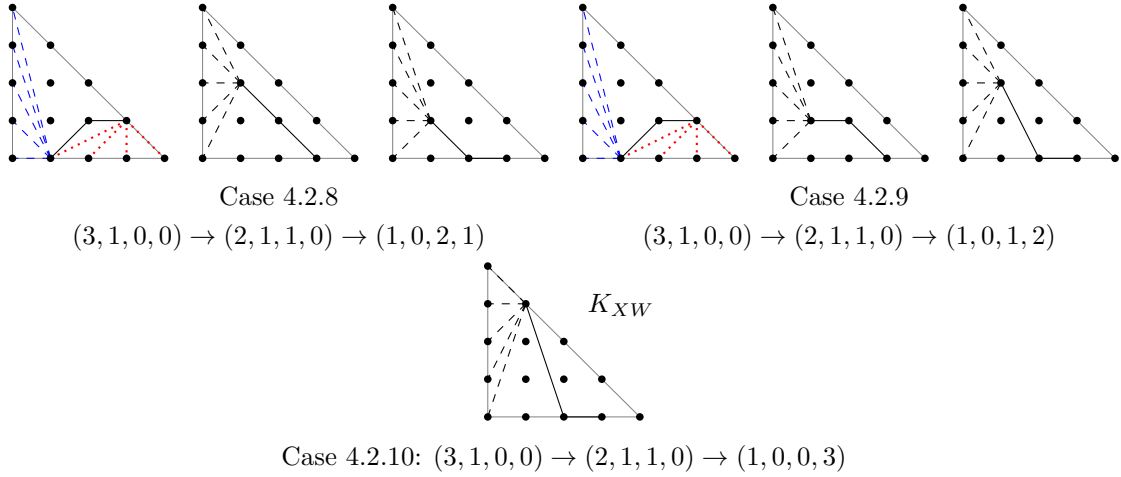

From the above, in Case $4. 2$, the skeleton of $C$ is not the lollipop graph.

\subsubsection{Case $4. 3$: $(3, 1, 0, 0)\rightarrow (2, 0, 2, 0)$}

We divide into the following cases.
\begin{eqnarray*}
&\text{Case $4. 3. 1$: }(2, 0, 2, 0)\rightarrow (1, 3, 0, 0),& \quad \text{Case $4. 3. 2$: }(2, 0, 2, 0)\rightarrow (1, 2, 1, 0),\\
&\text{Case $4. 3. 3$: }(2, 0, 2, 0)\rightarrow (1, 2, 0, 1),& \quad \text{Case $4. 3. 4$: }(2, 0, 2, 0)\rightarrow (1, 1, 2, 0),\\
&\text{Case $4. 3. 5$: }(2, 0, 2, 0)\rightarrow (1, 1, 1, 1),& \quad \text{Case $4. 3. 6$: }(2, 0, 2, 0)\rightarrow (1, 1, 0, 2),\\
&\text{Case $4. 3. 7$: }(2, 0, 2, 0)\rightarrow (1, 0, 3, 0),& \quad \text{Case $4. 3. 8$: }(2, 0, 2, 0)\rightarrow (1, 0, 2, 1),\\
&\text{Case $4. 3. 9$: }(2, 0, 2, 0)\rightarrow (1, 0, 1, 2),& \quad \text{Case $4. 3. 10$: }(2, 0, 2, 0)\rightarrow (1, 0, 0, 3).
\end{eqnarray*}

Case $4. 3. 1$: $(3, 1, 0, 0)\rightarrow (2, 0, 2, 0)\rightarrow (1, 3, 0, 0)$

Since the complex $K_{XY}$ contains the edge $(2, 0)$-$(1, 3)$ as its $1$-simplex and the points $(1, 1)$ and $(1, 2)$ in its interior, the skeleton of $C$ is not the lollipop graph by Lemma \ref{Case2.6}.
\medbreak
Case $4. 3. 2$: $(3, 1, 0, 0)\rightarrow (2, 0, 2, 0)\rightarrow (1, 2, 1, 0)$

The complexes $K_{XY}$, $K_{XZ}$ and the skeletons of $C_{XY}$ and $C_{XZ}$ are as in Figure \ref{fc4.3.2}.
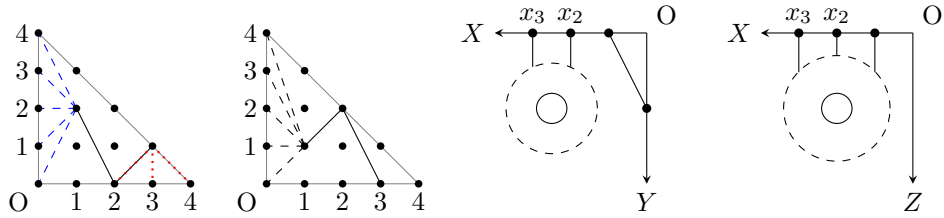
\begin{figure}[H]
\centering
\begin{tikzpicture}
 \coordinate[label=below left:O] (A1) at (0,0);
 \coordinate [label=below:1] (A2) at (0.5,0);
 \coordinate [label=below:2] (A3) at (1,0);
 \coordinate [label=below:3] (A4) at (1.5,0);
 \coordinate [label=below:4] (A5) at (2,0);
 \coordinate [label=left:1] (A6) at (0,0.5);
 \coordinate (A7) at (0.5,0.5);
 \coordinate (A8) at (1,0.5);
 \coordinate (A9) at (1.5,0.5);
 \coordinate [label=left:2] (A10) at (0,1);
 \coordinate (A11) at (0.5,1);
 \coordinate (A12) at (1,1);
 \coordinate [label=left:3] (A13) at (0,1.5);
 \coordinate (A14) at (0.5,1.5);
 \coordinate [label=left:4] (A15) at (0,2);
 \draw [very thin, gray] (A1)--(A5)--(A15)--cycle; 
 \draw (A9)--(A3)--(A11);
 \foreach \P in {1,6,10,13,15}  \draw[dashed, blue] (A11)--(A\P);
 \foreach \P in {3,4,5}  \draw[thick, dotted, red] (A9)--(A\P);
 \foreach \t in {1,2,...,15} \fill[black] (A\t) circle (0.05);
 \coordinate[label=below left:O] (B1) at (3,0);
 \coordinate [label=below:1] (B2) at (3.5,0);
 \coordinate [label=below:2] (B3) at (4,0);
 \coordinate [label=below:3] (B4) at (4.5,0);
 \coordinate [label=below:4] (B5) at (5,0);
 \coordinate [label=left:1] (B6) at (3,0.5);
 \coordinate (B7) at (3.5,0.5);
 \coordinate (B8) at (4,0.5);
 \coordinate (B9) at (4.5,0.5);
 \coordinate [label=left:2] (B10) at (3,1);
 \coordinate (B11) at (3.5,1);
 \coordinate (B12) at (4,1);
 \coordinate [label=left:3] (B13) at (3,1.5);
 \coordinate (B14) at (3.5,1.5);
 \coordinate [label=left:4] (B15) at (3,2);
 \draw [very thin, gray] (B1)--(B5)--(B15)--cycle; 
 \draw (B4)--(B12)--(B7);
 \foreach \t in {1,2,...,15} \fill[black] (B\t) circle (0.05);
 \foreach \P in {1,6,10,13,15}  \draw[dashed] (B7)--(B\P);
 \coordinate[label=above right:O] (O1) at (8,2);
 \coordinate (X1) at (6,2);
 \coordinate (Y1) at (8,0);
 \coordinate (X2) at (6.5,2);
 \coordinate (X3) at (7,2);
 \coordinate (X4) at (7.5,2);
 \coordinate (Y2) at (8,1);
 \draw[thin,->,>=stealth] (O1)--(X1) node[left] {$X$};
 \draw[thin,->,>=stealth] (O1)--(Y1) node[below] {$Y$};
 \fill[black] (X2) circle (0.06);
 \fill[black] (X3) circle (0.06);
 \fill[black] (X4) circle (0.06);
 \fill[black] (Y2) circle (0.06);
 \draw (X3)--(7,1.55);
 \draw (X2)--(6.5,1.55);
 \draw (X4)--(Y2);
 \draw[dashed] (6.75,1) circle (0.6);
 \draw (6.75,1) circle (0.2);
 \coordinate[label=above right:O] (O2) at (11.5,2);
 \coordinate (X5) at (9.5,2);
 \coordinate (Z1) at (11.5,0);
 \coordinate (X6) at (10,2);
 \coordinate (X7) at (10.5,2);
 \coordinate (X8) at (11,2);
 \draw[thin,->,>=stealth] (O2)--(X5) node[left] {$X$};
 \draw[thin,->,>=stealth] (O2)--(Z1) node[below] {$Z$};
 \fill[black] (X6) circle (0.06);
 \fill[black] (X7) circle (0.06);
 \fill[black] (X8) circle (0.06);
 \draw (X6)--(10,1.48);
 \draw (X7)--(10.5,1.7);
 \draw (X8)--(11,1.48);
 \draw[dashed] (10.5,1) circle (0.7);
 \draw (10.5,1) circle (0.2);
 \coordinate [label=above:$x_3$] (a1) at (6.5,2);
 \coordinate [label=above:$x_2$] (a2) at (7,2);
 \coordinate [label=above:$x_3$] (a2) at (10,2);
 \coordinate [label=above:$x_2$] (a2) at (10.5,2);   
\end{tikzpicture}
\caption{The complexes $K_{XY}$, $K_{XZ}$ and the skeletons of $C_{XY}$ and $C_{XZ}$ in Case $4. 3. 2$.}
\label{fc4.3.2}
\end{figure}

When we glue at $X$, we have the third cycle passing through the points $x_2$ and $x_3$.
Then any path connecting the cycle in $XY$ to the cycle in $XZ$ intersects the third cycle, so the skeleton of $C$ is not the lollipop graph.

\medbreak
Case $4. 3. 3$: $(3, 1, 0, 0)\rightarrow (2, 0, 2, 0)\rightarrow (1, 2, 0, 1)$

The complexes $K_{XY}$, $K_{XZ}$ and $K_{XW}$ are as in Figure \ref{fc4.3.3}.
\begin{figure}[H]
\center
\begin{tikzpicture}
 \coordinate[label=below left:O] (A1) at (0,0);
 \coordinate [label=below:1] (A2) at (0.5,0);
 \coordinate [label=below:2] (A3) at (1,0);
 \coordinate [label=below:3] (A4) at (1.5,0);
 \coordinate [label=below:4] (A5) at (2,0);
 \coordinate [label=left:1] (A6) at (0,0.5);
 \coordinate (A7) at (0.5,0.5);
 \coordinate (A8) at (1,0.5);
 \coordinate (A9) at (1.5,0.5);
 \coordinate [label=left:2] (A10) at (0,1);
 \coordinate (A11) at (0.5,1);
 \coordinate (A12) at (1,1);
 \coordinate [label=left:3] (A13) at (0,1.5);
 \coordinate (A14) at (0.5,1.5);
 \coordinate [label=left:4] (A15) at (0,2);
 \draw [very thin, gray] (A1)--(A5)--(A15)--cycle; 
 \draw (A9)--(A3)--(A11);
 \foreach \P in {1,6,10,13,15}  \draw[dashed, blue] (A11)--(A\P);
 \foreach \P in {3,4,5}  \draw[thick, dotted, red] (A9)--(A\P);
 \foreach \t in {1,2,...,15} \fill[black] (A\t) circle (0.05);
 \coordinate[label=below left:O] (B1) at (3,0);
 \coordinate [label=below:1] (B2) at (3.5,0);
 \coordinate [label=below:2] (B3) at (4,0);
 \coordinate [label=below:3] (B4) at (4.5,0);
 \coordinate [label=below:4] (B5) at (5,0);
 \coordinate [label=left:1] (B6) at (3,0.5);
 \coordinate (B7) at (3.5,0.5);
 \coordinate (B8) at (4,0.5);
 \coordinate (B9) at (4.5,0.5);
 \coordinate [label=left:2] (B10) at (3,1);
 \coordinate (B11) at (3.5,1);
 \coordinate (B12) at (4,1);
 \coordinate [label=left:3] (B13) at (3,1.5);
 \coordinate (B14) at (3.5,1.5);
 \coordinate [label=left:4] (B15) at (3,2);
 \draw [very thin, gray] (B1)--(B5)--(B15)--cycle; 
 \draw (B4)--(B12)--(B2);
 \foreach \t in {1,2,...,15} \fill[black] (B\t) circle (0.05);
 \foreach \P in {1,6,10,13,15}  \draw[dashed] (B2)--(B\P);
 \coordinate[label=below left:O] (C1) at (6,0);
 \coordinate [label=below:1] (C2) at (6.5,0);
 \coordinate [label=below:2] (C3) at (7,0);
 \coordinate [label=below:3] (C4) at (7.5,0);
 \coordinate [label=below:4] (C5) at (8,0);
 \coordinate [label=left:1] (C6) at (6,0.5);
 \coordinate (C7) at (6.5,0.5);
 \coordinate (C8) at (7,0.5);
 \coordinate (C9) at (7.5,0.5);
 \coordinate [label=left:2] (C10) at (6,1);
 \coordinate (C11) at (6.5,1);
 \coordinate (C12) at (7,1);
 \coordinate [label=left:3] (C13) at (6,1.5);
 \coordinate (C14) at (6.5,1.5);
 \coordinate [label=left:4] (C15) at (6,2);
 \draw [very thin, gray] (C1)--(C5)--(C15)--cycle; 
 \draw (C4)--(C3)--(C7);
 \foreach \t in {1,2,...,15} \fill[black] (C\t) circle (0.05);
 \foreach \P in {1,6,10,13,15}  \draw[dashed] (C7)--(C\P);
\end{tikzpicture}
\caption{The complexes $K_{XY}$, $K_{XZ}$ and $K_{XW}$ in Case $4. 3. 3$.}
\label{fc4.3.3}
\end{figure}
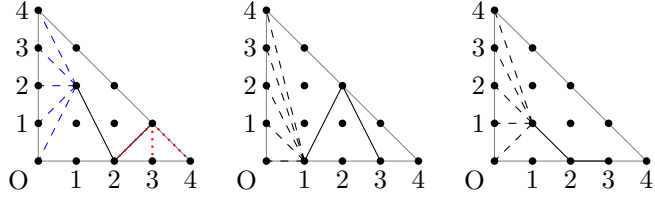

The skeletons of $C_{XY}$, $C_{XZ}$ and $C_{XW}$ are as in Figure \ref{fc4.3.3s}.
\begin{figure}[H]
\centering
\begin{tikzpicture}
 \coordinate[label=above right:O] (O1) at (-0.5,-1);
 \coordinate (X1) at (-2.5,-1);
 \coordinate (Y1) at (-0.5,-3);
 \coordinate (X2) at (-2,-1);
 \coordinate (X3) at (-1.5,-1);
 \coordinate (X4) at (-1,-1);
 \coordinate (Y2) at (-0.5,-2);
 \draw[thin,->,>=stealth] (O1)--(X1) node[left] {$X$};
 \draw[thin,->,>=stealth] (O1)--(Y1) node[below] {$Y$};
 \fill[black] (X2) circle (0.06);
 \fill[black] (X3) circle (0.06);
 \fill[black] (X4) circle (0.06);
 \fill[black] (Y2) circle (0.06);
 \draw (X3)--(-1.5,-1.45);
 \draw (X2)--(-2,-1.45);
 \draw (X4)--(Y2);
 \draw[dashed] (-1.75,-2) circle (0.6);
 \draw (-1.75,-2) circle (0.2);
 \coordinate[label=above right:O] (O2) at (3,-1);
 \coordinate (X5) at (1,-1);
 \coordinate (Z1) at (3,-3);
 \coordinate (X6) at (1.5,-1);
 \coordinate (X7) at (2,-1);
 \coordinate (X8) at (2.5,-1);
 \draw[thin,->,>=stealth] (O2)--(X5) node[left] {$X$};
 \draw[thin,->,>=stealth] (O2)--(Z1) node[below] {$Z$};
 \fill[black] (X6) circle (0.06);
 \fill[black] (X7) circle (0.06);
 \fill[black] (X8) circle (0.06);
 \draw (X7)--(2,-1.43);
 \draw (X8)--(2.5,-1.43);
 \draw (2.25,-2) circle (0.2);
 \draw[dashed] (2.25,-2) circle (0.6);
 \coordinate[label=below:O] (O3) at (6.5,-2.5);
 \coordinate (X9) at (4.5,-2.5);
 \coordinate (W1) at (7.6,-1);
 \coordinate (X10) at (5,-2.5);
 \coordinate (X11) at (5.5,-2.5);
 \coordinate (X12) at (6,-2.5);
 \draw[thin,->,>=stealth] (O3)--(X9) node[left] {$X$};
 \draw[thin,->,>=stealth] (O3)--(W1) node[above right] {$W$};
 \fill[black] (X10) circle (0.06);
 \fill[black] (X11) circle (0.06);
 \fill[black] (X12) circle (0.06);
 \draw (X10) arc (180:0:0.25);
 \coordinate [label=above:$x_3$] (a1) at (-2,-1);
 \coordinate [label=above:$x_2$] (a2) at (-1.5,-1);
 \coordinate [label=above:$x_2$] (a2) at (2,-1);
 \coordinate [label=below:$x_3$] (a1) at (5,-2.5);
 \coordinate [label=below:$x_2$] (a2) at (5.5,-2.5);   
\end{tikzpicture}
\caption{The skeletons of $C_{XY}$, $C_{XZ}$ and $C_{XW}$ in Case $4. 3. 3$.}
\label{fc4.3.3s}
\end{figure}
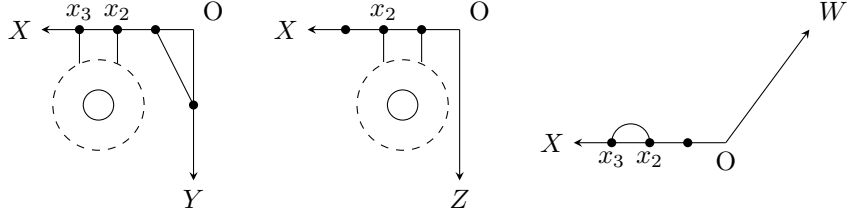

When we glue at $X$, we have the third cycle passing through the points $x_2$ and $x_3$.
Then any path connecting the cycle in $XY$ to the cycle in $XZ$ passes through the point $x_2$ and intersects the third cycle, so the skeleton of $C$ is not the lollipop graph.

\medbreak
Case $4. 3. 4$: $(3, 1, 0, 0)\rightarrow (2, 0, 2, 0)\rightarrow (1, 1, 2, 0)$

The complexes $K_{XY}$ and $K_{XZ}$ are as in Figure \ref{fc4.3.4} (a).
Assume that the skeleton of $C$ is the lollipop graph of genus $3$.
Then by Lemma \ref{12}, $K_{XZ}$ contains the $2$-simplices as in (b).
The skeletons of $C_{XY}$ and $C_{XZ}$ are as in (c).
When we glue these at $x_2$ and $x_3$, we have the shape as in (d) and the skeleton of $C$ is not the lollipop graph.

\begin{figure}[H]
\center
\begin{tikzpicture}
 \coordinate[label=below left:O] (B1) at (3,0);
 \coordinate [label=below:1] (B2) at (3.5,0);
 \coordinate [label=below:2] (B3) at (4,0);
 \coordinate [label=below:3] (B4) at (4.5,0);
 \coordinate [label=below:4] (B5) at (5,0);
 \coordinate [label=left:1] (B6) at (3,0.5);
 \coordinate (B7) at (3.5,0.5);
 \coordinate (B8) at (4,0.5);
 \coordinate (B9) at (4.5,0.5);
 \coordinate [label=left:2] (B10) at (3,1);
 \coordinate (B11) at (3.5,1);
 \coordinate (B12) at (4,1);
 \coordinate [label=left:3] (B13) at (3,1.5);
 \coordinate (B14) at (3.5,1.5);
 \coordinate [label=left:4] (B15) at (3,2);
 \draw [very thin, gray] (B1)--(B5)--(B15)--cycle;
 \draw (B9)--(B3)--(B7);
 \foreach \P in {1,6,10,13,15} \draw[dashed, blue] (B7)--(B\P);
 \foreach \P in {3,4,5} \draw[thick, dotted, red] (B9)--(B\P);
 \foreach \t in {1,2,...,15} \fill[black] (B\t) circle (0.05);
 \coordinate[label=below left:O] (C1) at (6,0);
 \coordinate [label=below:1] (C2) at (6.5,0);
 \coordinate [label=below:2] (C3) at (7,0);
 \coordinate [label=below:3] (C4) at (7.5,0);
 \coordinate [label=below:4] (C5) at (8,0);
 \coordinate [label=left:1] (C6) at (6,0.5);
 \coordinate (C7) at (6.5,0.5);
 \coordinate (C8) at (7,0.5);
 \coordinate (C9) at (7.5,0.5);
 \coordinate [label=left:2] (C10) at (6,1);
 \coordinate (C11) at (6.5,1);
 \coordinate (C12) at (7,1);
 \coordinate [label=left:3] (C13) at (6,1.5);
 \coordinate (C14) at (6.5,1.5);
 \coordinate [label=left:4] (C15) at (6,2);
 \draw [very thin, gray] (C1)--(C5)--(C15)--cycle;
 \draw (C4)--(C12)--(C11);
 \foreach \t in {1,2,...,15} \fill[black] (C\t) circle (0.05);
 \foreach \P in {1,6,10,13,15} \draw[dashed] (C11)--(C\P);
 \coordinate[label=below left:O] (D1) at (10,0);
 \coordinate [label=below:1] (D2) at (10.5,0);
 \coordinate [label=below:2] (D3) at (11,0);
 \coordinate [label=below:3] (D4) at (11.5,0);
 \coordinate [label=below:4] (D5) at (12,0);
 \coordinate [label=left:1] (D6) at (10,0.5);
 \coordinate (D7) at (10.5,0.5);
 \coordinate (D8) at (11,0.5);
 \coordinate (D9) at (11.5,0.5);
 \coordinate [label=left:2] (D10) at (10,1);
 \coordinate (D11) at (10.5,1);
 \coordinate (D12) at (11,1);
 \coordinate [label=left:3] (D13) at (10,1.5);
 \coordinate (D14) at (10.5,1.5);
 \coordinate [label=left:4] (D15) at (10,2);
 \draw [very thin, gray] (D1)--(D5)--(D15)--cycle;
 \foreach \P in {13,15} \draw[dashed] (D11)--(D\P);
 \draw (D4)--(D12);
 \draw (D11)--(D12);
 \draw (D13)--(D3)--(D11);
 \draw (D3)--(D7)--(D13);
 \draw (D3)--(D8)--(D11);
 \draw[very thick] (D8)--(D11);
 \foreach \t in {1,2,...,15} \fill[black] (D\t) circle (0.05);
 \coordinate [label=below:(a)] (a1) at (5.5,-0.5);
 \coordinate [label=below:(b)] (a2) at (11,-0.5); 
\end{tikzpicture}
\begin{tikzpicture}
 \coordinate[label=above right:O] (O1) at (-0.5,-1);
 \coordinate (X1) at (-2.5,-1);
 \coordinate (Y1) at (-0.5,-3);
 \coordinate (X2) at (-2,-1);
 \coordinate (X3) at (-1.5,-1);
 \coordinate (X4) at (-1,-1);
 \coordinate (Y2) at (-0.5,-2);
 \draw[thin,->,>=stealth] (O1)--(X1) node[left] {$X$};
 \draw[thin,->,>=stealth] (O1)--(Y1) node[below] {$Y$};
 \fill[black] (X2) circle (0.06);
 \fill[black] (X3) circle (0.06);
 \fill[black] (X4) circle (0.06);
 \fill[black] (Y2) circle (0.06);
 \draw (X3) arc (360:180:0.25);
 \draw (X4) to [out=270, in=160] (Y2);
 \coordinate[label=above right:O] (O2) at (3,-1);
 \coordinate (X5) at (1,-1);
 \coordinate (Z1) at (3,-3);
 \coordinate (X6) at (1.5,-1);
 \coordinate (X7) at (2,-1);
 \coordinate (X8) at (2.5,-1);
 \draw[thin,->,>=stealth] (O2)--(X5) node[left] {$X$};
 \draw[thin,->,>=stealth] (O2)--(Z1) node[below] {$Z$};
 \fill[black] (X6) circle (0.06);
 \fill[black] (X7) circle (0.06);
 \fill[black] (X8) circle (0.06);
 \draw (X6)--(1.5,-1.8);
 \draw (2.05,-1.5)--(1.5,-1.5);
 \draw (X7)--(2.13,-1.34);
 \draw (X8)--(2.37,-1.34);
 \draw (1.5,-2) circle (0.2);
 \draw (2.25,-1.5) circle (0.2);
 \coordinate[label=above right:O] (O2) at (7,-1);
 \coordinate (X5) at (5,-1);
 \coordinate (Z1) at (7,-3);
 \coordinate (X6) at (5.5,-1);
 \coordinate (X7) at (6,-1);
 \coordinate (X8) at (6.5,-1);
 \draw[thin,->,>=stealth] (O2)--(X5) node[left] {$X$};
 \draw[thin,->,>=stealth] (O2)--(Z1) node[below] {$Z$};
 \fill[black] (X6) circle (0.06);
 \fill[black] (X7) circle (0.06);
 \fill[black] (X8) circle (0.06);
 \draw (X6)--(5.5,-1.8);
 \draw (6.05,-1.5)--(5.5,-1.5);
 \draw (X7)--(6.13,-1.34);
 \draw (X8)--(6.37,-1.34);
 \draw (5.5,-2) circle (0.2);
 \draw (6.25,-1.5) circle (0.2);
 \draw (X6) arc (180:0:0.25);
 \draw[very thick] (6.05,-1.5) arc (180:125:0.2);
 \coordinate [label=below:(c)] (a1) at (0.25,-3.5);
 \coordinate [label=below:(d)] (a2) at (6,-3.5);
\end{tikzpicture}
\caption{The complexes and skeletons in Case $4. 3. 4$.}
\label{fc4.3.4}
\end{figure}
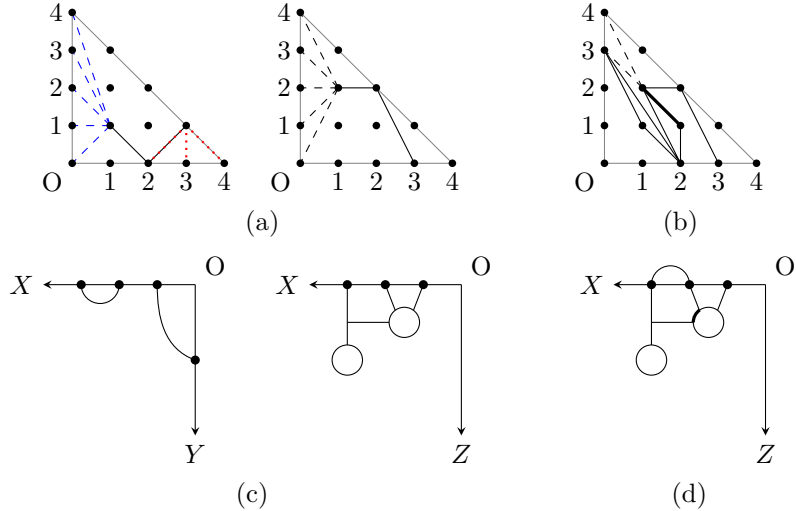

Case $4. 3. 5$: $(3, 1, 0, 0)\rightarrow (2, 0, 2, 0)\rightarrow (1, 1, 1, 1)$

The complexes $K_{XY}$, $K_{XZ}$ and $K_{XW}$ are as in Figure \ref{fc4.3.5}.
Thus $b_1(\overline{C'})=1$ and the skeleton of $C$ is not the lollipop graph by Lemma \ref{case4hodai}.
\begin{figure}[H]
\centering
\begin{tikzpicture}
 \coordinate[label=below left:O] (A1) at (0,0);
 \coordinate [label=below:1] (A2) at (0.5,0);
 \coordinate [label=below:2] (A3) at (1,0);
 \coordinate [label=below:3] (A4) at (1.5,0);
 \coordinate [label=below:4] (A5) at (2,0);
 \coordinate [label=left:1] (A6) at (0,0.5);
 \coordinate (A7) at (0.5,0.5);
 \coordinate (A8) at (1,0.5);
 \coordinate (A9) at (1.5,0.5);
 \coordinate [label=left:2] (A10) at (0,1);
 \coordinate (A11) at (0.5,1);
 \coordinate (A12) at (1,1);
 \coordinate [label=left:3] (A13) at (0,1.5);
 \coordinate (A14) at (0.5,1.5);
 \coordinate [label=left:4] (A15) at (0,2);
 \draw [very thin, gray] (A1)--(A5)--(A15)--cycle; 
 \draw (A9)--(A3)--(A7);
 \foreach \P in {1,6,10,13,15}  \draw[dashed, blue] (A7)--(A\P);
 \foreach \P in {3,4,5}  \draw[thick, dotted, red] (A9)--(A\P);
 \foreach \t in {1,2,...,15} \fill[black] (A\t) circle (0.05);
 \coordinate[label=below left:O] (B1) at (3,0);
 \coordinate [label=below:1] (B2) at (3.5,0);
 \coordinate [label=below:2] (B3) at (4,0);
 \coordinate [label=below:3] (B4) at (4.5,0);
 \coordinate [label=below:4] (B5) at (5,0);
 \coordinate [label=left:1] (B6) at (3,0.5);
 \coordinate (B7) at (3.5,0.5);
 \coordinate (B8) at (4,0.5);
 \coordinate (B9) at (4.5,0.5);
 \coordinate [label=left:2] (B10) at (3,1);
 \coordinate (B11) at (3.5,1);
 \coordinate (B12) at (4,1);
 \coordinate [label=left:3] (B13) at (3,1.5);
 \coordinate (B14) at (3.5,1.5);
 \coordinate [label=left:4] (B15) at (3,2);
 \draw [very thin, gray] (B1)--(B5)--(B15)--cycle; 
 \draw (B4)--(B12)--(B7);
 \foreach \t in {1,2,...,15} \fill[black] (B\t) circle (0.05);
 \foreach \P in {1,6,10,13,15}  \draw[dashed] (B7)--(B\P);
 \coordinate[label=below left:O] (C1) at (6,0);
 \coordinate [label=below:1] (C2) at (6.5,0);
 \coordinate [label=below:2] (C3) at (7,0);
 \coordinate [label=below:3] (C4) at (7.5,0);
 \coordinate [label=below:4] (C5) at (8,0);
 \coordinate [label=left:1] (C6) at (6,0.5);
 \coordinate (C7) at (6.5,0.5);
 \coordinate (C8) at (7,0.5);
 \coordinate (C9) at (7.5,0.5);
 \coordinate [label=left:2] (C10) at (6,1);
 \coordinate (C11) at (6.5,1);
 \coordinate (C12) at (7,1);
 \coordinate [label=left:3] (C13) at (6,1.5);
 \coordinate (C14) at (6.5,1.5);
 \coordinate [label=left:4] (C15) at (6,2);
 \draw [very thin, gray] (C1)--(C5)--(C15)--cycle; 
 \draw (C4)--(C3)--(C7);
 \foreach \t in {1,2,...,15} \fill[black] (C\t) circle (0.05);
 \foreach \P in {1,6,10,13,15}  \draw[dashed] (C7)--(C\P);
\end{tikzpicture}
\caption{The complexes $K_{XY}$, $K_{XZ}$ and $K_{XW}$ in Case $4. 3. 5$.}
\label{fc4.3.5}
\end{figure}
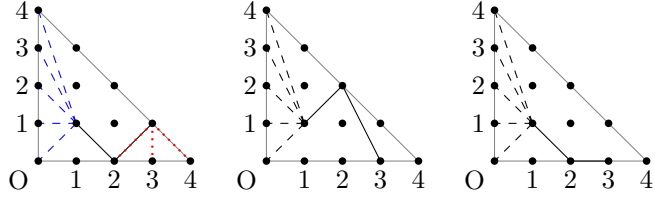

Case $4. 3. 6$: $(3, 1, 0, 0)\rightarrow (2, 0, 2, 0)\rightarrow (1, 1, 0, 2)$

The complexes $K_{XY}$, $K_{XZ}$ and $K_{XW}$ are as in Figure \ref{fc4.3.6}.
\begin{figure}[H]
\centering
\begin{tikzpicture}
 \coordinate[label=below left:O] (A1) at (0,0);
 \coordinate [label=below:1] (A2) at (0.5,0);
 \coordinate [label=below:2] (A3) at (1,0);
 \coordinate [label=below:3] (A4) at (1.5,0);
 \coordinate [label=below:4] (A5) at (2,0);
 \coordinate [label=left:1] (A6) at (0,0.5);
 \coordinate (A7) at (0.5,0.5);
 \coordinate (A8) at (1,0.5);
 \coordinate (A9) at (1.5,0.5);
 \coordinate [label=left:2] (A10) at (0,1);
 \coordinate (A11) at (0.5,1);
 \coordinate (A12) at (1,1);
 \coordinate [label=left:3] (A13) at (0,1.5);
 \coordinate (A14) at (0.5,1.5);
 \coordinate [label=left:4] (A15) at (0,2);
 \draw [very thin, gray] (A1)--(A5)--(A15)--cycle; 
 \draw (A9)--(A3)--(A7);
 \foreach \P in {1,6,10,13,15}  \draw[dashed, blue] (A7)--(A\P);
 \foreach \P in {3,4,5}  \draw[thick, dotted, red] (A9)--(A\P);
 \foreach \t in {1,2,...,15} \fill[black] (A\t) circle (0.05);
 \coordinate[label=below left:O] (B1) at (3,0);
 \coordinate [label=below:1] (B2) at (3.5,0);
 \coordinate [label=below:2] (B3) at (4,0);
 \coordinate [label=below:3] (B4) at (4.5,0);
 \coordinate [label=below:4] (B5) at (5,0);
 \coordinate [label=left:1] (B6) at (3,0.5);
 \coordinate (B7) at (3.5,0.5);
 \coordinate (B8) at (4,0.5);
 \coordinate (B9) at (4.5,0.5);
 \coordinate [label=left:2] (B10) at (3,1);
 \coordinate (B11) at (3.5,1);
 \coordinate (B12) at (4,1);
 \coordinate [label=left:3] (B13) at (3,1.5);
 \coordinate (B14) at (3.5,1.5);
 \coordinate [label=left:4] (B15) at (3,2);
 \draw [very thin, gray] (B1)--(B5)--(B15)--cycle; 
 \draw (B4)--(B12)--(B2);
 \foreach \t in {1,2,...,15} \fill[black] (B\t) circle (0.05);
 \foreach \P in {1,6,10,13,15}  \draw[dashed] (B2)--(B\P);
 \coordinate[label=below left:O] (C1) at (6,0);
 \coordinate [label=below:1] (C2) at (6.5,0);
 \coordinate [label=below:2] (C3) at (7,0);
 \coordinate [label=below:3] (C4) at (7.5,0);
 \coordinate [label=below:4] (C5) at (8,0);
 \coordinate [label=left:1] (C6) at (6,0.5);
 \coordinate (C7) at (6.5,0.5);
 \coordinate (C8) at (7,0.5);
 \coordinate (C9) at (7.5,0.5);
 \coordinate [label=left:2] (C10) at (6,1);
 \coordinate (C11) at (6.5,1);
 \coordinate (C12) at (7,1);
 \coordinate [label=left:3] (C13) at (6,1.5);
 \coordinate (C14) at (6.5,1.5);
 \coordinate [label=left:4] (C15) at (6,2);
 \draw [very thin, gray] (C1)--(C5)--(C15)--cycle; 
 \draw (C4)--(C3)--(C11);
 \foreach \t in {1,2,...,15} \fill[black] (C\t) circle (0.05);
 \foreach \P in {1,6,10,13,15}  \draw[dashed] (C11)--(C\P);
\end{tikzpicture}
\caption{The complexes $K_{XY}$, $K_{XZ}$ and $K_{XW}$ in Case $4. 3. 6$}
\label{fc4.3.6}
\end{figure}
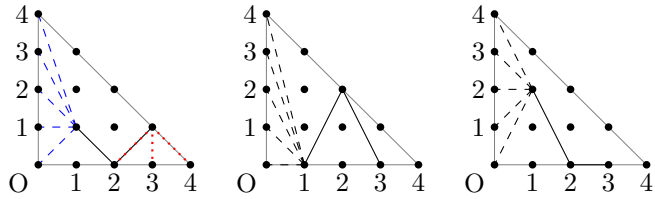

The skeletons of $C_{XY}$, $C_{XZ}$ and $C_{XW}$ are as in Figure \ref{fc4.3.6s}.
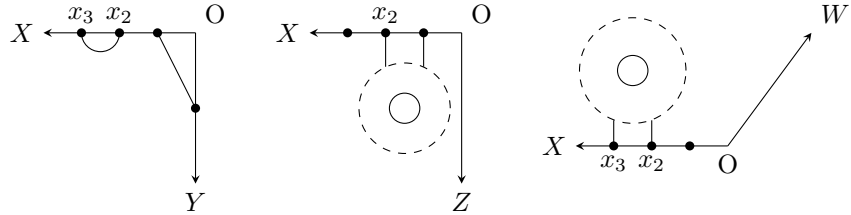
\begin{figure}[H]
\centering
\begin{tikzpicture}
 \coordinate[label=above right:O] (O1) at (-0.5,-1);
 \coordinate (X1) at (-2.5,-1);
 \coordinate (Y1) at (-0.5,-3);
 \coordinate (X2) at (-2,-1);
 \coordinate (X3) at (-1.5,-1);
 \coordinate (X4) at (-1,-1);
 \coordinate (Y2) at (-0.5,-2);
 \draw[thin,->,>=stealth] (O1)--(X1) node[left] {$X$};
 \draw[thin,->,>=stealth] (O1)--(Y1) node[below] {$Y$};
 \fill[black] (X2) circle (0.06);
 \fill[black] (X3) circle (0.06);
 \fill[black] (X4) circle (0.06);
 \fill[black] (Y2) circle (0.06);
 \draw (X3) arc (360:180:0.25);
 \draw (X4)--(Y2);
 \coordinate[label=above right:O] (O2) at (3,-1);
 \coordinate (X5) at (1,-1);
 \coordinate (Z1) at (3,-3);
 \coordinate (X6) at (1.5,-1);
 \coordinate (X7) at (2,-1);
 \coordinate (X8) at (2.5,-1);
 \draw[thin,->,>=stealth] (O2)--(X5) node[left] {$X$};
 \draw[thin,->,>=stealth] (O2)--(Z1) node[below] {$Z$};
 \fill[black] (X6) circle (0.06);
 \fill[black] (X7) circle (0.06);
 \fill[black] (X8) circle (0.06);
 \draw (X7)--(2,-1.45);
 \draw (X8)--(2.5,-1.45);
 \draw (2.25,-2) circle (0.2);
 \draw[dashed] (2.25,-2) circle (0.6);
 \coordinate[label=below:O] (O3) at (6.5,-2.5);
 \coordinate (X9) at (4.5,-2.5);
 \coordinate (W1) at (7.6,-1);
 \coordinate (X10) at (5,-2.5);
 \coordinate (X11) at (5.5,-2.5);
 \coordinate (X12) at (6,-2.5);
 \draw[thin,->,>=stealth] (O3)--(X9) node[left] {$X$};
 \draw[thin,->,>=stealth] (O3)--(W1) node[above right] {$W$};
 \fill[black] (X10) circle (0.06);
 \fill[black] (X11) circle (0.06);
 \fill[black] (X12) circle (0.06);
 \draw (X10)--(5,-2.16);
 \draw (X11)--(5.5,-2.16);
 \draw[dashed] (5.25,-1.5) circle (0.7);
 \draw (5.25,-1.5) circle (0.2);
 \coordinate [label=above:$x_3$] (a1) at (-2,-1);
 \coordinate [label=above:$x_2$] (a2) at (-1.5,-1);
 \coordinate [label=above:$x_2$] (a2) at (2,-1);
 \coordinate [label=below:$x_3$] (a1) at (5,-2.5);
 \coordinate [label=below:$x_2$] (a2) at (5.5,-2.5);  
\end{tikzpicture}
\caption{The skeletons of $C_{XY}$, $C_{XZ}$ and $C_{XW}$ in Case $4. 3. 6$.}
\label{fc4.3.6s}
\end{figure}

When we glue at $X$, we have the third cycle passing through the points $x_2$ and $x_3$.
Then any path connecting the cycle in $XW$ to the cycle in $XZ$ passes through the point $x_2$ and intersects the third cycle, so the skeleton of $C$ is not the lollipop graph.

\medbreak
Case $4. 3. 7$: $(3, 1, 0, 0)\rightarrow (2, 0, 2, 0)\rightarrow (1, 0, 3, 0)$

The complex $K_{XZ}$ is as in Figure \ref{fc4.3.7} and the skeleton of $C$ is not the lollipop graph by Lemma \ref{3in1}.
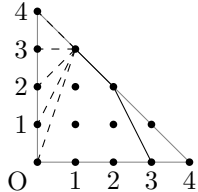
\begin{figure}[H]
\centering
\begin{tikzpicture}
 \coordinate[label=below left:O] (B1) at (3,0);
 \coordinate [label=below:1] (B2) at (3.5,0);
 \coordinate [label=below:2] (B3) at (4,0);
 \coordinate [label=below:3] (B4) at (4.5,0);
 \coordinate [label=below:4] (B5) at (5,0);
 \coordinate [label=left:1] (B6) at (3,0.5);
 \coordinate (B7) at (3.5,0.5);
 \coordinate (B8) at (4,0.5);
 \coordinate (B9) at (4.5,0.5);
 \coordinate [label=left:2] (B10) at (3,1);
 \coordinate (B11) at (3.5,1);
 \coordinate (B12) at (4,1);
 \coordinate [label=left:3] (B13) at (3,1.5);
 \coordinate (B14) at (3.5,1.5);
 \coordinate [label=left:4] (B15) at (3,2);
 \draw [very thin, gray] (B1)--(B5)--(B15)--cycle; 
 \draw (B4)--(B12)--(B14);
 \foreach \t in {1,2,...,15} \fill[black] (B\t) circle (0.05);
 \foreach \P in {1,6,10,13,15}  \draw[dashed] (B14)--(B\P);
\end{tikzpicture}
\caption{The complex $K_{XZ}$ in Case $4. 3. 7$.}
\label{fc4.3.7}
\end{figure}

Case $4. 3. 8$: $(3, 1, 0, 0)\rightarrow (2, 0, 2, 0)\rightarrow (1, 0, 2, 1)$

The complexes $K_{XZ}$ and $K_{XW}$ are as in Figure \ref{fc4.3.8} (a).
Assume that the skeleton of $C$ is the lollipop graph of genus $3$.
Then by Lemma \ref{12}, $K_{XZ}$ contains the $2$-simplices as in (b).
The skeletons of $C_{XZ}$ and $C_{XW}$ are as in (c).
If we glue these together, we have the shape as in (d) and the skeleton of $C$ is not the lollipop graph.
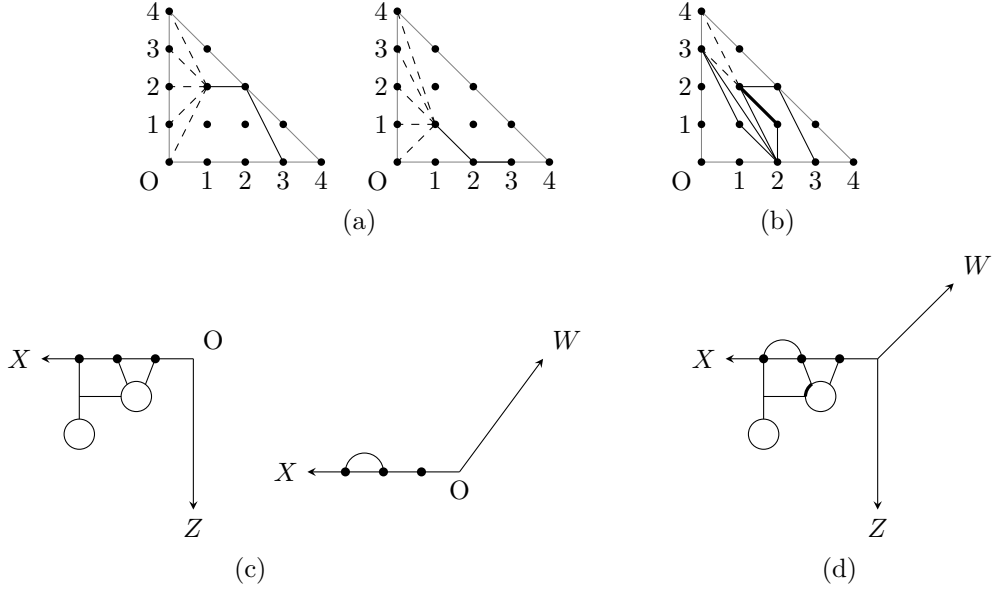
\begin{figure}[H]
\centering
\begin{tikzpicture}
 \coordinate[label=below left:O] (B1) at (3,0);
 \coordinate [label=below:1] (B2) at (3.5,0);
 \coordinate [label=below:2] (B3) at (4,0);
 \coordinate [label=below:3] (B4) at (4.5,0);
 \coordinate [label=below:4] (B5) at (5,0);
 \coordinate [label=left:1] (B6) at (3,0.5);
 \coordinate (B7) at (3.5,0.5);
 \coordinate (B8) at (4,0.5);
 \coordinate (B9) at (4.5,0.5);
 \coordinate [label=left:2] (B10) at (3,1);
 \coordinate (B11) at (3.5,1);
 \coordinate (B12) at (4,1);
 \coordinate [label=left:3] (B13) at (3,1.5);
 \coordinate (B14) at (3.5,1.5);
 \coordinate [label=left:4] (B15) at (3,2);
 \draw [very thin, gray] (B1)--(B5)--(B15)--cycle; 
 \draw (B4)--(B12)--(B11);
 \foreach \t in {1,2,...,15} \fill[black] (B\t) circle (0.05);
 \foreach \P in {1,6,10,13,15}  \draw[dashed] (B11)--(B\P);
 \coordinate[label=below left:O] (C1) at (6,0);
 \coordinate [label=below:1] (C2) at (6.5,0);
 \coordinate [label=below:2] (C3) at (7,0);
 \coordinate [label=below:3] (C4) at (7.5,0);
 \coordinate [label=below:4] (C5) at (8,0);
 \coordinate [label=left:1] (C6) at (6,0.5);
 \coordinate (C7) at (6.5,0.5);
 \coordinate (C8) at (7,0.5);
 \coordinate (C9) at (7.5,0.5);
 \coordinate [label=left:2] (C10) at (6,1);
 \coordinate (C11) at (6.5,1);
 \coordinate (C12) at (7,1);
 \coordinate [label=left:3] (C13) at (6,1.5);
 \coordinate (C14) at (6.5,1.5);
 \coordinate [label=left:4] (C15) at (6,2);
 \draw [very thin, gray] (C1)--(C5)--(C15)--cycle; 
 \draw (C4)--(C3)--(C7);
 \foreach \t in {1,2,...,15} \fill[black] (C\t) circle (0.05);
 \foreach \P in {1,6,10,13,15}  \draw[dashed] (C7)--(C\P);
 \coordinate[label=below left:O] (D1) at (10,0);
 \coordinate [label=below:1] (D2) at (10.5,0);
 \coordinate [label=below:2] (D3) at (11,0);
 \coordinate [label=below:3] (D4) at (11.5,0);
 \coordinate [label=below:4] (D5) at (12,0);
 \coordinate [label=left:1] (D6) at (10,0.5);
 \coordinate (D7) at (10.5,0.5);
 \coordinate (D8) at (11,0.5);
 \coordinate (D9) at (11.5,0.5);
 \coordinate [label=left:2] (D10) at (10,1);
 \coordinate (D11) at (10.5,1);
 \coordinate (D12) at (11,1);
 \coordinate [label=left:3] (D13) at (10,1.5);
 \coordinate (D14) at (10.5,1.5);
 \coordinate [label=left:4] (D15) at (10,2);
 \draw [very thin, gray] (D1)--(D5)--(D15)--cycle;
 \foreach \P in {13,15} \draw[dashed] (D11)--(D\P);
 \draw (D4)--(D12);
 \draw (D11)--(D12);
 \draw (D13)--(D3)--(D11);
 \draw (D3)--(D7)--(D13);
 \draw (D3)--(D8)--(D11);
 \draw[very thick] (D8)--(D11);
 \foreach \t in {1,2,...,15} \fill[black] (D\t) circle (0.05);
 \coordinate [label=below:(a)] (a1) at (5.5,-0.5);
 \coordinate [label=below:(b)] (a2) at (11,-0.5);
\end{tikzpicture}
\begin{tikzpicture}
 \coordinate[label=above right:O] (O2) at (3,-1);
 \coordinate (X5) at (1,-1);
 \coordinate (Z1) at (3,-3);
 \coordinate (X6) at (1.5,-1);
 \coordinate (X7) at (2,-1);
 \coordinate (X8) at (2.5,-1);
 \draw[thin,->,>=stealth] (O2)--(X5) node[left] {$X$};
 \draw[thin,->,>=stealth] (O2)--(Z1) node[below] {$Z$};
 \fill[black] (X6) circle (0.06);
 \fill[black] (X7) circle (0.06);
 \fill[black] (X8) circle (0.06);
 \draw (X6)--(1.5,-1.8);
 \draw (2.05,-1.5)--(1.5,-1.5);
 \draw (X7)--(2.13,-1.34);
 \draw (X8)--(2.37,-1.34);
 \draw (1.5,-2) circle (0.2);
 \draw (2.25,-1.5) circle (0.2);
 \coordinate[label=below:O] (O3) at (6.5,-2.5);
 \coordinate (X9) at (4.5,-2.5);
 \coordinate (W1) at (7.6,-1);
 \coordinate (X10) at (5,-2.5);
 \coordinate (X11) at (5.5,-2.5);
 \coordinate (X12) at (6,-2.5);
 \draw[thin,->,>=stealth] (O3)--(X9) node[left] {$X$};
 \draw[thin,->,>=stealth] (O3)--(W1) node[above right] {$W$};
 \fill[black] (X10) circle (0.06);
 \fill[black] (X11) circle (0.06);
 \fill[black] (X12) circle (0.06);
 \draw (X10) arc (180:0:0.25);
 \coordinate (O2) at (12,-1);
 \coordinate (X5) at (10,-1);
 \coordinate (Z1) at (12,-3);
 \coordinate (X6) at (10.5,-1);
 \coordinate (X7) at (11,-1);
 \coordinate (X8) at (11.5,-1);
 \coordinate (X1) at (13,0);
 \draw[thin,->,>=stealth] (O2)--(X1) node[above right] {$W$}; 
 \draw[thin,->,>=stealth] (O2)--(X5) node[left] {$X$};
 \draw[thin,->,>=stealth] (O2)--(Z1) node[below] {$Z$};
 \fill[black] (X6) circle (0.06);
 \fill[black] (X7) circle (0.06);
 \fill[black] (X8) circle (0.06);
 \draw (X6)--(10.5,-1.8);
 \draw (11.05,-1.5)--(10.5,-1.5);
 \draw (X7)--(11.13,-1.34);
 \draw (10.5,-2) circle (0.2);
 \draw (11.25,-1.5) circle (0.2);
 \draw (X6) arc (180:0:0.25);
 \draw (X8)--(11.37,-1.34);
 \draw[very thick] (11.05,-1.5) arc (180:125:0.2);
 \coordinate [label=below:(c)] (a1) at (3.75,-3.5);
 \coordinate [label=below:(d)] (a2) at (11.5,-3.5);
\end{tikzpicture}
\caption{The complexes and skeletons in Case $4. 3. 8$.}
\label{fc4.3.8}
\end{figure}

Case $4. 3. 9$: $(3, 1, 0, 0)\rightarrow (2, 0, 2, 0)\rightarrow (1, 0, 1, 2)$

The complexes $K_{XZ}$, $K_{XW}$ and the skeletons of $C_{XZ}$ and $C_{XW}$ are as in Figure \ref{fc4.3.9}.
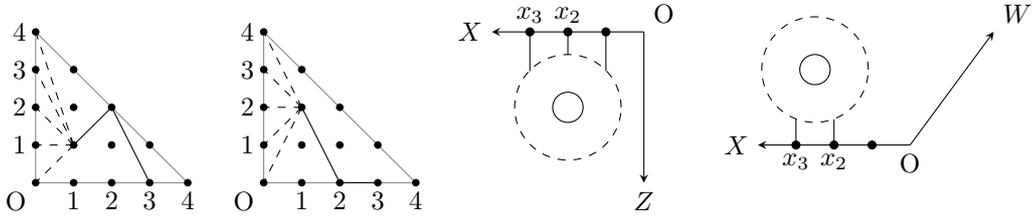
\begin{figure}[H]
\centering
\begin{tikzpicture}
 \coordinate[label=below left:O] (B1) at (3,0);
 \coordinate [label=below:1] (B2) at (3.5,0);
 \coordinate [label=below:2] (B3) at (4,0);
 \coordinate [label=below:3] (B4) at (4.5,0);
 \coordinate [label=below:4] (B5) at (5,0);
 \coordinate [label=left:1] (B6) at (3,0.5);
 \coordinate (B7) at (3.5,0.5);
 \coordinate (B8) at (4,0.5);
 \coordinate (B9) at (4.5,0.5);
 \coordinate [label=left:2] (B10) at (3,1);
 \coordinate (B11) at (3.5,1);
 \coordinate (B12) at (4,1);
 \coordinate [label=left:3] (B13) at (3,1.5);
 \coordinate (B14) at (3.5,1.5);
 \coordinate [label=left:4] (B15) at (3,2);
 \draw [very thin, gray] (B1)--(B5)--(B15)--cycle; 
 \draw (B4)--(B12)--(B7);
 \foreach \t in {1,2,...,15} \fill[black] (B\t) circle (0.05);
 \foreach \P in {1,6,10,13,15}  \draw[dashed] (B7)--(B\P);
 \coordinate[label=below left:O] (C1) at (6,0);
 \coordinate [label=below:1] (C2) at (6.5,0);
 \coordinate [label=below:2] (C3) at (7,0);
 \coordinate [label=below:3] (C4) at (7.5,0);
 \coordinate [label=below:4] (C5) at (8,0);
 \coordinate [label=left:1] (C6) at (6,0.5);
 \coordinate (C7) at (6.5,0.5);
 \coordinate (C8) at (7,0.5);
 \coordinate (C9) at (7.5,0.5);
 \coordinate [label=left:2] (C10) at (6,1);
 \coordinate (C11) at (6.5,1);
 \coordinate (C12) at (7,1);
 \coordinate [label=left:3] (C13) at (6,1.5);
 \coordinate (C14) at (6.5,1.5);
 \coordinate [label=left:4] (C15) at (6,2);
 \draw [very thin, gray] (C1)--(C5)--(C15)--cycle; 
 \draw (C4)--(C3)--(C11);
 \foreach \t in {1,2,...,15} \fill[black] (C\t) circle (0.05);
 \foreach \P in {1,6,10,13,15}  \draw[dashed] (C11)--(C\P);
 \coordinate[label=above right:O] (O2) at (11,2);
 \coordinate (X5) at (9,2);
 \coordinate (Z1) at (11,0);
 \coordinate (X6) at (9.5,2);
 \coordinate (X7) at (10,2);
 \coordinate (X8) at (10.5,2);
 \draw[thin,->,>=stealth] (O2)--(X5) node[left] {$X$};
 \draw[thin,->,>=stealth] (O2)--(Z1) node[below] {$Z$};
 \fill[black] (X6) circle (0.06);
 \fill[black] (X7) circle (0.06);
 \fill[black] (X8) circle (0.06);
 \draw (X6)--(9.5,1.48);
 \draw (X7)--(10,1.7);
 \draw (X8)--(10.5,1.48);
 \draw[dashed] (10,1) circle (0.7);
 \draw (10,1) circle (0.2);
 \coordinate[label=below:O] (O3) at (14.5,0.5);
 \coordinate (X9) at (12.5,0.5);
 \coordinate (W1) at (15.6,2);
 \coordinate (X10) at (13,0.5);
 \coordinate (X11) at (13.5,0.5);
 \coordinate (X12) at (14,0.5);
 \draw[thin,->,>=stealth] (O3)--(X9) node[left] {$X$};
 \draw[thin,->,>=stealth] (O3)--(W1) node[above right] {$W$};
 \fill[black] (X10) circle (0.06);
 \fill[black] (X11) circle (0.06);
 \fill[black] (X12) circle (0.06);
 \draw (X10)--(13,0.84);
 \draw (X11)--(13.5,0.84);
 \draw[dashed] (13.25,1.5) circle (0.7);
 \draw (13.25,1.5) circle (0.2);
 \coordinate [label=above:$x_3$] (a1) at (9.5,2);
 \coordinate [label=above:$x_2$] (a2) at (10,2); 
 \coordinate [label=below:$x_3$] (a1) at (13,0.5);
 \coordinate [label=below:$x_2$] (a2) at (13.5,0.5); 
\end{tikzpicture}
\caption{The complexes $K_{XZ}$, $K_{XW}$ and the skeletons of $C_{XZ}$ and $C_{XW}$ in Case $4. 3. 9$.}
\label{fc4.3.9}
\end{figure}

When we glue at $X$, we have the third cycle passing through the points $x_2$ and $x_3$.
Then any path connecting the cycle in $XW$ to the cycle in $XZ$ intersects the third cycle, so the skeleton of $C$ is not the lollipop graph.

\medbreak
Case $4. 3. 10$: $(3, 1, 0, 0)\rightarrow (2, 0, 2, 0)\rightarrow (1, 0, 0, 3)$

Since $K_{XW}$ contains the edge $(2, 0)$-$(1, 3)$ as its $1$-simplex and the points $(1, 1)$ and $(1, 2)$ in its interior, the skeleton of $C$ is not the lollipop graph by Lemma \ref{Case2.6}.

From the above, in Case $4. 3$, the skeleton of $C$ is not the lollipop graph.

\subsubsection{Case $4. 4$: $(3, 1, 0, 0)\rightarrow (2, 0, 1, 1)$}

We divide into the following cases.
\begin{eqnarray*}
\text{Case $4. 4. 1$: }(2, 0, 1, 1)\rightarrow (1, 3, 0, 0), \quad \text{Case $4. 4. 2$: }(2, 0, 1, 1)\rightarrow (1, 2, 1, 0),\\
\text{Case $4. 4. 3$: }(2, 0, 1, 1)\rightarrow (1, 1, 2, 0), \quad \text{Case $4. 4. 4$: }(2, 0, 1, 1)\rightarrow (1, 1, 1, 1),\\
\text{Case $4. 4. 5$: }(2, 0, 1, 1)\rightarrow (1, 0, 3, 0), \quad \text{Case $4. 4. 6$: }(2, 0, 1, 1)\rightarrow (1, 0, 2, 1).\hspace{0.5mm}
\end{eqnarray*}

Cases $4. 4. 1$ to $4. 4. 4$:

The complexes $K_{XY}$, $K_{XZ}$ and $K_{XW}$ are as in Figure \ref{fc4.4.1to4.4.4}.
In Case $4. 4. 1$, since $K_{XY}$ contains the edge $(1, 3)$-$(2, 0)$ as its $1$-simplex, the skeleton of $C$ is not the lollipop graph by Lemma \ref{Case2.6}.
In the other cases, since $b_1(\overline{C'})$ is at most $1$, the skeleton of $C$ is not the lollipop graph by Lemma \ref{saikurunashi} and \ref{case4hodai}.

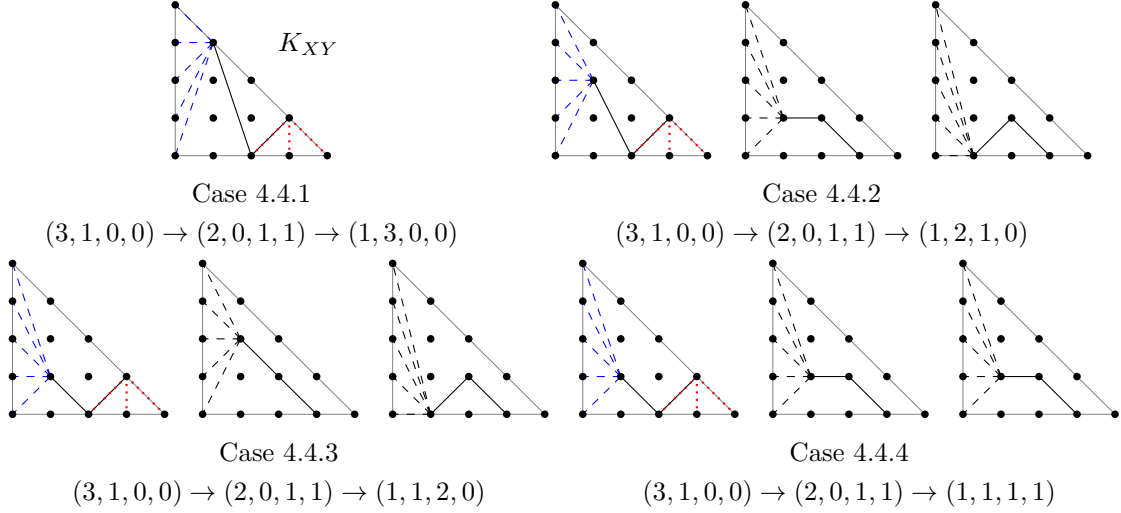
\begin{figure}[H]
\centering
\begin{tikzpicture}
 \coordinate (B1) at (2.5,0);
 \coordinate (B2) at (3,0);
 \coordinate (B3) at (3.5,0);
 \coordinate (B4) at (4,0);
 \coordinate (B5) at (4.5,0);
 \coordinate (B6) at (2.5,0.5);
 \coordinate (B7) at (3,0.5);
 \coordinate (B8) at (3.5,0.5);
 \coordinate (B9) at (4,0.5);
 \coordinate (B10) at (2.5,1);
 \coordinate (B11) at (3,1);
 \coordinate (B12) at (3.5,1);
 \coordinate (B13) at (2.5,1.5);
 \coordinate (B14) at (3,1.5);
 \coordinate (B15) at (2.5,2);
 \draw [very thin, gray] (B1)--(B5)--(B15)--cycle;
 \draw (B9)--(B3)--(B14);
 \foreach \P in {1,6,10,13,15}  \draw[dashed, blue] (B14)--(B\P);
 \foreach \P in {3,4,5}  \draw[thick, dotted, red] (B9)--(B\P);
 \foreach \t in {1,2,...,15} \fill[black] (B\t) circle (0.05);
 \coordinate (A1) at (7.5,0);
 \coordinate (A2) at (8,0);
 \coordinate (A3) at (8.5,0);
 \coordinate (A4) at (9,0);
 \coordinate (A5) at (9.5,0);
 \coordinate (A6) at (7.5,0.5);
 \coordinate (A7) at (8,0.5);
 \coordinate (A8) at (8.5,0.5);
 \coordinate (A9) at (9,0.5);
 \coordinate (A10) at (7.5,1);
 \coordinate (A11) at (8,1);
 \coordinate (A12) at (8.5,1);
 \coordinate (A13) at (7.5,1.5);
 \coordinate (A14) at (8,1.5);
 \coordinate (A15) at (7.5,2);
 \draw [very thin, gray] (A1)--(A5)--(A15)--cycle;
 \draw (A9)--(A3)--(A11);
 \foreach \P in {1,6,10,13,15}  \draw[dashed, blue] (A11)--(A\P);
 \foreach \P in {3,4,5}  \draw[thick, dotted, red] (A9)--(A\P);
 \foreach \t in {1,2,...,15} \fill[black] (A\t) circle (0.05);
 \coordinate (B1) at (10,0);
 \coordinate (B2) at (10.5,0);
 \coordinate (B3) at (11,0);
 \coordinate (B4) at (11.5,0);
 \coordinate (B5) at (12,0);
 \coordinate (B6) at (10,0.5);
 \coordinate (B7) at (10.5,0.5);
 \coordinate (B8) at (11,0.5);
 \coordinate (B9) at (11.5,0.5);
 \coordinate (B10) at (10,1);
 \coordinate (B11) at (10.5,1);
 \coordinate (B12) at (11,1);
 \coordinate (B13) at (10,1.5);
 \coordinate (B14) at (10.5,1.5);
 \coordinate (B15) at (10,2);
 \draw [very thin, gray] (B1)--(B5)--(B15)--cycle;
 \draw (B4)--(B8)--(B7);
 \foreach \t in {1,2,...,15} \fill[black] (B\t) circle (0.05);
 \foreach \P in {1,6,10,13,15}  \draw[dashed] (B7)--(B\P);
 \coordinate (C1) at (12.5,0);
 \coordinate (C2) at (13,0);
 \coordinate (C3) at (13.5,0);
 \coordinate (C4) at (14,0);
 \coordinate (C5) at (14.5,0);
 \coordinate (C6) at (12.5,0.5);
 \coordinate (C7) at (13,0.5);
 \coordinate (C8) at (13.5,0.5);
 \coordinate (C9) at (14,0.5);
 \coordinate (C10) at (12.5,1);
 \coordinate (C11) at (13,1);
 \coordinate (C12) at (13.5,1);
 \coordinate (C13) at (12.5,1.5);
 \coordinate (C14) at (13,1.5);
 \coordinate (C15) at (12.5,2);
 \draw [very thin, gray] (C1)--(C5)--(C15)--cycle;
 \draw (C4)--(C8)--(C2);
 \foreach \t in {1,2,...,15} \fill[black] (C\t) circle (0.05);
 \foreach \P in {1,6,10,13,15}  \draw[dashed] (C2)--(C\P);
 \coordinate [label=below:$\text{Case $4. 4. 1$}$] (a1) at (3.5,-0.25);
 \coordinate [label=below:$\text{$(3, 1, 0, 0)\rightarrow (2, 0, 1, 1)\rightarrow (1, 3, 0, 0)$}$] (a1) at (3.5,-0.75);
 \coordinate [label=below:$\text{Case $4. 4. 2$}$] (a2) at (11,-0.25);
 \coordinate [label=below:$\text{$(3, 1, 0, 0)\rightarrow (2, 0, 1, 1)\rightarrow (1, 2, 1, 0)$}$] (a2) at (11,-0.75);
 \coordinate[label=below:$K_{XY}$] (a3) at (4.25,1.75);
\end{tikzpicture}
\begin{tikzpicture}
 \coordinate (A1) at (0,0);
 \coordinate (A2) at (0.5,0);
 \coordinate (A3) at (1,0);
 \coordinate (A4) at (1.5,0);
 \coordinate (A5) at (2,0);
 \coordinate (A6) at (0,0.5);
 \coordinate (A7) at (0.5,0.5);
 \coordinate (A8) at (1,0.5);
 \coordinate (A9) at (1.5,0.5);
 \coordinate (A10) at (0,1);
 \coordinate (A11) at (0.5,1);
 \coordinate (A12) at (1,1);
 \coordinate (A13) at (0,1.5);
 \coordinate (A14) at (0.5,1.5);
 \coordinate (A15) at (0,2);
 \draw [very thin, gray] (A1)--(A5)--(A15)--cycle;
 \draw (A9)--(A3)--(A7);
 \foreach \P in {1,6,10,13,15}  \draw[dashed, blue] (A7)--(A\P);
 \foreach \P in {3,4,5}  \draw[thick, dotted, red] (A9)--(A\P);
 \foreach \t in {1,2,...,15} \fill[black] (A\t) circle (0.05);
 \coordinate (B1) at (2.5,0);
 \coordinate (B2) at (3,0);
 \coordinate (B3) at (3.5,0);
 \coordinate (B4) at (4,0);
 \coordinate (B5) at (4.5,0);
 \coordinate (B6) at (2.5,0.5);
 \coordinate (B7) at (3,0.5);
 \coordinate (B8) at (3.5,0.5);
 \coordinate (B9) at (4,0.5);
 \coordinate (B10) at (2.5,1);
 \coordinate (B11) at (3,1);
 \coordinate (B12) at (3.5,1);
 \coordinate (B13) at (2.5,1.5);
 \coordinate (B14) at (3,1.5);
 \coordinate (B15) at (2.5,2);
 \draw [very thin, gray] (B1)--(B5)--(B15)--cycle;
 \draw (B4)--(B8)--(B11);
 \foreach \t in {1,2,...,15} \fill[black] (B\t) circle (0.05);
 \foreach \P in {1,6,10,13,15}  \draw[dashed] (B11)--(B\P);
 \coordinate (C1) at (5,0);
 \coordinate (C2) at (5.5,0);
 \coordinate (C3) at (6,0);
 \coordinate (C4) at (6.5,0);
 \coordinate (C5) at (7,0);
 \coordinate (C6) at (5,0.5);
 \coordinate (C7) at (5.5,0.5);
 \coordinate (C8) at (6,0.5);
 \coordinate (C9) at (6.5,0.5);
 \coordinate (C10) at (5,1);
 \coordinate (C11) at (5.5,1);
 \coordinate (C12) at (6,1);
 \coordinate (C13) at (5,1.5);
 \coordinate (C14) at (5.5,1.5);
 \coordinate (C15) at (5,2);
 \draw [very thin, gray] (C1)--(C5)--(C15)--cycle;
 \draw (C4)--(C8)--(C2);
 \foreach \t in {1,2,...,15} \fill[black] (C\t) circle (0.05);
 \foreach \P in {1,6,10,13,15}  \draw[dashed] (C2)--(C\P);
 \coordinate (A1) at (7.5,0);
 \coordinate (A2) at (8,0);
 \coordinate (A3) at (8.5,0);
 \coordinate (A4) at (9,0);
 \coordinate (A5) at (9.5,0);
 \coordinate (A6) at (7.5,0.5);
 \coordinate (A7) at (8,0.5);
 \coordinate (A8) at (8.5,0.5);
 \coordinate (A9) at (9,0.5);
 \coordinate (A10) at (7.5,1);
 \coordinate (A11) at (8,1);
 \coordinate (A12) at (8.5,1);
 \coordinate (A13) at (7.5,1.5);
 \coordinate (A14) at (8,1.5);
 \coordinate (A15) at (7.5,2);
 \draw [very thin, gray] (A1)--(A5)--(A15)--cycle;
 \draw (A9)--(A3)--(A7);
 \foreach \P in {1,6,10,13,15}  \draw[dashed, blue] (A7)--(A\P);
 \foreach \P in {3,4,5}  \draw[thick, dotted, red] (A9)--(A\P);
 \foreach \t in {1,2,...,15} \fill[black] (A\t) circle (0.05);
 \coordinate (B1) at (10,0);
 \coordinate (B2) at (10.5,0);
 \coordinate (B3) at (11,0);
 \coordinate (B4) at (11.5,0);
 \coordinate (B5) at (12,0);
 \coordinate (B6) at (10,0.5);
 \coordinate (B7) at (10.5,0.5);
 \coordinate (B8) at (11,0.5);
 \coordinate (B9) at (11.5,0.5);
 \coordinate (B10) at (10,1);
 \coordinate (B11) at (10.5,1);
 \coordinate (B12) at (11,1);
 \coordinate (B13) at (10,1.5);
 \coordinate (B14) at (10.5,1.5);
 \coordinate (B15) at (10,2);
 \draw [very thin, gray] (B1)--(B5)--(B15)--cycle;
 \draw (B4)--(B8)--(B7);
 \foreach \t in {1,2,...,15} \fill[black] (B\t) circle (0.05);
 \foreach \P in {1,6,10,13,15}  \draw[dashed] (B7)--(B\P);
 \coordinate (C1) at (12.5,0);
 \coordinate (C2) at (13,0);
 \coordinate (C3) at (13.5,0);
 \coordinate (C4) at (14,0);
 \coordinate (C5) at (14.5,0);
 \coordinate (C6) at (12.5,0.5);
 \coordinate (C7) at (13,0.5);
 \coordinate (C8) at (13.5,0.5);
 \coordinate (C9) at (14,0.5);
 \coordinate (C10) at (12.5,1);
 \coordinate (C11) at (13,1);
 \coordinate (C12) at (13.5,1);
 \coordinate (C13) at (12.5,1.5);
 \coordinate (C14) at (13,1.5);
 \coordinate (C15) at (12.5,2);
 \draw [very thin, gray] (C1)--(C5)--(C15)--cycle;
 \draw (C4)--(C8)--(C7);
 \foreach \t in {1,2,...,15} \fill[black] (C\t) circle (0.05);
 \foreach \P in {1,6,10,13,15}  \draw[dashed] (C7)--(C\P);
 \coordinate [label=below:$\text{Case $4. 4. 3$}$] (a1) at (3.5,-0.25);
 \coordinate [label=below:$\text{$(3, 1, 0, 0)\rightarrow (2, 0, 1, 1)\rightarrow (1, 1, 2, 0)$}$] (a1) at (3.5,-0.75);
 \coordinate [label=below:$\text{Case $4. 4. 4$}$] (a2) at (11,-0.25);
 \coordinate [label=below:$\text{$(3, 1, 0, 0)\rightarrow (2, 0, 1, 1)\rightarrow (1, 1, 1, 1)$}$] (a2) at (11,-0.75);
\end{tikzpicture}
\caption{The complexes in Cases $4. 4. 1$ to $4. 4. 4$.}
\label{fc4.4.1to4.4.4}
\end{figure}

Case $4. 4. 5$: $(3, 1, 0, 0)\rightarrow (2, 0, 1, 1)\rightarrow (1, 0, 3, 0)$

The complexes $K_{XZ}$ and $K_{XW}$ are as in Figure \ref{fc4.4.5} (a).
Assume that the skeleton of $C$ is the lollipop graph of genus $3$.
Then by Lemma \ref{12}, $K_{XZ}$ contains the $2$-simplices as in (b), and the skeletons of $C_{XZ}$ and $C_{XW}$ are as in (c).
\begin{figure}[H]
\centering
\begin{tikzpicture}
 \coordinate (B1) at (3,0);
 \coordinate [label=below:1] (B2) at (3.5,0);
 \coordinate [label=below:2] (B3) at (4,0);
 \coordinate [label=below:3] (B4) at (4.5,0);
 \coordinate [label=below:4] (B5) at (5,0);
 \coordinate [label=left:1] (B6) at (3,0.5);
 \coordinate (B7) at (3.5,0.5);
 \coordinate (B8) at (4,0.5);
 \coordinate (B9) at (4.5,0.5);
 \coordinate [label=left:2] (B10) at (3,1);
 \coordinate (B11) at (3.5,1);
 \coordinate (B12) at (4,1);
 \coordinate [label=left:3] (B13) at (3,1.5);
 \coordinate (B14) at (3.5,1.5);
 \coordinate [label=left:4] (B15) at (3,2);
 \draw [very thin, gray] (B1)--(B5)--(B15)--cycle; 
 \draw (B4)--(B8)--(B14);
 \foreach \t in {1,2,...,15} \fill[black] (B\t) circle (0.05);
 \foreach \P in {1,6,10,13,15}  \draw[dashed] (B14)--(B\P);
 \coordinate (C1) at (5.5,0);
 \coordinate [label=below:1] (C2) at (6,0);
 \coordinate [label=below:2] (C3) at (6.5,0);
 \coordinate [label=below:3] (C4) at (7,0);
 \coordinate [label=below:4] (C5) at (7.5,0);
 \coordinate [label=left:1] (C6) at (5.5,0.5);
 \coordinate (C7) at (6,0.5);
 \coordinate (C8) at (6.5,0.5);
 \coordinate (C9) at (7,0.5);
 \coordinate [label=left:2] (C10) at (5.5,1);
 \coordinate (C11) at (6,1);
 \coordinate (C12) at (6.5,1);
 \coordinate [label=left:3] (C13) at (5.5,1.5);
 \coordinate (C14) at (6,1.5);
 \coordinate [label=left:4] (C15) at (5.5,2);
 \draw [very thin, gray] (C1)--(C5)--(C15)--cycle; 
 \draw (C4)--(C8)--(C2);
 \foreach \t in {1,2,...,15} \fill[black] (C\t) circle (0.05);
 \foreach \P in {1,6,10,13,15}  \draw[dashed] (C2)--(C\P);
 \coordinate (D1) at (8.5,0);
 \coordinate [label=below:1] (D2) at (9,0);
 \coordinate [label=below:2] (D3) at (9.5,0);
 \coordinate [label=below:3] (D4) at (10,0);
 \coordinate [label=below:4] (D5) at (10.5,0);
 \coordinate [label=left:1] (D6) at (8.5,0.5);
 \coordinate (D7) at (9,0.5);
 \coordinate (D8) at (9.5,0.5);
 \coordinate (D9) at (10,0.5);
 \coordinate [label=left:2] (D10) at (8.5,1);
 \coordinate (D11) at (9,1);
 \coordinate (D12) at (9.5,1);
 \coordinate [label=left:3] (D13) at (8.5,1.5);
 \coordinate (D14) at (9,1.5);
 \coordinate [label=left:4] (D15) at (8.5,2);
 \draw [very thin, gray] (D1)--(D5)--(D15)--cycle;
 \foreach \P in {10,13,15}  \draw[dashed] (D14)--(D\P);
 \draw (D8)--(D14);
 \draw (D4)--(D8)--(D10)--cycle;
 \draw (D4)--(D7)--(D10);
 \draw (D10)--(D11)--(D8);
 \draw[very thick] (D8)--(D11);
 \foreach \t in {1,2,...,15} \fill[black] (D\t) circle (0.05);
 \coordinate [label=below:(a)] (a1) at (5.25,-0.5);
 \coordinate [label=below:(b)] (a2) at (9.5,-0.5);
 \coordinate[label=above right:O] (O1) at (13,2);
 \coordinate (X1) at (11,2);
 \coordinate (Y1) at (13,0);
 \coordinate (X2) at (11.5,2);
 \coordinate (X3) at (12,2);
 \coordinate (X4) at (12.5,2);
 \coordinate (Y2) at (13,1);
 \draw[thin,->,>=stealth] (O1)--(X1) node[left] {$X$};
 \draw[thin,->,>=stealth] (O1)--(Y1) node[below] {$Z$};
 \fill[black] (X2) circle (0.06);
 \fill[black] (X3) circle (0.06);
 \fill[black] (X4) circle (0.06);
 \draw[dashed] (12,1) circle (0.6);
 \draw (12,0.75) circle (0.2);
 \draw (12,1.25) circle (0.2);
 \draw (X2)--(11.5,1.34);
 \draw (X3)--(12,1.6);
 \draw (12.5,1.34)--(X4);
 \draw (12,1.45)--(12,1.6);
 \draw (12.51,1.34)--(12.18,1.34);
 \draw[very thick] (12,1.45) arc (90:25:0.2);
 \coordinate[label=below:O] (O3) at (16,0.5);
 \coordinate (X9) at (14,0.5);
 \coordinate (W1) at (17,2);
 \coordinate (X10) at (14.5,0.5);
 \coordinate (X11) at (15,0.5);
 \coordinate (X12) at (15.5,0.5);
 \draw[thin,->,>=stealth] (O3)--(X9) node[left] {$X$};
 \draw[thin,->,>=stealth] (O3)--(W1) node[above right] {$W$};
 \fill[black] (X10) circle (0.06);
 \fill[black] (X11) circle (0.06);
 \fill[black] (X12) circle (0.06);
 \draw (X12) arc (0:180:0.25);
 \coordinate [label=above:$x_1$] (a1) at (12.5,2);
 \coordinate [label=above:$x_2$] (a2) at (12,2); 
 \coordinate [label=below:$x_1$] (a1) at (15.5,0.5);
 \coordinate [label=below:$x_2$] (a2) at (15,0.5);
 \coordinate [label=below:(c)] (a3) at (13.75,-0.5);  
\end{tikzpicture}
\caption{The complexes and skeletons in Case $4. 4. 5$.}
\label{fc4.4.5}
\end{figure}
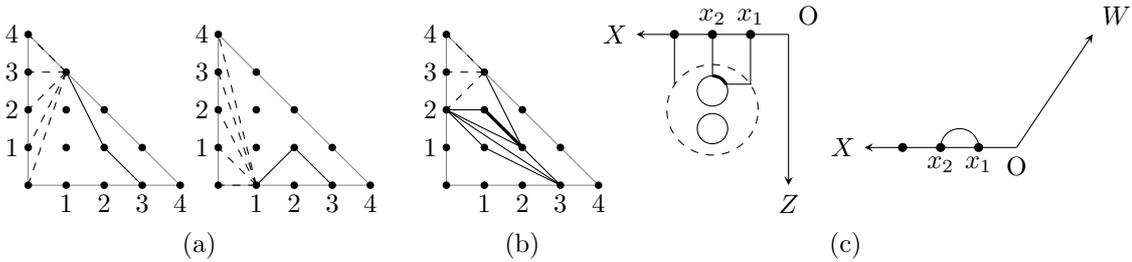

The path from $x_1$ to $x_2$ in $C_{XZ}$ corresponding to the union of simplices of $K_{XY}$ which orbits around the point $(2, 1)$ intersects the cycle corresponding to the point $(1, 2)$.
When we glue at $X$, we have the cycle which contains the points $x_1$ and $x_2$.
Thus the two cycles intersect and the skeleton of $C$ is not the lollipop graph.

\medbreak
Case $4. 4. 6$: $(3, 1, 0, 0)\rightarrow (2, 0, 1, 1)\rightarrow (1, 0, 2, 1)$

The complexes $K_{XY}$, $K_{XZ}$ and $K_{XW}$ are as in Figure \ref{fc4.4.6}.
Thus $b_1(\overline{C'})=1$ and the skeleton of $C$ is not the lollipop graph by Lemma \ref{case4hodai}.
\begin{figure}[H]
\centering
\begin{tikzpicture}
 \coordinate[label=below left:O] (A1) at (0,0);
 \coordinate [label=below:1] (A2) at (0.5,0);
 \coordinate [label=below:2] (A3) at (1,0);
 \coordinate [label=below:3] (A4) at (1.5,0);
 \coordinate [label=below:4] (A5) at (2,0);
 \coordinate [label=left:1] (A6) at (0,0.5);
 \coordinate (A7) at (0.5,0.5);
 \coordinate (A8) at (1,0.5);
 \coordinate (A9) at (1.5,0.5);
 \coordinate [label=left:2] (A10) at (0,1);
 \coordinate (A11) at (0.5,1);
 \coordinate (A12) at (1,1);
 \coordinate [label=left:3] (A13) at (0,1.5);
 \coordinate (A14) at (0.5,1.5);
 \coordinate [label=left:4] (A15) at (0,2);
 \draw [very thin, gray] (A1)--(A5)--(A15)--cycle; 
 \draw (A9)--(A3)--(A2);
 \foreach \P in {1,6,10,13,15}  \draw[dashed, blue] (A2)--(A\P);
 \foreach \P in {3,4,5}  \draw[thick, dotted, red] (A9)--(A\P);
 \foreach \t in {1,2,...,15} \fill[black] (A\t) circle (0.05);
 \coordinate[label=below left:O] (B1) at (3,0);
 \coordinate [label=below:1] (B2) at (3.5,0);
 \coordinate [label=below:2] (B3) at (4,0);
 \coordinate [label=below:3] (B4) at (4.5,0);
 \coordinate [label=below:4] (B5) at (5,0);
 \coordinate [label=left:1] (B6) at (3,0.5);
 \coordinate (B7) at (3.5,0.5);
 \coordinate (B8) at (4,0.5);
 \coordinate (B9) at (4.5,0.5);
 \coordinate [label=left:2] (B10) at (3,1);
 \coordinate (B11) at (3.5,1);
 \coordinate (B12) at (4,1);
 \coordinate [label=left:3] (B13) at (3,1.5);
 \coordinate (B14) at (3.5,1.5);
 \coordinate [label=left:4] (B15) at (3,2);
 \draw [very thin, gray] (B1)--(B5)--(B15)--cycle; 
 \draw (B4)--(B8)--(B11);
 \foreach \t in {1,2,...,15} \fill[black] (B\t) circle (0.05);
 \foreach \P in {1,6,10,13,15}  \draw[dashed] (B11)--(B\P);
 \coordinate[label=below left:O] (C1) at (6,0);
 \coordinate [label=below:1] (C2) at (6.5,0);
 \coordinate [label=below:2] (C3) at (7,0);
 \coordinate [label=below:3] (C4) at (7.5,0);
 \coordinate [label=below:4] (C5) at (8,0);
 \coordinate [label=left:1] (C6) at (6,0.5);
 \coordinate (C7) at (6.5,0.5);
 \coordinate (C8) at (7,0.5);
 \coordinate (C9) at (7.5,0.5);
 \coordinate [label=left:2] (C10) at (6,1);
 \coordinate (C11) at (6.5,1);
 \coordinate (C12) at (7,1);
 \coordinate [label=left:3] (C13) at (6,1.5);
 \coordinate (C14) at (6.5,1.5);
 \coordinate [label=left:4] (C15) at (6,2);
 \draw [very thin, gray] (C1)--(C5)--(C15)--cycle; 
 \draw (C4)--(C8)--(C7);
 \foreach \t in {1,2,...,15} \fill[black] (C\t) circle (0.05);
 \foreach \P in {1,6,10,13,15}  \draw[dashed] (C7)--(C\P);
\end{tikzpicture}
\caption{The complexes $K_{XY}$, $K_{XZ}$ and $K_{XW}$ in Case $4. 4. 6$.}
\label{fc4.4.6}
\end{figure}
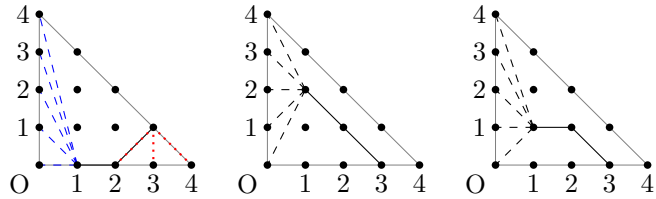

From the above, in Case $4. 4$, the skeleton of $C$ is not the lollipop graph.
Therefore in Case $4$, the skeleton of $C$ is not the lollipop graph.

Thus we conclude the proof of Theorem \ref{mainth}.

\end{document}